\documentclass[11p]{article}

\usepackage[a4paper, total={6in, 9in}]{geometry}

\usepackage{fixmath}
\usepackage{mathtools}
\usepackage{microtype}
\usepackage{wrapfig}
\usepackage{amsmath}
\usepackage{amssymb}
\usepackage{bbm}
\usepackage{tikz}
\usetikzlibrary{fadings}
\usetikzlibrary{patterns}
\usetikzlibrary{shadows.blur}
\usetikzlibrary{shapes}
\usepackage{hyperref}
\usepackage{amsthm}
\newtheorem{theorem}{Theorem}
\newtheorem{lemma}{Lemma}
\newtheorem{proposition}{Proposition}
\newtheorem{corollary}{Corollary}

\newtheorem{example}{Example}
\theoremstyle{definition}
\newtheorem{remark}{Remark}
\newtheorem{definition}{Definition}
\allowdisplaybreaks
\usepackage{algpseudocode}
\usepackage{float}
\usepackage[ruled,vlined,linesnumbered]{algorithm2e}
\usepackage{url}

\makeatletter
\def\bbordermatrix#1{\begingroup \m@th
  \@tempdima 4.75\p@
  \setbox\z@\vbox{%
    \def\cr{\crcr\noalign{\kern2\p@\global\let\cr\endline}}%
    \ialign{$##$\hfil\kern2\p@\kern\@tempdima&\thinspace\hfil$##$\hfil
      &&\quad\hfil$##$\hfil\crcr
      \omit\strut\hfil\crcr\noalign{\kern-\baselineskip}%
      #1\crcr\omit\strut\cr}}%
  \setbox\tw@\vbox{\unvcopy\z@\global\setbox\@ne\lastbox}%
  \setbox\tw@\hbox{\unhbox\@ne\unskip\global\setbox\@ne\lastbox}%
  \setbox\tw@\hbox{$\kern\wd\@ne\kern-\@tempdima\left[\kern-\wd\@ne
    \global\setbox\@ne\vbox{\box\@ne\kern2\p@}%
    \vcenter{\kern-\ht\@ne\unvbox\z@\kern-\baselineskip}\,\right]$}%
  \null\;\vbox{\kern\ht\@ne\box\tw@}\endgroup}
\makeatother

\usepackage{booktabs}

\usepackage[colorinlistoftodos]{todonotes}
\usepackage{caption}
\usepackage{subcaption}
\usepackage{lscape}
\usepackage[affil-it]{authblk}
\usepackage{comment} 

\newcommand{\tr}{\textup{tr}}
\newcommand{\cl}{\textup{cl}_{CG}}
\newcommand{\Diag}{\textup{Diag}}
\newcommand{\diag}{\textup{diag}}
\newcommand{\triu}{\textup{triu}}

\newcommand{\Coll}{\textup{Col}}
\newcommand{\Nul}{\textup{Nul}}
\newcommand{\Conv}{\textup{Conv}}

\newcommand{\Aff}{\textup{Aff}}

\newcommand{\Pz}{\textup{P}}
\newcommand{\cone}{\textup{cone}}
\newcommand{\svec}{\textup{svec}}
\newcommand{\xb}{\bold{x}}
\newcommand{\Xb}{\bold{X}}
\newcommand{\Cb}{\bold{C}}
\newcommand{\Ab}{\bold{A_i}}
\newcommand{\zb}{\bold{0}}
\newcommand{\yb}{\bold{y}}

\makeatletter
\newcommand{\setword}[2]{%
  \phantomsection
  #1\def\@currentlabel{\unexpanded{#1}}\label{#2}%
}
\allowdisplaybreaks
\makeatother

\begin{document}

\title{The Chv\'atal-Gomory Procedure for Integer SDPs with Applications in Combinatorial Optimization}

\author[ ]{Frank de Meijer \thanks{Delft Institute of Applied Mathematics, Delft University of Technology, The Netherlands, {\tt f.j.j.demeijer@tudelft.nl}} \qquad  Renata Sotirov \thanks{CentER, Department of Econometrics and OR, Tilburg University, The Netherlands, {\tt r.sotirov@uvt.nl}}}

\date{}
\maketitle

\begin{abstract}
In this paper we study the well-known Chv\'atal-Gomory (CG) procedure for the class of integer semidefinite programs (ISDPs).
We prove several results regarding the hierarchy of relaxations obtained by iterating this procedure.
We also study different formulations of the elementary closure of spectrahedra. A polyhedral description of the elementary closure for a specific type of spectrahedra is derived by exploiting total dual integrality for SDPs.
Moreover, we show how to exploit (strengthened) CG cuts in a branch-and-cut framework for ISDPs.
Different from existing algorithms in the literature, the separation routine in our approach
exploits both the semidefinite and the integrality constraints. We provide separation routines
for several common classes of binary SDPs resulting from combinatorial optimization problems.
In the second part of the paper we present a comprehensive application of our approach to the quadratic traveling salesman
problem (\textsc{QTSP}). 
Based on the algebraic connectivity of the directed Hamiltonian cycle, two ISDPs that model the \textsc{QTSP} are introduced.
We show that the CG cuts resulting from these formulations contain several
well-known families of cutting planes. Numerical results illustrate the practical strength of the
CG cuts in our branch-and-cut algorithm, which outperforms alternative
ISDP solvers and is able to solve large \textsc{QTSP} instances to optimality.

\end{abstract}
\hfill \break

\section{Introduction}
Convex integer nonlinear programs (CINLPs) are optimization problems in which the objective function is convex and the continuous relaxation of the feasible region is a convex set.
Nonlinearities in CINLPs can appear in both the objective function and/or the constraints.
Motivated by their numerous applications and their ability to generalize several well-known problem classes, CINLPs have been studied for decades. In this paper we focus on a specific class of CINLPs: the integer semidefinite programs (ISDPs). These problems can be formulated as:
\begin{align} \label{eq:ISDP}
\sup ~   \bold{b}^\top \bold{x} 
\quad \text{s.t.} \quad  \bold{C} - \sum_{i = 1}^m \bold{A}_ix_i \succeq \bold{0}, \quad  \bold{x} \in \mathbb{Z}^m,  
\end{align}
with $\bold{b} \in \mathbb{R}^m$, $\bold{C}, \bold{A}_i \in \mathcal{S}^n$, where $\mathcal{S}^n$ denotes the cone of symmetric matrices of order $n$. Note that   $\bold{C} - \sum_{i = 1}^m \bold{A}_ix_i \succeq \bold{0}$ is referred to as a linear matrix inequality (LMI) and it is the SDP analogue of a system of linear inequalities defining a polyhedron.
Since integer linear programs belong to the family of ISDPs, problems of the form \eqref{eq:ISDP} are generally $\mathcal{NP}$-hard to solve.

Although CINLPs have been studied extensively, see e.g., the survey of Bonami et al.\ \cite{BonamiEtAl}, the special case of ISDPs has received attention only very recently.
This is remarkable, as the mixture of positive semidefiniteness and integrality leads naturally to a broad range of applications, e.g., in architecture \cite{CerveiraEtAl, YonekuraKanno}, signal processing \cite{GallyPfetsch, PhilippEtAl} and combinatorial optimization \cite{GallyEtAl, RendlEtAl}. For a more detailed overview of applications of ISDPs, we refer the reader to \cite{GallyEtAl, KobayashiTakano}.

Only a few solution approaches for solving SDPs with integrality constraints have been considered. Gally et al.\ \cite{GallyEtAl} propose a general framework called SCIP-SDP for solving mixed integer semidefinite programs (MISDPs) using a branch-and-bound (B\&B) procedure with continuous SDPs as subproblems. They show that strict duality of the relaxations is maintained in the B\&B tree and study several solver components. Alternatively, Kobayashi and Takano \cite{KobayashiTakano} propose a cutting-plane algorithm that initially relaxes the positive semidefinite (PSD) constraint and solves a mixed integer linear programming problem, where the PSD constraint is imposed dynamically via cutting planes. This leads to a general branch-and-cut (B\&C) algorithm for solving MISDPs. 
A third project that encounters general ISDPs is YALMIP \cite{Lofberg}. However, it is noted by the authors of \cite{GallyEtAl} and \cite{KobayashiTakano} that the branch-and-bound ISDP solver in YALMIP is not yet competitive to the performance of the other two methods. Recently, Matter and Pfetsch \cite{MatterPfetsch} study different presolving strategies for MISDPs for both the B\&B and B\&C approach.

Apart from solution methods for solving general ISDPs or MISDPs, there are
several other approaches in the literature that aim to solve integer problems by utilizing SDP relaxations in a B\&B framework. Although these approaches are very related to problems of the form \eqref{eq:ISDP} in the sense that they also combine semidefinite programs with a branching strategy, they differ in the sense that the problem at hand is not necessarily formulated as a MISDP. Examples are  the BiqCrunch solver for constrained binary quadratic
problems~\cite{KrislockEtAl} and  the Biq Mac solver for unconstrained binary quadratic problems~\cite{RendlEtAl}.

In the light of improving the performance of the B\&C algorithm of \cite{KobayashiTakano}, we consider the exploitation of cutting planes for ISDPs. Practical algorithms for CINLPs have benefited a lot from the addition of strong cutting planes, see e.g., \cite{AtamturkNarayanan2007, AtamturkNarayanan2010, BELOTTI20173, StubbsMehrotra}, where many of these cutting plane frameworks are based on generalizations from integer linear programming. Among the most well-known cutting planes for integer linear programs (ILPs) are the
Chv\'atal-Gomory (CG) cuts \cite{Chvatal, Gomory}. Gomory~\cite{Gomory} introduced these cuts to design the first finite cutting plane algorithm for ILPs. Chv\'atal \cite{Chvatal} later generalized this notion and introduced the closure of all such cuts that leads to a hierarchy of relaxations of the ILP with increasing strength. Chv\'atal \cite{Chvatal} and Schrijver \cite{Schrijver} prove that this hierarchy is finite for bounded real polyhedra and rational polyhedra, respectively. Later on, the CG procedure is introduced for more general convex sets, see e.g.,~\cite{DeyVielma, DadushEtAl2011, DunkelSchulz, BraunPokutta, DadushEtAl}. In particular, \c{C}ezik and Iyengar \cite{CezikIyengar} show how to generate  CG cuts for CINLPs where the continuous relaxation of the feasible region is conic representable.

A leading application in this work is a combinatorial optimization problem that can be modelled as an ISDP: the quadratic traveling salesman problem (\textsc{QTSP}). J\"ager and Molitor \cite{Jager} introduce the QTSP as the problem of finding a Hamiltonian cycle in a graph that minimizes the total interaction costs among consecutive arcs. The problem is motivated by an important application in bioinformatics \cite{Jager, AandFFischer}, but has also applications in telecommunication, precision farming and robotics, see e.g.,~\cite{WirthSteffan, FeketeKrupke, Aggarwal}. The \textsc{QTSP} is $\mathcal{NP}$-hard in the strong sense and is currently considered as one of the hardest combinatorial optimization problems to solve in practice.

Several papers have studied the \textsc{QTSP}. In \cite{Fischer, Fischer2014, FischerHelmberg} the polyhedral structure of the asymmetric and symmetric \textsc{QTSP}--polytope is discussed.  Rostami et al.\ \cite{RostamiEtAl} provide several lower bounding procedures for the \textsc{QTSP}, including a column generation approach. Woods and Punnen \cite{WoodsPunnen} provide different classes of neighbourhoods for the \textsc{QTSP}, while Stan{\v e}k et al.\ \cite{Stanek} discuss several heuristics for the quadratic traveling salesman problem in the plane. The linearization problem for the \textsc{QTSP} is studied in \cite{PunnenEtAl}. Fischer et al.\ \cite{AandFFischer, AandFFischer_RNA} introduce several exact algorithms and heuristics for the asymmetric \textsc{QTSP}, while Aichholzer et al.\ \cite{Aichholzer} consider exact solution methods for the minimization and maximization version of the symmetric \textsc{QTSP}.

\subsection{Main results and outline}
In this paper we consider the Chv\'atal-Gomory procedure for ISDPs from a theoretical as well as a practical point of view. On the theoretical side, we derive several results on the elementary closure of all CG cuts for spectrahedra. On the practical side, we show how to apply these cuts in a generic branch-and-cut algorithm for ISDPs that exploits both the positive semidefiniteness and the integrality of the problem. We extensively study the application of this new approach to the \textsc{QTSP}, which confirms the practical strength of the proposed method.

We start by  reformulating a CG cut for a spectrahedron in terms of its data matrices in combination with the elements from the dual cone.
This leads to a constructive description of the elementary closure of spectrahedra  rather than the implicit description that is known for general convex sets.  Equivalent to the case of polyhedra, the elementary closure operation can be repeated, leading to a hierarchy of stronger approximations of the integer hull of the spectrahedron. For the case of bounded spectrahedra, we provide a compact proof of a homogeneity property for the elementary closure operation that is based on a theorem of alternatives and Dirichlet's approximation theorem. We prove this property for halfspaces that are sufficient to describe any compact convex set. Homogeneity is the cornerstone in showing that the elementary closure of a bounded spectrahedron is polyhedral.
 Although the latter result is known in the literature, our proof significantly simplifies compared to the general proofs given in \cite{DadushEtAl, BraunPokutta}.
 Finally, we exploit the recently introduced notion of total dual integrality for SDPs~\cite{DeCarliSilvaTuncel} to derive {a closed-form expression for the elementary closure of spectrahedra defined by a totally dual integral linear matrix inequality. We additionally provide a characterization of bounded spectrahedra with this property and several more general sufficient conditions.}

It is known that the practical strength of CG cuts in integer linear programming is mainly due to their application in branch-and-bound methods.
In this vein, we propose a generic branch-and-cut (B\&C) framework for ISDPs.
Our algorithm initially relaxes the PSD constraint and solves a mixed integer linear program (MILP), where the PSD constraint is imposed iteratively via CG and/or strengthened CG cuts.
To derive strengthened  CG cuts, we use a similar approach to the one for rational polyhedra by Dash et al.\ \cite{DashEtAl}.
Our B\&C algorithm is an extension of the algorithm of~\cite{KobayashiTakano}, in which separation is only based on positive semidefiniteness without taking into account the integrality of the variables.
Our approach also builds up on the work by \c{C}ezik and Iyengar \cite{CezikIyengar}, in which the authors leave the separation of CG cuts for conic problems as an open problem and do not include these cuts in their computational study.
{We provide an example of our approach for a common class of binary SDPs that frequently appears in combinatorial optimization.}

In the third part of this paper we apply our results to a difficult-to-solve combinatorial optimization problem: the quadratic traveling salesman problem. We derive two ISDP formulations of this problem based on the notion of algebaic connectivity. To solve these models using our B\&C algorithm, we propose several CG separation routines and show that various of these routines lead to well-known cuts for the \textsc{QTSP}. 
Computational results on a large set of benchmark \textsc{QTSP} instances show that the practical potential of our new method is twofold. The method significantly outperforms the ISDP solvers from the literature, whereas it also provides competitive results to the state-of-the-art \textsc{QTSP} solution method of \cite{AandFFischer}.

The paper is organized as follows. In Section~\ref{Section:CGprocedure} we study the Chv\'atal-Gomory procedure for spectrahedra. Section~\ref{Section:B&C} provides a CG-based B\&C framework for general ISDPs and provides specific CG separation routines for two classes of binary SDPs. In Section~\ref{Section:QTSP} we formally define the \textsc{QTSP} and present two ISDP formulations of this problem. Numerical results are given in Section~\ref{Section:ComputationResults}.

\subsection{Notation} \label{sect:notation}
A directed graph is given by $G = (N,A)$, where $N$ is a set of nodes and $A \subseteq N \times N$ is a set of arcs.
We use $K_n$ to denote the complete directed graph on $n$ nodes, i.e., a directed graph in which every pair of nodes is connected by a bidirectional edge.

We denote by $\mathbf{0}_n \in \mathbb{R}^n$  the vector of all zeros, and by $\mathbf{1}_n \in \mathbb{R}^n$ the vector of all  ones. The identity matrix and the matrix of ones of order $n$ are denoted by $\bold{I_n}$ and $\bold{J_n}$, respectively. We omit the subscripts of these matrices when there is no confusion about the order.
The $i$-th elementary vector is denoted by $\mathbf{e}_i$ and we define $\bold{E_{ij}} := \bold{e}_i\bold{e}_j^\top$. {For any two matrices $\bold{A}$ and $\bold{B}$, the direct sum is defined as $\bold{A} \oplus \bold{B} = \big[ \begin{smallmatrix}
    \bold{A} & \bold{0} \\ \bold{0} & \bold{B}
\end{smallmatrix} \big]$.}

The set of integer numbers and non-negative integer numbers is denoted by $\mathbb{Z}$ and $\mathbb{Z}_+$, respectively. For any integer vector $c \in \mathbb{Z}^m$, we let $\gcd(c)$ denote the greatest common divisor of the entries in $c$. We define the floor (resp.\ ceil) operator $\lfloor \cdot \rfloor$ (resp.\ $\lceil \cdot \rceil$) as the largest (resp.\ smallest) integer smaller (resp.\ larger) than or equal to the input number. For $n \in \mathbb{Z}_+$, we define the set $[n] := \{1, \ldots, n\}$.
Also, for any $S \subseteq [n]$, we let $\bold{\mathbbm{1}_{S}}$ be the binary indicator vector of $S$.

We let $\mathcal{S}^n$ be the set of all $n \times n$ real symmetric matrices and
denote by $\bold{X} \succeq \mathbf{0}$ that a symmetric matrix $\bold{X}$ is positive semidefinite. We use $\bold{X} \succneqq \mathbf{0}$ to denote that $\bold{X}$ is positive semidefinite, but not equal to the zero matrix.
The cone of symmetric positive semidefinite matrices is defined as $\mathcal{S}^n_+ := \{ \bold{X} \in \mathcal{S}^n \, : \, \, \bold{X} \succeq \mathbf{0} \}$. The trace of a square matrix $\bold{X}=(x_{ij})$ is given by $\tr(\bold{X})=\sum_{i}x_{ii}$.
 For any $\bold{X},\bold{Y} \in \mathbb{R}^{n \times n}$ the trace inner product is defined as $\langle \bold{X}, \bold{Y} \rangle := \tr(\bold{X}^\top \bold{Y}) = \sum_{i = 1}^n\sum_{j = 1}^n x_{ij}y_{ij}$.

The operator $\diag : \mathbb{R}^{n \times n} \rightarrow \mathbb{R}^n$ maps a square matrix to a vector consisting of its diagonal elements. We denote by $\Diag : \mathbb{R}^n \rightarrow \mathbb{R}^{n \times n}$ its adjoint operator.

\section{The Chv\'atal-Gomory procedure for ISDPs} \label{Section:CGprocedure}

In this section we study the extension of the cutting-plane procedure by Chv\'atal~\cite{Chvatal} and Gomory~\cite{Gomory} for integer linear programs to the class of integer semidefinite programs. We show that several concepts, such as the Chv\'atal-Gomory closure and the Chv\'atal rank, can be generalized to ISDPs. We start by recollecting the procedure for general convex sets.

\subsection{The Chv\'atal-Gomory procedure} \label{Subsection:CGintro}

Let $C \subseteq \mathbb{R}^m$ be a non-empty closed convex set and let $C_I$ be its integer hull, i.e., $C_I := \Conv(C \cap \mathbb{Z}^m )$.
The Chv\'atal-Gomory cutting-plane procedure is introduced by Chv\' atal \cite{Chvatal} and Gomory \cite{Gomory} and is regarded to be among the most celebrated results in integer programming. The CG procedure aims at systematically identifying valid inequalities for $C$ that cut off non-integer solutions. By adding these new cuts to the relaxation and repeating this process, one obtains a hierarchy of stronger relaxations that converges to $C_I$.

The CG procedure relies on the notion of rational halfspaces. A rational halfspace is of the form $H = \{ \bold{x} \in \mathbb{R}^m \, : \, \, \bold{c}^\top \bold{x} \leq d \}$ for some $\bold{c} \in \mathbb{Q}^m, d \in \mathbb{Q}$. It is known that all such halfspaces can be represented by $\bold{c} \in \mathbb{Z}^m$ such that the entries of $\bold{c}$ are relatively prime. If $H = \{ \bold{x} \in \mathbb{R}^m \, : \, \, \bold{c}^\top \bold{x} \leq d \}$ with $\bold{c} \in \mathbb{Z}^m$, $\gcd(\bold{c}) = 1$, then $H_I = \{ \bold{x} \in \mathbb{R}^m \, : \, \, \bold{c}^\top \bold{x} \leq \lfloor d \rfloor \}$.

\begin{definition} \label{Def:elementaryclosure}
The elementary closure of a closed convex set $C$ is the set
\begin{align}
\cl (C) := \bigcap_{\substack{(\bold{c}, d) \in \mathbb{Q}^m \times \mathbb{Q} \\
C \subseteq H = \{\bold{x} \, : \, \, \bold{c}^\top \bold{x} \leq d \}}} H_I.
\end{align}
\end{definition}

Equivalently, the elementary closure of $C$ can be written as:
\begin{align} \label{eq:elementaryclosure}
\cl (C) = \bigcap_{\substack{(\bold{c}, d) \in \mathbb{Z}^m \times \mathbb{R} \\
C \subseteq \{\bold{x} \, : \, \, \bold{c}^\top \bold{x} \leq d \}}} \left\{ \bold{x} \in \mathbb{R}^m \, : \, \, \bold{c}^\top \bold{x} \leq \lfloor d \rfloor \right\},
\end{align}
and we will primarily use this form in this work. The inequalities that define $\cl(C)$ in \eqref{eq:elementaryclosure} are known as CG cuts \cite{Gomory}. One can verify that $C_I \subseteq \cl(C)$. When $C$ is compact, we can exploit the following proposition due to Dadush et al.\ \cite{DadushEtAl2011} and De Carli Silva and Tun\c{c}el \cite{DeCarliSilvaTuncel}.
\begin{proposition} \label{Prop:compactouter}
If $C \subseteq \mathbb{R}^m$ is a compact convex set, then
\begin{align*}
C = \bigcap_{\substack{(\bold{c}, d) \in \mathbb{Z}^m \times \mathbb{R} \\
C \subseteq \{\bold{x} \, : \, \, \bold{c}^\top \bold{x} \leq d \}}} \left\{\bold{x} \in \mathbb{R}^m \, : \, \, \bold{c}^\top \bold{x} \leq d\right\}.
\end{align*}
\end{proposition}
It follows from Proposition~\ref{Prop:compactouter} that for compact convex sets $C$ we have $\cl(C) \subseteq C$. We can now repeat the procedure by defining $C^{(0)} := C$ and $C^{(k+1)} := \cl(C^{(k)})$ for all integer $k \geq 0$, where $C^{(k)}$ is referred to as the $k$th CG closure of $C$. For any compact convex set $C$ this leads to the hierarchy $C_I \subseteq \ldots \subseteq C^{(k+1)} \subseteq C^{(k)} \subseteq \ldots \subseteq C^{(0)} = C$. The smallest $k$ for which $C_I = C^{(k)}$ is known as the \emph{Chv\'atal rank} of $C$. In the same vein, the Chv\'atal rank of an inequality $\bold{c}^\top \bold{x} \leq d$ valid for $C_I$ is defined as the smallest $k$ such that $C^{(k)} \subseteq \{ \bold{x} \in \mathbb{R}^m \, : \,\, \bold{c}^\top \bold{x} \leq d\}$.

\begin{remark}  
Observe that for an unbounded closed convex set $C$,  $\cl(C) \subseteq C$ does not have to hold. 
For instance, the irrational halfspace $\{ \bold{x} \in \mathbb{R}^2 \, : \, \, x_1 + \sqrt{2} x_2 \leq 0\}$ is not contained in any halfspace of the form $\{\bold{x} \in \mathbb{R}^2 \, : \, \, \bold{c}^\top \bold{x} \leq d\}$ with $\bold{c} \in \mathbb{Z}^2$. Therefore, $\cl(C)$ is the intersection over an empty set of halfspaces, resulting in $\cl(C) = \mathbb{R}^2$.
\end{remark}

The finiteness of the Chv\'atal rank is proven in the literature for bounded real polyhedra~\cite{Chvatal}, unbounded rational polyhedra~\cite{Schrijver} and conic representable sets in the 0/1-cube~\cite{CezikIyengar}. However, the Chv\'atal rank for unbounded real polyhedra can be infinite as shown by Schrijver \cite{Schrijver}. Schrijver also shows that the elementary closure of a rational polyhedron is a rational polyhedron. This result is later generalized to irrational polytopes \cite{DunkelSchulz}, bounded rational ellipsoids \cite{DeyVielma}, strictly convex bodies~\cite{DadushEtAl2011} and general compact convex sets \cite{DadushEtAl,BraunPokutta}. As a consequence, the Chv\'atal rank of these sets is also known to be finite.

\subsection{The elementary closure of spectrahedra}
We now apply the notions from Section~\ref{Subsection:CGintro} to integer semidefinite programming problems in standard primal and dual {forms}. On top of the general definition given in the previous section, we derive alternative formulations of the elementary closure of spectrahedra.

Let $\bold{b} \in \mathbb{R}^m$, $\bold{C} \in \mathcal{S}^n$ and $\bold{A_i} \in \mathcal{S}^n$ for all $i \in [m]$. An ISDP in standard primal form is given by:
\begin{align}
(P_{ISDP}) & \left \{ \begin{aligned} \inf \quad &  \langle \bold{C}, \bold{X} \rangle \\
\text{s.t.} ~~ & \langle \bold{A_i} , \bold{X} \rangle = b_i 
\quad \forall i \in [m],
~~\bold{X} \succeq \bold{0}, \,\, \bold{X} \in \mathbb{Z}^{n \times n},  \end{aligned} \right. \label{primalISDP}
\intertext{while an ISDP in standard dual form is given by:}
(D_{ISDP}) & \left \{ \begin{aligned} \sup \quad &  \bold{b}^\top \bold{x} \\
\text{s.t.} \quad & \bold{C} - \sum_{i = 1}^m \bold{A_i}x_i \succeq \mathbf{0}, 
~~\bold{x} \in \mathbb{Z}^m. \end{aligned} \right. \label{dualISDP}
\end{align}
Using standard techniques, one can syntactically rewrite an integer SDP from primal form to dual form and vice versa. Consistent with most of the literature, we mainly consider, but not restrict ourselves to, ISDPs in dual form.

The continuous relaxation of the feasible set of \eqref{dualISDP} is defined as follows:
\begin{align} \label{def:P}
P & := \left\{ \bold{x} \in \mathbb{R}^m \, \, : \, \bold{C} - \sum_{i = 1}^m \bold{A_i} x_i \succeq \mathbf{0} \right\}.
\end{align}
The set $P$ is a spectrahedron that is closed, semialgebraic and convex, which we assume to be non-empty. 
{Throughout the paper, we make the following non-restrictive assumption on the linear matrix inequality defining $P$. In case $P$ is not full-dimensional, i.e., the subspace $\mathcal{L} := \Aff(P)^\perp$ is nontrivial, we extend $\bold{C}$ and $\bold{A_i}$, $i \in [m]$, to
\begin{align*}
\bold{C} \oplus \Diag(\bold{L}\bold{x_0}) \oplus -\Diag (\bold{Lx_0}) \quad \text{and}\quad \bold{A_i} \oplus \Diag(\bold{\ell_i}) \oplus -\Diag(\bold{\ell_i})~\text{for all } i \in [m]
\end{align*}
where $\bold{L} := [\bold{\ell_1}~\dots~\bold{\ell_m}] \in \mathbb{R}^{\dim(\mathcal{L}) \times m}$ is a matrix whose rows form a basis for $\mathcal{L}$ and $\bold{x_0} \in P$. Observe that the resulting extended map has no effect on the spectrahedron $P$ itself. We only include it to obtain a more proper algebraic representation, see also~\cite{RamanaGoldman, DeMeijerThesis}.}
We define the integer hull of $P$ to be $P_I := \Conv(P \cap \mathbb{Z}^m)$, i.e., the convex hull of the integral points in $P$. {We briefly consider some illustrative examples of spectrahedra and their integer hulls.}

\begin{example}[Examples in $\mathbb{R}^2$] \label{Example:spectrahedra} 
Let $\bold{C} = {\begin{bsmallmatrix} 
    0 & 3 \\ 3 & 3
\end{bsmallmatrix}}, \bold{A_1} = \begin{bsmallmatrix}
    -3 & 1 \\ 1 & 1
\end{bsmallmatrix}$ and $\bold{A_2} = \begin{bsmallmatrix}
    \frac{1}{2} & 1 \\ 1 & 0
\end{bsmallmatrix}$. Then, the induced spectrahedron $P$ in the dual form~\eqref{def:P} is the semialgebraic set of points in $\mathbb{R}^2$ described by the quadratic inequality $4x_1^2 + x_2^2 \leq 15x_1 + 4\frac{1}{2}x_2 - 1\frac{1}{2}x_1x_2 - 9$. This spectrahedron is bounded and given in Figure~\ref{Fig:example_bounded}.

Let $Q$ be described by~\eqref{def:P} with $\bold{C} = \begin{bsmallmatrix}
    1 & 0 \\ 0 & 0
\end{bsmallmatrix}, \bold{A_1} = \begin{bsmallmatrix}
    0 & -1 \\ -1 & 0
\end{bsmallmatrix}$ and $\bold{A_2} = \begin{bsmallmatrix}
    0 & 0 \\ 0 & -2
\end{bsmallmatrix}$. The spectrahedron $Q$ is the unbounded semialgebraic set $\{\bold{x} \in \mathbb{R}^2 \, : \, \, x_2 \geq \frac{1}{2}x_1^2\}$, see Figure~\ref{Fig:example_unbounded}. 

\begin{figure}[H]
    \centering
    \begin{subfigure}[b]{0.48\textwidth}
        \includegraphics[scale=0.3]{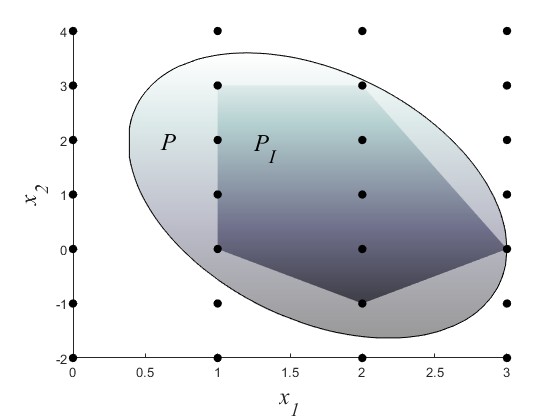}
        \caption{Bounded spectrahedron $P$ \label{Fig:example_bounded}}
    \end{subfigure}
    \begin{subfigure}[b]{0.48\textwidth}
        \includegraphics[scale=0.3]{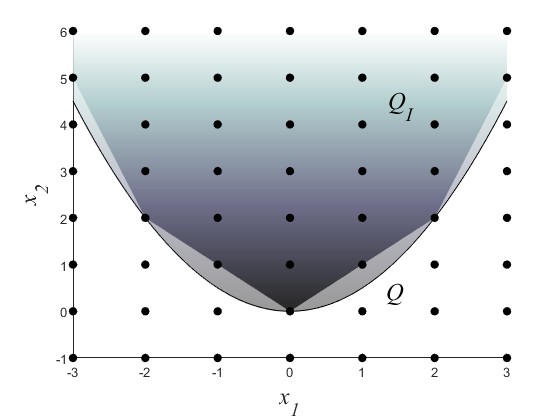}
        \caption{Unbounded spectrahedron $Q$ \label{Fig:example_unbounded}}
    \end{subfigure}
    \caption{Spectrahedra $P$ and $Q$ defined in Example~\ref{Example:spectrahedra}. Their corresponding integer hulls are given by the dark gradient areas.}
\end{figure}

\end{example}

\begin{example}[Example in $\mathbb{R}^3$] \label{Example:spectrahedra3D} 
    Let $\bold{C} = \begin{bsmallmatrix}
        1 & 2 \\ 2 & 2
    \end{bsmallmatrix} \oplus  \begin{bsmallmatrix}
        5 & 0 \\ 0 & 5
    \end{bsmallmatrix} , \bold{A_1} = \begin{bsmallmatrix}
        -1 & \frac{1}{2} \\ \frac{1}{2} & 1
    \end{bsmallmatrix} \oplus  \begin{bsmallmatrix}
        0 & 0 \\ 0 & 0
    \end{bsmallmatrix}, \bold{A_2} = \begin{bsmallmatrix}
        -\frac{3}{5} & \frac{3}{10} \\ \frac{3}{10} & 0
    \end{bsmallmatrix} \oplus  \begin{bsmallmatrix}
        1 & 0 \\ 0 & -1
    \end{bsmallmatrix}$ and $\bold{A_3} = \begin{bsmallmatrix}
        \frac{1}{2} & 2 \\ 2 & -3
    \end{bsmallmatrix} \oplus  \begin{bsmallmatrix}
        0 & 0 \\ 0 & 0
    \end{bsmallmatrix}$ and let $P$ be the induced spectrahedron of the form ~\eqref{def:P}. Then, $P$ is the semialgebraic set in $\mathbb{R}^3$ described by the inequalities $1\frac{1}{4}x_1^2 + \frac{9}{100}x_2^2 + 5\frac{1}{2}x_3^2 \leq -2 + 3x_1 + 2\frac{2}{5}x_2 + 10x_3 - \frac{9}{10} x_1x_2 + \frac{3}{5}x_2x_3 + 1\frac{1}{2}x_1x_3$, $1 + x_1 + \frac{3}{5}x_2 - \frac{1}{2}x_3 \geq 0$, $2 - x_1 + 3x_3 \geq 0$, $-5 \leq x_2$ and $x_2 \leq 5$, see Figure~\ref{Fig:spectra3D}. 

    \begin{figure}[H]
    \centering
    \includegraphics[scale = 0.25]{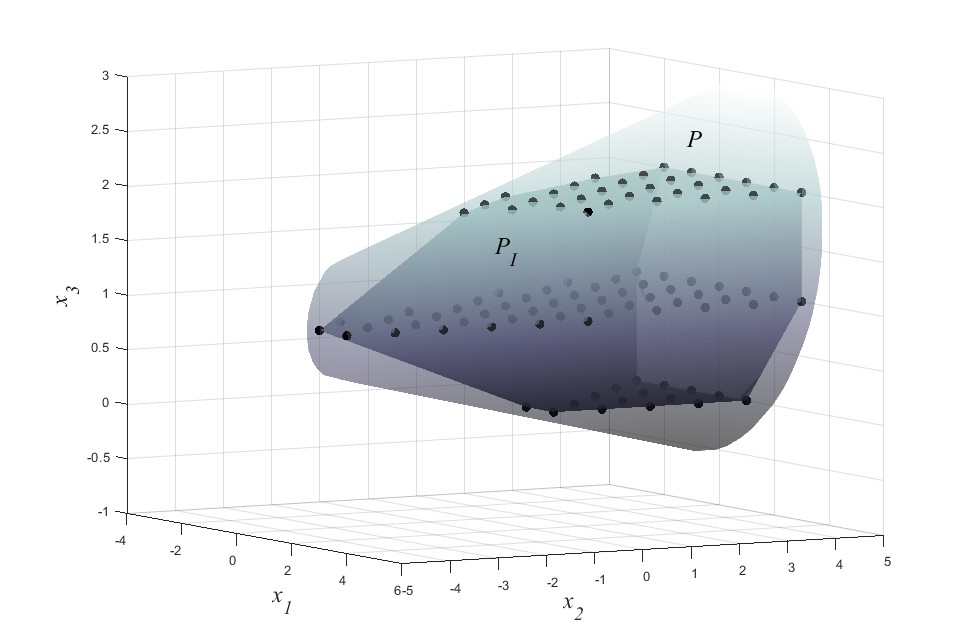}
    \caption{Spectrahedron $P$ in $\mathbb{R}^3$ defined in Example~\ref{Example:spectrahedra3D}. \label{Fig:spectra3D}}
\end{figure}

\end{example}

In the remaining part of this section we  study the  elementary closure,  see Definition~\ref{Def:elementaryclosure},  of spectrahedra in primal and dual standard {forms}. 
    The proofs of several results that will follow rely on the following semidefinite version of the theorem of alternatives, see e.g., Balakrishnan and Vandenberghe~\cite{Balakrishnan}.

\begin{proposition}[Theorem of the alternatives for SDP \cite{Balakrishnan}] \label{Prop:Farkas}
Let $\bold{C}, \bold{A_1}, \ldots, \bold{A_m} \in \mathcal{S}^n$. Then at most one of the following is true:
\begin{enumerate}
\item There exists an $\bold{X} \succ \bold{0}$, $\langle \bold{A_i}, \bold{X} \rangle = 0$ for all $i \in [m]$ and $\langle \bold{C}, \bold{X} \rangle \leq 0$;
\item There exists an $\bold{x} \in \mathbb{R}^m$ such that $\bold{C} - \sum_{i = 1}^m\bold{A_i}x_i \succneqq \bold{0}$.
\end{enumerate}
Moreover, if there exists no $\bold{x} \in \mathbb{R}^m$ such that $\sum_{i = 1}^m\bold{A_i}x_i  \succneqq \bold{0}$, then exactly one of the statements above is true.
\end{proposition}

Using the fact that a matrix $\bold{C} - \sum_{i = 1}^m \bold{A_i}x_i$ is positive semidefinite if and only if $\langle \bold{C} - \sum_{i = 1}^m \bold{A_i}x_i , \bold{U} \rangle \geq 0$ for all $\bold{U} \in \mathcal{S}^n_+$, we can rewrite $P$ as follows:
\begin{align}
P & = \left\{ \bold{x} \in \mathbb{R}^m \, : \, \, \langle \bold{C} - \sum_{i = 1}^m \bold{A_i}x_i, \bold{U} \rangle \geq 0 , \, \bold{U} \in \mathcal{S}^n_+ \right\} \nonumber \\
& = \bigcap_{\bold{U} \in \mathcal{S}^n_+} \left\{ \bold{x} \in \mathbb{R}^m \, : \, \, \sum_{i = 1}^m x_i \langle \bold{A_i}, \bold{U} \rangle \leq \langle \bold{C}, \bold{U} \rangle  \right\}. \label{P:dual}
\end{align}
Moreover, since $P$ is a closed convex set, we can write $P$ as the intersection of the halfspaces that contain it:
\begin{align} \label{P:intersect}
P = \bigcap_{\substack{(\bold{c},d) \in \mathbb{R}^{m+1} \\ P \subseteq \{\bold{x} \, : \, \, \bold{c}^\top \bold{x} \leq d \}}} \left\{ \bold{x} \in \mathbb{R}^m \, : \, \, \bold{c}^\top \bold{x} \leq d \right\}.
\end{align}
It is clear that all halfspaces in the intersection of \eqref{P:dual} are contained in the intersection \eqref{P:intersect}. The converse statement is also true, as stated by the following theorem.This theorem was proven in~\cite{DeMeijerThesis} and also related to the algebraic polar studied in~\cite{RamanaGoldman}.

\begin{theorem}[\cite{DeMeijerThesis,RamanaGoldman}] \label{Thm:dualcut} Let $P = \{\bold{x} \in \mathbb{R}^m \, : \, \, \bold{C} - \sum_{i = 1}^m \bold{A_i}x_i \succeq \bold{0} \}$ be a non-empty spectrahedron.
Let $(\bold{c},d) \in \mathbb{R}^{m+1}$ be such that $P \subseteq \{ \bold{x} \in \mathbb{R}^m \, : \, \, \bold{c}^\top \bold{x} \leq d \}$. Then there exists a matrix $\bold{U} \in \mathcal{S}^n_+$ such that $\langle \bold{A_i}, \bold{U} \rangle = c_i$ for all $i \in [m]$ and $\langle \bold{C}, \bold{U} \rangle \leq d$.
\end{theorem}

Using the representation of $P$ given by \eqref{P:dual} and the result of Theorem~\ref{Thm:dualcut}, we now provide an alternative formulation of the elementary closure for spectrahedra of the form $P$. We have,
\begin{align} \label{eq:coniccuts}
\cl(P) = \bigcap_{\substack{\bold{U} \in \mathcal{S}^n_+ \, \text{s.t.} \\ \langle \bold{A_i}, \bold{U} \rangle \in \mathbb{Z}, ~i \in [m]}} \left\{ x \in \mathbb{R}^m \, : \, \, \sum_{i = 1}^m x_i \langle \bold{A_i}, \bold{U} \rangle \leq \lfloor \langle \bold{C}, \bold{U} \rangle \rfloor  \right\}.
\end{align}
Hence, any possible CG cut for a spectrahedron is constructed by a matrix $\bold{U} \in \mathcal{S}^n_+$ such that $\langle \bold{A_i}, \bold{U} \rangle \in \mathbb{Z}$ for $i \in [m]$.

A similar alternative definition of the elementary closure of spectrahedra in standard primal form can be obtained. Let $Q \subseteq \mathcal{S}^n$ denote the continuous relaxation of the feasible set of \eqref{primalISDP}, i.e.,
\begin{align*}
Q & = \left\{ \bold{X} \in \mathcal{S}^n \, : \, \, \langle \bold{A_i}, \bold{X} \rangle = b_i, i \in [m], \, \bold{X} \succeq \mathbf{0} \right\} \\
&  = \left\{ \bold{X} \in \mathcal{S}^n \, : \, \, \langle \bold{A_i}, \bold{X} \rangle = b_i, i \in [m], \, \langle \bold{X}, \bold{U} \rangle \geq 0, \, \bold{U} \in \mathcal{S}^n_+ \right\} \\
&  = \left\{ \bold{X} \in \mathcal{S}^n \, : \, \, \left\langle \bold{X}, \bold{U} + \sum_{i = 1}^m \bold{A_i} \lambda_i \right\rangle \geq \sum_{i = 1}^m b_i \lambda_i, \, \bold{U} \in \mathcal{S}^n_+, \, \boldsymbol{\lambda} \in \mathbb{R}^m \right\},
\end{align*}
where the last equality follows from the fact that the choices $(\bold{U}, \boldsymbol{\lambda}) = (\bold{0}, \mathbf{e}_i )$ and $(\bold{U}, \boldsymbol{\lambda} ) = (\bold{0}, -\mathbf{e}_i)$ lead to the cuts $\langle \bold{A_i}, \bold{X} \rangle \geq b_i$ and $\langle \bold{A_i}, \bold{X} \rangle \leq b_i$, respectively. Now, the elementary closure of $Q$ can be described by the following intersection of CG cuts:
\begin{align} \label{CGclosure:primal}
\cl(Q) = \bigcap_{\substack{(\bold{U},\boldsymbol{\lambda}) \in \mathcal{S}^n_+ \times \mathbb{R}^m \, \text{s.t.} \\ \bold{U} + \sum_{i =1}^m \bold{A_i} \lambda_i \in \mathbb{Z}^{n \times n}}} \left\{ \bold{X} \in \mathcal{S}^n \, : \, \, \left\langle \bold{X}, \bold{U} + \sum_{i = 1}^m \bold{A_i} \lambda_i \right\rangle \geq \left\lceil \sum_{i = 1}^m b_i \lambda_i \right\rceil \right\}.
\end{align}
For many SDPs resulting from applications the spectrahedra that define the feasible sets are contained in the cone of non-negative vectors or matrices. When $P \subseteq \mathbb{R}^m_+$ or $Q \subseteq \{\bold{X} \in \mathbb{R}^{n \times n} \, : \, \, \bold{X} \geq \bold{0}\}$, alternative equivalent formulations of the elementary closure can be given, see also \cite{CezikIyengar}.

\begin{theorem}
Let $P = \left\{ \bold{x} \in \mathbb{R}^m_+ \, : \, \, \bold{C} - \sum_{i = 1}^m \bold{A_i}x_i \succeq \bold{0} \right\}$ { be a non-empty spectrahedron}. Then $\cl(P)$ can equivalently be written as
\begin{align} \label{eq:coniccuts2}
\cl(P) = \bigcap_{\bold{U} \in \mathcal{S}^n_+} \left\{ \bold{x} \in \mathbb{R}^m \, : \, \, \sum_{i = 1}^m x_i \lfloor \langle \bold{A_i}, \bold{U} \rangle \rfloor \leq \lfloor \langle \bold{C}, \bold{U} \rangle \rfloor  \right\}.
\end{align}
Similarly, let $Q = \left\{ \bold{X} \in \mathcal{S}^n \, : \, \, \langle \bold{A_i}, \bold{X} \rangle = b_i, i \in [m], \bold{X} \succeq \bold{0}, \bold{X} \geq \bold{0} \right\}$. Then $\cl(Q)$ can equivalently be written as
\begin{align}
\cl(Q) = \bigcap_{(\bold{U},\boldsymbol{\lambda}) \in \mathcal{S}^n_+ \times \mathbb{R}^m } \left\{ \bold{X} \in \mathcal{S}^n  \, : \, \, \left\langle \bold{X}, \left\lceil \bold{U} + \sum_{i = 1}^m \bold{A_i} \lambda_i \right\rceil \right\rangle \geq \left\lceil \sum_{i = 1}^m b_i \lambda_i \right\rceil \right\}.
\end{align}
\end{theorem}
\begin{proof}
We prove the statement for the dual form \eqref{eq:coniccuts2}. The proof for the primal form is similar.

Let $\overline{\cl(P)} := \bigcap_{\bold{U} \in \mathcal{S}^n_+} \left\{ \bold{x} \in \mathbb{R}^m \, : \, \, \sum_{i = 1}^m x_i \lfloor \langle \bold{A_i}, \bold{U} \rangle \rfloor \leq \lfloor \langle \bold{C}, \bold{U} \rangle \rfloor  \right\}$ and let $\cl(P)$ be as given in \eqref{eq:coniccuts}. The inclusion $\overline{\cl(P)} \subseteq \cl(P)$ is obvious, as any halfspace in the intersection defining $\cl(P)$ is also in the intersection defining $\overline{\cl(P)}$. Now, consider a halfspace $\bar{H} = \{ \bold{x} \in \mathbb{R}^m \, : \, \, \sum_{i = 1}^m x_i \lfloor \langle \bold{A_i}, \bold{U} \rangle \rfloor \leq \lfloor \langle \bold{C}, \bold{U} \rangle \rfloor \}$ for some $\bold{U} \in \mathcal{S}^n_+$, that is included in the intersection defining $\overline{\cl(P)}$. Since $P \subseteq \mathbb{R}^n_+$, we know
\begin{align*}
P  \subseteq \left\{ \bold{x} \in \mathbb{R}^m_+ \, : \, \, \sum_{i = 1}^m x_i \langle \bold{A_i}, \bold{U} \rangle \leq  \langle \bold{C}, \bold{U} \rangle   \right\} & \subseteq \left\{ \bold{x} \in \mathbb{R}^m_+ \, : \, \, \sum_{i = 1}^m x_i \lfloor \langle \bold{A_i}, \bold{U} \rangle \rfloor \leq  \langle \bold{C}, \bold{U} \rangle   \right\} \\
& \subseteq \left\{ \bold{x} \in \mathbb{R}^m \, : \, \, \sum_{i = 1}^m x_i \lfloor \langle \bold{A_i}, \bold{U} \rangle \rfloor \leq  \langle \bold{C}, \bold{U} \rangle   \right\}.
\end{align*}
Now we apply Theorem~\ref{Thm:dualcut} to the latter halfspace. It follows that there exists a matrix $\bold{V} \in \mathcal{S}^n_+$ such that
$$
\langle \bold{A_i}, \bold{V} \rangle = \lfloor \langle \bold{A_i}, \bold{U} \rangle \rfloor \quad \text{for all } i \in [m], \quad \text{and} \quad \langle \bold{C}, \bold{V} \rangle \leq \langle \bold{C}, \bold{U} \rangle.
$$
We define the halfspace $H := \{ \bold{x} \in \mathbb{R}^m \, : \, \, \sum_{i = 1}^m x_i \langle \bold{A_i}, \bold{V} \rangle \leq \lfloor \langle \bold{C}, \bold{V}\rangle \rfloor \}$. Since $\lfloor \langle \bold{C}, \bold{V} \rangle \rfloor \leq \lfloor \langle \bold{C}, \bold{U} \rangle \rfloor$, it follows that the halfspace $\bar{H}$ contains the halfspace $H$, while $H$ is contained in the intersection of $\cl(P)$ given in \eqref{eq:coniccuts}. Since this construction can be repeated for all halfspaces in the intersection \eqref{eq:coniccuts2} defining $\overline{\cl(P)}$, it follows that $\cl(P) \subseteq \overline{\cl(P)}$.
\end{proof}

\begin{figure}[H]
\begin{minipage}{0.35\textwidth}
\begin{example} 
Let us reconsider the bounded spectrahedron $P$ defined in Example~\ref{Example:spectrahedra}. The elementary closure $\cl(P)$ of this spectrahedron is the intersection of six rational halfspaces, represented by the dashed lines in Figure~\ref{Fig:example_bounded_closure}. Each such halfspace is obtained from a rational halfspace $\{\bold{x} \in \mathbb{R}^2 \, : \, \, \bold{c}^\top \bold{x} \leq d\}$ containing $P$, where $d$ is shifted towards $P_I$ until the corresponding hyperplane hits an integral point. The integer hull $P_I$ is the intersection of only five halfspaces. Thus, for this example we have $P_I \subsetneq \cl(P) \subsetneq P$.
\end{example} 
\end{minipage} \hfill\begin{minipage}{0.6\textwidth}
\centering
\includegraphics[scale=0.7]{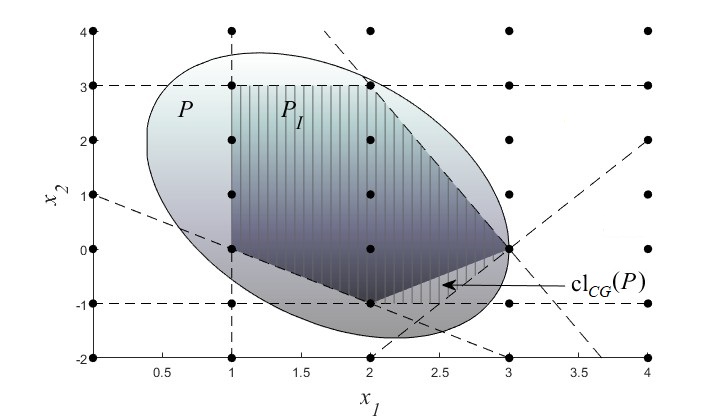}
\caption{Spectrahedron $P$, its integer hull $P_I$ and its elementary closure $\cl(P)$. \label{Fig:example_bounded_closure}}
\end{minipage}
\end{figure}

In Section~\ref{Section:TDI} we provide a polyhedral description of the elementary closure of spectrahedra that satisfy the notion of total dual integrality.

\subsection{The Chv\'atal rank of bounded spectrahedra}
In this section we derive several results on the sequence of relaxations resulting from the Chv\'atal-Gomory procedure. Although some of these results are already known for general compact convex sets, we provide simplified proofs for the case of bounded spectrahedra.
Throughout this section we assume $P$ to be a spectrahedron of the form \eqref{def:P} that is bounded. For unbounded sets it is in general not even clear whether $C^{(k+1)} \subseteq C^{(k)}$.

 {It is known that the Chv\'atal rank of a compact convex set  is finite, including the special case of bounded spectrahedra. 
 This result follows from the polyhedrality result of Dadush et al.~\cite{DadushEtAl}
  and the folklore that the Chv\'atal rank of a rational polytope is finite due to  Chv\'atal~\cite{Chvatal}. }
\begin{proposition}[\cite{Chvatal,DadushEtAl}] \label{Thm:FiniteCGsequence} Let $P = \left\{ \bold{x} \in \mathbb{R}^m \, : \, \, \bold{C} - \sum_{i = 1}^m \bold{A_i}x_i \succeq \bold{0} \right\}$ be bounded. Then, $P^{(k)} = P_I$ for some finite $k$.
\end{proposition} 

Next, we aim to prove a homogeneity property of the CG procedure for bounded spectrahedra, which states that the elementary closure operation commutes with taking the intersection with supporting hyperplanes. This property plays a key role in showing that the elementary closure of $P$ is a rational polytope, following the proof of Braun and Pokutta \cite{BraunPokutta}. We provide a simplified proof of this property for bounded spectrahedra, which can be seen as the conic analogue to a polyhedral result of Schrijver \cite{Schrijver}. In the proof we restrict ourselves to halfspaces of the form $\{\bold{x} \in \mathbb{R}^m \, : \, \, \bold{w}^\top \bold{x} \leq d\}$ where $\bold{w} \in \mathbb{Z}^m$ and $d \in \mathbb{R}$. It follows from Proposition~\ref{Prop:compactouter} that these halfspaces are sufficient to describe a compact convex set.

Before we show the main theorem, we need a chain of intermediate results, starting with a {proposition regarding the condition of Proposition~\ref{Prop:Farkas}}.

\begin{proposition} \label{Prop:CompactFull}
Let $P = \left\{ \bold{x} \in \mathbb{R}^m \, : \, \, \bold{C} - \sum_{i = 1}^m \bold{A_i}x_i \succeq \mathbf{0} \right\}$ be a non-empty and bounded spectrahedron. Then there does not exist an $\bold{x} \in \mathbb{R}^m$ such that $\sum_{i = 1}^m \bold{A_i}x_i \succneqq \bold{0}$.
\end{proposition}

\begin{proof}
Since $P$ is non-empty, there exists a point $\bold{x^*} \in P$, i.e., $\bold{C} - \sum_{i = 1}^m\bold{A_i}x^*_i \succeq \mathbf{0}$. Now suppose there exists a point $\bold{\hat{x}}$ such that  $\sum_{i = 1}^m \bold{A_i}\hat{x}_i \succneqq \bold{0}$. Then clearly $\bold{\hat{x}} \neq \mathbf{0}_m$ and for all $t \geq 0$ we have
$$
\bold{C} - \sum_{i = 1}^m\bold{A_i}x^*_i + t \sum_{i = 1}^m\bold{A_i}\hat{x}_i = \bold{C} - \sum_{i = 1}^m \bold{A_i}(x^*_i - t\hat{x}_i) \succeq \mathbf{0},
$$
i.e., $\bold{x^*} - t\bold{\hat{x}} \in P$ for all $t \geq 0$. Thus, $P$ is unbounded, so such $\bold{\hat{x}}$ cannot exist.
\end{proof}

We also need  Dirichlet's approximation theorem and its weakened version.
\begin{proposition}[Dirichlet's Approximation Theorem] \label{Prop:Dirichlet}
Let $d \in \mathbb{R}$ and $N \geq 2$ be a positive integer. Then there exist integers $p$ and $q$ with {$1\leq p \leq N$} such that $|pd - q| \leq \frac{1}{N}$.
\end{proposition}
We now derive its one-sided variant below.

\begin{corollary}[One-sided Approximation Theorem] 
\label{Cor:onesidedDirichlet} Let $d \in \mathbb{R}$ and $N \geq 2$ be a positive integer number. Then there exists an integer $p \in \mathbb{Z}_+$ such that
$d - \lfloor p d \rfloor \leq \frac{1}{N}.$
\end{corollary}

\begin{proof}
By Dirichlet's Theorem, we know that for the given $d$ and $N$, there exist integers $q_1$ and $q_2$ with { $1\leq q_1 \leq N$} such that $| q_1 d - q_2 | \leq \frac{1}{N}$. If $q_1d \geq q_2$, then we have
$q_1d - \lfloor q_1 d \rfloor \leq q_1 d - q_2 = |q_1 d - q_2 | \leq \frac{1}{N},$
so the choice $p = q_1$ leads to the desired result.
Next, we consider the case $q_1 d < q_2$, for which we have $-\frac{1}{N} \leq q_1 d - q_2 < 0$. Let $M \geq 1$ be the smallest integer such that $M(q_1 d - q_2) \leq -\frac{N-1}{N}$, which exists because $q_1 d - q_2 < 0$. For this $M$ we must have $-1 \leq M(q_1 d - q_2)$. Namely, if $M(q_1 d - q_2) < -1$, then $(M-1)(q_1 d - q_2) \leq -\frac{N-1}{N}$, contradicting the minimality of $M$. Thus,
\begin{align*}
-1 \leq M (q_1d - q_2) \leq -\frac{N-1}{N} \quad \Longleftrightarrow \quad 0 \leq Mq_1 d - (Mq_2 - 1) \leq \frac{1}{N}.
\end{align*}
Since $Mq_2 - 1$ is integer, it follows that
$Mq_1 d - \lfloor Mq_1 d \rfloor \leq Mq_1 d - (Mq_2 - 1) \leq 1/N,$
so taking $p = Mq_1$ gives the desired result.
\end{proof}
We are now ready to present a simplified proof of Braun and Pokutta \cite{BraunPokutta}
for the homogeneity property of the elementary closure of bounded spectrahedra, { see also Proposition 1 in~\cite{DadushEtAl}}.

\begin{theorem}[Homogeneity property of elementary closure] \label{Thm:closureHyperplane}
Let $P  = \{\bold{x} \in \mathbb{R}^m \, : \, \, \bold{C} - \sum_{i = 1}^m \bold{A_i}x_i \succeq \mathbf{0} \}$ be a bounded spectrahedron that is contained in a halfspace $\{\bold{x} \in \mathbb{R}^m : \bold{w}^\top \bold{x} \leq d\}$ with $\bold{w} \in \mathbb{Z}^m$ and $d \in \mathbb{R}$. Let $K := \{\bold{x} \in \mathbb{R}^m : \bold{w}^\top \bold{x} = d\}$. Then $\cl(P) \cap K = \cl(P \cap K)$.
\end{theorem}

\begin{proof}
If $P$ is empty the claim is obvious, hence we assume that $P$ is non-empty.

The inclusion $\cl(P \cap K) \subseteq \cl(P) \cap K$ is trivial. In order to prove the reverse statement, we assume that $H$ is a rational halfspace containing $P \cap K$, i.e., $H = \{\bold{x} \in \mathbb{R}^m \, : \, \, \bold{v}^\top \bold{x} \leq \alpha \}$ where $\bold{v}$ is a vector of relative prime integers. It suffices to show that there exists a halfspace $\hat{H}$ containing $P$ such that $\hat{H}_I \cap K \subseteq H_I$. As {$P \cap K$ is the intersection of all such halfspaces $H$}, we establish $\cl(P) \cap K \subseteq \cl(P \cap K)$.

For each $i \in [m]$ we define the following extended matrix $\bold{\tilde{A}_i} \in \mathcal{S}^{n +2}$:
$
\bold{\tilde{A}_i} := \begin{bsmallmatrix}
\bold{A_i} & \bold{0} & \bold{0} \\
\bold{0}^\top & -w_i & 0 \\
\bold{0}^\top & 0 & -v_i
\end{bsmallmatrix}.
$
We first show that there does not exist an $\bold{x} \in \mathbb{R}^m$ such that $\sum_{i = 1}^m \bold{\tilde{A}_i}x_i \succneqq \bold{0}$. For the sake of contradiction, suppose such a vector exists, i.e., we have $\sum_{i = 1}^m \bold{A_i}\tilde{x}_i \succeq \mathbf{0}$, $\bold{w}^\top\bold{\tilde{x}} \leq 0$ and $\bold{v}^\top\bold{\tilde{x}} \leq 0$ for some $\bold{\tilde{x}}$, but not all of them are satisfied with equality. Since $P$ is non-empty and bounded, it follows from Proposition~\ref{Prop:CompactFull}  that there does not exist an $\bold{x} \in \mathbb{R}^m$ such that $\sum_{i = 1}^m \bold{A_i}x_i \succneqq \bold{0}$. Hence, we must have $\sum_{i = 1}^m \bold{A_i}\tilde{x}_i = \bold{0}$. This implies that either $\bold{w}^\top \bold{\tilde{x}} < 0$ or $\bold{v}^\top \bold{\tilde{x}} < 0$, or both.

Since $P$ is contained in $\{\bold{x} \in \mathbb{R}^m \, : \, \, \bold{w}^\top \bold{x} \leq d\}$, it follows from Theorem~\ref{Thm:dualcut} that there exists $T \succeq \mathbf{0}$ such that $\langle \bold{A_i}, \bold{T} \rangle = w_i$ for all $i \in [m]$. Since $\sum_{i = 1}^m\bold{A_i}\tilde{x}_i = \bold{0}$, we have
$
\left \langle \sum_{i = 1}^m \bold{A_i} \tilde{x}_i , \bold{T} \right \rangle = \sum_{i = 1}^m \tilde{x}_i \langle \bold{A_i}, \bold{T} \rangle = \bold{w}^\top \bold{\tilde{x}} = 0.
$

Since $P \cap K$ is contained in $H = \{\bold{x} \in \mathbb{R}^m \, : \, \, \bold{v}^\top \bold{x} \leq \alpha\}$, we can in a similar fashion show that $v_i = \langle \bold{A_i}, \bold{S} \rangle + \beta w_i$ for some $\bold{S} \succeq \mathbf{0}$ and $\beta \in \mathbb{R}$. From this it follows that $\bold{v}^\top \bold{\tilde{x}} = 0$. We conclude that there exists no $\bold{x} \in \mathbb{R}^m$ such that $\sum_{i = 1}^m \bold{\tilde{A}_i}x_i \succneqq \bold{0}$.

Next, we define the following extended matrix $\bold{\tilde{C}} \in \mathcal{S}^{n+2}$ and parameter $\epsilon > 0$:
\begin{align*}
\bold{\tilde{C}} := \begin{bmatrix}
\bold{C} & \bold{0} & \bold{0} \\
\bold{0}^\top & -d & 0 \\
\bold{0}^\top & 0 & -(\alpha + \epsilon)
\end{bmatrix} \quad \text{and} \quad \epsilon := \begin{cases} \frac{1}{2} \left( \lceil \alpha \rceil - \alpha \right) & \text{if $\alpha$ is not integer,} \\
\frac{1}{2} & \text{otherwise.}
\end{cases}
\end{align*}
Since $P \cap K$ is contained in $H$, it follows that $(P \cap K) \cap \{\bold{x} \in \mathbb{R}^m \, : \, \, \bold{v}^\top \bold{x} \geq \alpha + \epsilon\} = \emptyset$. Equivalently, we know that there does not exist an $\bold{x} \in \mathbb{R}^m$ such that $\bold{\tilde{C}} - \sum_{i = 1}^m \bold{\tilde{A}_i}x_i \succneqq \bold{0}$. We can now apply Proposition~\ref{Prop:Farkas} to this system, from where it follows that the first of the two alternative statements should be satisfied. Hence, there exist $\bold{\hat{U}} \succ \bold{0}$, $\lambda > 0$ and $\mu > 0$ such that $\langle \bold{A_i}, \bold{\hat{U}} \rangle - w_i\lambda - v_i \mu = 0$ for all $i \in [m]$ and $\langle \bold{C}, \bold{\hat{U}} \rangle - d\lambda - (\alpha + \epsilon) \mu \leq 0$. Without loss of generality, we may assume that $\mu = 1$ and we define
$$
\hat{\alpha} := \langle \bold{C}, \bold{\hat{U}} \rangle \quad \text{and} \quad \hat{v}_i := \langle \bold{A_i}, \bold{\hat{U}} \rangle \text{ for all } i \in [m].
$$
It follows from above that this particular $\hat{\alpha}$ and $\bold{\hat{v}}$ satisfy
\begin{align} \label{eq:property1}
\hat{\alpha} \leq \alpha + \epsilon + d\lambda \quad \text{and} \quad \hat{v}_i = v_i + w_i\lambda \text{ for all } i \in [m].
\end{align}
Also, since $\bold{\hat{U}} \succ \bold{0}$, we know that for all $\bold{x} \in P$ we have
\begin{align} \label{eq:property2}
\bold{\hat{v}}^\top \bold{x} = \sum_{i = 1}^m \langle \bold{A_i}, \bold{\hat{U}} \rangle x_i = \left \langle \sum_{i = 1}^m \bold{A_i}x_i, \bold{\hat{U}} \right\rangle \leq \langle \bold{C}, \bold{\hat{U}} \rangle = \hat{\alpha},
\end{align}
where we use the fact that $\langle \bold{C} - \sum_{i = 1}^m\bold{A_i}x_i, \bold{\hat{U}} \rangle \geq 0$. Observe that the tuple $(\lambda, \bold{\hat{v}}, \hat{\alpha})$ can be replaced by $(\lambda + \lambda_0, \bold{\hat{v}} + \lambda_0 \bold{w}, \hat{\alpha} + \lambda_0 d)$ for all $\lambda_0 \geq 0$ without affecting \eqref{eq:property1} and \eqref{eq:property2}, where for the maintenance of \eqref{eq:property2} we use the fact that $P \subseteq \{\bold{x} \in \mathbb{R}^m \, : \, \, \bold{w}^\top \bold{x} \leq d\}$. Now we choose $\lambda_0$ such that $\lambda + \lambda_0 \in \mathbb{Z}_+$ and $d(\lambda + \lambda_0) - \lfloor d(\lambda + \lambda_0) \rfloor < \epsilon$, which can be done by Corollary~\ref{Cor:onesidedDirichlet}. Moreover, we define $d_f := d(\lambda + \lambda_0) - \lfloor d(\lambda + \lambda_0) \rfloor$.

Define $\hat{H} := \{\bold{x} \in \mathbb{R}^m \, : \, \, (\bold{\hat{v}} + \lambda_0\bold{w})^\top \bold{x} \leq \hat{\alpha} + \lambda_0 d\}$. It follows from \eqref{eq:property2} that $P \subseteq \hat{H}$. Moreover, we have
\begin{align*}
\hat{H}_I \cap K & \subseteq \{\bold{x} \in \mathbb{R}^m \, : \, \, (\bold{\hat{v}} + \lambda_0\bold{w})^\top \bold{x} \leq \lfloor \hat{\alpha} + \lambda_0d \rfloor \} \cap \{\bold{x} \in \mathbb{R}^m \, : \, \, \bold{w}^\top \bold{x} = d \} \\
& \subseteq \{\bold{x} \in \mathbb{R}^m \, : \, \, \bold{v}^\top \bold{x} + \bold{w}^\top \bold{x} (\lambda + \lambda_0)  \leq \lfloor \alpha + \epsilon + d(\lambda + \lambda_0) \rfloor, \bold{w}^\top \bold{x} = d \} \\
& \subseteq \{\bold{x} \in \mathbb{R}^m \, : \, \, \bold{v}^\top \bold{x} + d_f \leq  \lfloor \alpha + \epsilon + d_f \rfloor \} \\
& \subseteq \{\bold{x} \in \mathbb{R}^m \, : \, \, \bold{v}^\top \bold{x} \leq \lfloor \alpha \rfloor \} = H_I ,
\end{align*}
where the last inclusion follows from the fact that $d_f \geq 0$ and $\epsilon + d_f < 1$ if $\alpha$ is integer and $\epsilon + d_f < \lceil \alpha \rceil - \alpha$ otherwise.
\end{proof}

The result of Theorem~\ref{Thm:closureHyperplane} holds for any halfspace $\{ \bold{x} \in \mathbb{R}^m \, : \, \, \bold{w}^\top\bold{x} \leq d\}$ with $\bold{w} \in \mathbb{Z}^m$ containing $P$. In particular, it holds for all such halfspaces that support $P$, meaning that $P \cap K \neq \emptyset$, where $K$ is the corresponding hyperplane. In such case, the set $P \cap K$ defines a face of the spectrahedron. It is known that all proper faces of spectrahedra are exposed, meaning that they can be obtained as the intersection of $P$ with a supporting hyperplane. Note, however, that for the faces of bounded spectrahedra these hyperplanes are not necessarily such that the entries in $\bold{w}$ are integral, even if the data matrices describing the spectrahedron are rational (as is the case for polyhedra).

Homogeneity plays a key role in Braun and Pokutta's \cite{BraunPokutta} proof for the polyhedrality of the elementary closure of compact convex sets. For the sake of completeness, we include this result here for the case of bounded spectrahedra.

\begin{theorem} [Dadush et al.,~\cite{DadushEtAl}, Braun and Pokutta~\cite{BraunPokutta}]  \label{Thm:CGpolytope}
The elementary closure $\cl(P)$ of a bounded spectrahedron $P$ is a rational polytope.
\end{theorem}

From Theorem~\ref{Thm:CGpolytope} and the fact that the elementary closure of a rational polytope is again a rational polytope \cite{Schrijver}, it follows that the finite sequence
\begin{align*}
P = P^{(0)} \supseteq P^{(1)} \supseteq \ldots \supseteq P^{(k)} \supseteq P^{(k+1)} \supseteq \ldots \supseteq P_I,
\end{align*}
consists of rational polyhedra from the first closure onwards. Observe that the boundedness assumption cannot be relaxed. Indeed, if $P$ is unbounded, it is not even clear whether $P_I$ is a polyhedron, {as the following example suggests}.

\begin{example}
    
     Consider the spectrahedron $Q$ in Example~\ref{Example:spectrahedra}. The integer hull $Q_I$ is the convex hull of the integer points in the epigraph of $f(x_1) = \frac{1}{2}x_1^2$. This convex hull is not polyhedral. To verify this, observe that the recession cone of $Q_I$ is contained in the recession cone of $Q$, which is $ \textup{rec}(Q) := \{\bold{x} \in \mathbb{R}^2 \, : \, \, x_2 \geq 0,~x_1 = 0\}$. Since $Q_I$ is unbounded and $\textup{rec}(Q)$ has only one ray, the recession cone of $Q_I$ must also be $\textup{rec}(Q)$. If $Q_I$ would be polyhedral, this implies that the halfspace $x_1 \leq N$ supports $Q_I$ for some finite value of $N$. However, this cannot be true as $Q_I$ contains integral points $(x_1, x_2) \in \mathbb{Z}^2$ for arbitrarily large $x_1$. 

    One can verify that $\cl(Q) = Q_I$. Namely, each facet of $Q_I$ is induced by a line between the points $(2k, 2k^2), (2(k-1),2(k-1)^2) \in \mathbb{Z}^2$ for any $k \in \mathbb{Z}$. Let such line for a fixed $k$ be described by $x_2 = cx_1 + d$ with $c,d  \in \mathbb{Z}$. Then, the parallel line $x_2 = cx_1 + d - 1$ lies strictly below $Q$. This implies that the halfspace $x_2 \geq cx_1 + d - 1 + \epsilon$ for any $\epsilon > 0$ contains $Q$ and that its integer hull is $x_2 \geq cx_1 + d$. Therefore, all facet-defining inequalities of $Q_I$ have Chv\'atal rank one and $\cl(Q) = Q_I$. This shows that $\cl(Q)$ is not a polyhedron. 
\end{example}

\subsection{ The elementary closure of spectrahedra and total dual integrality} \label{Section:TDI}

In this section we derive a class of spectrahedra for which we can find an explicit expression for the elementary closure. For rational polyhedra such an expression can be derived from a total dual integral representation of the linear system \cite{Schrijver}. It is therefore not surprising that a similar construction can be applied for bounded spectrahedra, albeit with a bit more technicalities. {After connecting total dual integrality for SDPs to the elementary closure, we derive a characterization and several sufficient conditions for a linear matrix inequality to be totally dual integral.} 

Recently, De Carli Silva and Tun\c{c}el \cite{DeCarliSilvaTuncel} introduced a notion of total dual integrality for SDPs. The authors of \cite{DeCarliSilvaTuncel} argue that the term integrality in SDPs should be defined with care. For instance, the rank-one property that is sometimes used in the literature as the notion of SDP integrality is proven to be primal-dual asymmetric and therefore not the favoured choice. Instead, {the authors of~\cite{DeCarliSilvaTuncel} propose a notion of SDP integrality that is based on a set of integer generating matrices.}
\newtagform{nobrackets}{}{}
\usetagform{nobrackets}

\begin{definition}[Property $(\Pz \mathbb{Z})_{\mathcal{V}}$] {Let $\mathcal{V} := \{\bold{V_1}, \ldots, \bold{V_k}\} \subseteq \mathcal{S}^n_+$ be a finite  set of integer PSD matrices.} 
A matrix $\bold{X} \in \mathcal{S}_n^+$ satisfies integrality property { $(\Pz\mathbb{Z})_{\mathcal{V}}$}  if
\begin{align}  \tag{$(\Pz \mathbb{Z})_{\mathcal{V}}$} \label{PZ}
\bold{X} = \sum_{j \in [k]} y_j\bold{V_j} \quad \text{for some $\bold{y} \in \mathbb{Z}^k_+$.}
\end{align}
\end{definition}
The authors of \cite{DeCarliSilvaTuncel} restricted to the set $\mathcal{V} = \{ \mathbbm{1}_S\mathbbm{1}_S^\top \, : \, \, S \subseteq [n]\}$, which could be seen as a natural embedding for the combinatorial problems that are considered in \cite{DeCarliSilvaTuncel}. One could argue, however, that this embedding is rather arbitrary. For that reason, we consider a general set of generating matrices. Note that the matrices $\bold{X}$ that satisfy property  \ref{PZ} are also integral in the sense that $\bold{X} \in \mathbb{Z}^{n \times n}$. To overcome confusion between these definitions, we will always explicitly refer to property  \ref{PZ} if that notion is meant.

Now we present the definition of total dual integrality for SDPs,
see also \cite{DeCarliSilvaTuncel}.

\newtagform{brackets}{(}{)}
\usetagform{brackets}

\begin{definition}[Total dual integrality] 
Let $Z \subseteq \mathbb{Z}^m$.
A linear matrix inequality  $\bold{C} - \sum_{i = 1}^m \bold{A_i}x_i \succeq \bold{0}$ is called totally dual integral (TDI) {on $Z$ if there exists some finite set of integer PSD matrices $\mathcal{V}$ such that}, for every $\bold{b} \in Z$, the SDP dual to $\sup\left \{ \bold{b}^\top \bold{x} \, : \, \, \bold{C} - \sum_{i = 1}^m \bold{A_i}x_i \succeq 0 \right\}$ has an optimal solution satisfying property \ref{PZ} whenever it has an optimal solution.
\end{definition}

{
A main difference with the original definition of total dual integrality for polyhedra, see e.g.~\cite{EdmondsGiles}, is that we restrict the objective vectors for which dual integrality should hold to a subset $Z$ of $\mathbb{Z}^m$. As explained in \cite{DeCarliSilvaTuncel}, this follows from the fact that semidefinite programs often follow from lifted formulations. For instance, $Z$ could be the range from a linear lifting map, e.g., $Z = \{ 0 \oplus \bold{b}' \, : \, \, \bold{b}' \in \mathbb{Z}^{m-1}\}.$

Based on this restriction to vectors in $Z$, it makes sense to consider a relaxed version of the CG closure in which we take the intersection of halfspaces induced by coefficient vectors in $Z$. More precisely, we define the CG closure with respect to $Z$ as
\begin{align} \label{eq:elementaryclosure_relax}
\cl (P, Z) := \bigcap_{\substack{(\bold{c}, d) \in Z \times \mathbb{R} \\
P \subseteq \{\bold{x} \, : \, \, \bold{c}^\top \bold{x} \leq d \}}} \left\{ \bold{x} \in \mathbb{R}^m \, : \, \, \bold{c}^\top \bold{x} \leq \lfloor d \rfloor \right\}.
\end{align}
 This relaxation of the CG closure is also considered in the literature, see e.g., \cite{DadushEtAl2011, DadushEtAl}.   The standard CG closure $\cl(P)$ that we considered so far equals $\cl(P, \mathbb{Z}^m)$.

The following theorem shows that if a spectrahedron is defined by an LMI that is TDI on $Z$, its (relaxed) CG closure $\cl(P,Z)$ can be explicitly defined.
}
\begin{theorem} \label{Thm:TDI}
Let $P = \left\{ \bold{x} \in \mathbb{R}^m \, : \, \, \bold{C} - \sum_{i = 1}^m \bold{A_i}x_i \succeq \bold{0} \right\}$
be such that the LMI $\bold{C} - \sum_{i = 1}^m \bold{A_i}x_i \succeq \bold{0}$ is TDI {on $Z$} and satisfies Slater's condition. {Let $\mathcal{V} = \{\bold{V_1}, \ldots, \bold{V_k}\}$ denote the corresponding generating set of integer PSD matrices and suppose $\begin{bmatrix}
    \langle \bold{V_j}, \bold{A_1} \rangle & \cdots & \langle \bold{V_j}, \bold{A_m} \rangle
\end{bmatrix}^\top \in Z$ for all $j \in [k]$. } Define $\bold{B} \in \mathbb{Z}^{k \times m}$ and $\bold{d} \in \mathbb{Z}^{k}$ such that: 
$
B_{j,i} := \left\langle \bold{A_i}, \bold{V_j} \right \rangle  \text{ and } d_j := \left \lfloor \left\langle \bold{C}, \bold{V_j} \right \rangle \right \rfloor,
$
for all {$j \in [k]$ and $i \in [m]$}. Then,
$\cl(P, Z) = Q := \left \{ \bold{x} \in \mathbb{R}^m \, : \, \, \bold{B}\bold{x} \leq \bold{d}\right \}.$
\end{theorem}
\begin{proof}
To prove that $\cl(P, Z) \subseteq Q$, {observe that $\bold{V_j} \succeq \zb$ with $\begin{bmatrix}
    \langle \bold{V_j}, \bold{A_1} \rangle & \cdots & \langle \bold{V_j}, \bold{A_m} \rangle
\end{bmatrix}^\top \in Z$ for all $j \in [k]$}. Consequently, we know that $P \subseteq \left\{ \bold{x} \in \mathbb{R}^m \, : \, \, \sum_{i = 1}^m x_i \langle \bold{A_i}, {\bold{V_j} }\rangle \leq \langle \bold{C}, {\bold{V_j} } \rangle \right \}$. It follows from \eqref{eq:elementaryclosure_relax} that $\cl(P,Z) \subseteq \{ \bold{x} \in \mathbb{R}^m \, : \, \, \sum_{i = 1}^m x_i \langle \bold{A_i}, {\bold{V_j} } \rangle \leq \left \lfloor \langle \bold{C}, {\bold{V_j} } \rangle \right \rfloor \}$. Since all inequalities in $\bold{Bx} \leq \bold{d}$ are of this form, it follows that $\cl(P, Z) \subseteq Q$.

To prove the reverse direction, let $H := \left\{\bold{x} \in \mathbb{R}^m \, : \, \, { \bold{b}}^\top \bold{x} \leq q\right\}$ be a halfspace containing $P$ {with ${\bold{b}} \in Z$}. Since $P \subseteq H$, we have
\begin{align}
q & \geq \sup_{\bold{x}} \left\{ {\bold{b}}^\top \bold{x} \, : \, \, \bold{C} - \sum_{i = 1}^m \bold{A_i}x_i \succeq \bold{0} \right\} \label{eq:TDIdual}\\
& = \inf_{\bold{X}} \left\{ \langle \bold{C}, \bold{X} \rangle \, : \, \, \langle \bold{A_i}, \bold{X} \rangle = {b_i}, \, i \in [m],  \, \bold{X} \succeq \bold{0} \right\}, \label{eq:TDIprimal}
\end{align}
where strong duality among \eqref{eq:TDIdual} and \eqref{eq:TDIprimal} holds since the former problem has a Slater feasible point. By the same argument, we know that the infimum in \eqref{eq:TDIprimal} is attained. Since $\bold{C} - \sum_{i = 1}^m \bold{A_i}x_i \succeq \bold{0}$ is TDI {on $Z$}, it follows that there exists an optimal solution $\bold{\hat{X}}$ to \eqref{eq:TDIprimal} satisfying property~\ref{PZ}. In other words, there exists an $\bold{\hat{y}} \in \mathbb{Z}_+^{k}$ such that
\begin{align*}
{\bold{\hat{X}} = \sum_{j \in [k]}\hat{y}_j \bold{V_j}}, \quad \langle \bold{A_i}, \bold{\hat{X}} \rangle = {b_i} \text{ for all } i \in [m], \quad \bold{\hat{X}} \succeq 0.
\end{align*}
Consequently, we have
$\lfloor q \rfloor  \geq \lfloor \langle \bold{C}, \bold{\hat{X}} \rangle \rfloor  = {\left \lfloor \sum_{j \in [k]} \hat{y}_j \left \langle \bold{C}, \bold{V_j} \right \rangle \right \rfloor \geq  \sum_{j \in [k]} \hat{y}_j \left \lfloor\left \langle \bold{C}, \bold{V_j} \right \rangle \right \rfloor } = \bold{d}^\top \bold{\hat{y}}.$
Now, consider the following linear optimization problem and its corresponding dual:
\begin{align*}
\max \{ {\bold{b}}^\top \bold{x} \, : \, \, \bold{B}\bold{x} \leq \bold{d} \} = \min\{ \bold{d}^\top \bold{y} \, : \, \, \bold{y} \geq \bold{0}, \bold{y}^\top \bold{B} = {\bold{b}}^\top \}.
\end{align*}
Since $\bold{\hat{y}} \geq \bold{0}$ and $(\bold{\hat{y}}^\top \bold{B} )_i = {\sum_{j \in [k]} \hat{y}_j \langle \bold{A_i}, \bold{V_j} \rangle } = \langle \bold{A_i}, \bold{\hat{X}} \rangle = { b_i}$, the solution $\bold{\hat{y}}$ is feasible for the minimization problem above. This yields
$\max \{ {\bold{b}}^\top \bold{x} \, : \, \, \bold{B}\bold{x} \leq \bold{d} \} \leq \bold{d}^\top\bold{\hat{y}} \leq \lfloor q \rfloor.$
Hence, $Q \subseteq \left\{ \bold{x} \in \mathbb{R}^m \, : \, \, {\bold{b}}^\top \bold{x} \leq \lfloor q \rfloor \right\}$. Since this holds for all {halfspaces $H$ induced by coefficient vectors in $Z$}, it follows that $Q \subseteq \cl(P, Z)$.
\end{proof}
For the special case where $Z = \mathbb{Z}^m$, Theorem~\ref{Thm:TDI} provides a closed-form expression for $\cl(P)$. Observe that for that special case the condition that $\begin{bmatrix}
    \langle \bold{V_j}, \bold{A_1} \rangle & \cdots & \langle \bold{V_j}, \bold{A_m} \rangle
\end{bmatrix}^\top \in Z$ for all $j \in [k]$ can be simplified to $\langle \bold{A_i}, \bold{V_j} \rangle \in \mathbb{Z}$ for all $i \in [m]$ and $j \in [k]$.

Besides providing a closed-form expression for $\cl(P)$, Theorem~\ref{Thm:TDI} can be used to identify bounded spectrahedra for which $P = P_I$. Namely, if the matrix $\bold{C}$ is such that $\langle \bold{C}, \bold{V_j} \rangle \in \mathbb{Z}$ for all $j \in [k]$, then $P \subseteq Q$. {For spectrahedra that are bounded,} this implies that the chain $Q = \cl(P) \subseteq P \subseteq Q$ holds with equality, hence $\cl(P) = P$. As $P^{(k)} = P_I$ for some finite $k$ for all bounded spectrahedra, we must have $P = P_I$. De Carli Silva and Tun\c{c}el \cite{DeCarliSilvaTuncel} show that this, for example, happens for the SDP formulation of the Lov\'asz theta function when the underlying graph is perfect.

\medskip 

{A natural question is under which conditions a linear matrix inequality is TDI on a certain set $Z$. Below we first derive a full characterization of LMIs that are totally dual integral on the full set $\mathbb{Z}^m$. The characterization relates to the faces of the spectrahedron induced by the LMI. It is well-known that the faces of $\mathcal{S}^n_+$ are associated with linear subspaces of $\mathbb{R}^n$, see e.g., \cite{BarkerCarlson}. In the same vein, the facial structure of a spectrahedron can be characterized as follows.

\begin{lemma}[Ramana and Goldman~\cite{RamanaGoldman}]
    Let $P = \{\xb \in \mathbb{R}^m \, : \, \, \Cb - \sum_{i = 1}^m \Ab x_i \succeq \zb \}$ be a spectrahedron and let $F \subseteq P$ be a nonempty face of $P$. Then, there exists a subspace $\mathcal{R}_F \subseteq \mathbb{R}^n$ such that 
    \begin{align*}
        F = \left\{\xb \in P \, : \, \, \mathcal{R}_F \subseteq \Nul \left( \Cb - \sum_{i = 1}^m \Ab x_i \right) \right\},
    \end{align*}
where any point $\xb$ in the relative interior of $F$ satisfies~$\Nul \big( \Cb - \sum_{i = 1}^m \Ab x_i \big) = \mathcal{R}_F$.  \label{Lem:RamanaGoldman}
\end{lemma} Lemma~\ref{Lem:RamanaGoldman} implies that in the particular case where the face $F$ of $P$ is an extreme point~$\bold{\bar{x}}$, we have $\mathcal{R}_{\bold{\bar{x}}} = \Nul ( \Cb - \sum_{i=1}^m \Ab \bar{x}_i)$.  
For any nonempty face $F$ of $P$, we define the cone of objective vectors $\bold{b}$ for which the elements in $F$ maximize $\bold{b}^\top \xb$ over $P$, i.e.,
\begin{align} \label{Def:KF}
    K(F) := \left\{ \bold{b} \in \mathbb{R}^m \, : \, \, \bold{b}^\top \bold{y} = \max \{\bold{b}^\top \xb \, : \, \, \xb \in P \} \text{ for all } \bold{y} \in F  \right\}. 
\end{align}
For any proper face $F \subseteq P$, the cone $K(F)$ is nonempty and equals the intersection over all normal cones of $P$ at the points in $F$. 

Next, we recall the definition of a so-called Hilbert basis. 
\begin{definition}
    A set $\{\bold{v}_1, \ldots, \bold{v}_k\} \subseteq \mathbb{Z}^m$ is a Hilbert basis if every integral vector~$\xb \in \cone (\{\bold{v}_1, \ldots, \bold{v}_k\})$ can be written as $\xb = \sum_{j = 1}^k \alpha_j \bold{v}_j$, $\alpha_j \geq 0$, $\alpha_j \in \mathbb{Z}$, for all $j \in [k]$. 
\end{definition}
By abuse of terminology, we will refer to an LMI whose solution set is bounded as a bounded LMI. The following theorem provides a full characterization of bounded LMIs that are TDI on the full set of integer vectors.  

\begin{theorem} \label{Thm:Hilbert} Let the linear matrix inequality  $\Cb - \sum_{i = 1}^m \Ab x_i \succeq \bold{0}$ be bounded and assume Slater's condition holds. Then, $\bold{C} - \sum_{i = 1}^m \bold{A_i}x_i \succeq \bold{0}$ is totally dual integral on $\mathbb{Z}^m$ if and only if there exists some finite set of integer PSD matrices $\mathcal{V} = \{\bold{V_1}, \ldots, \bold{V_k}\}$ such that for each extreme point $\bold{\bar{x}}$ of the induced spectrahedron $P = \{\bold{x} \in \mathbb{R}^m \, : \, \, \bold{C} - \sum_{i = 1}^m \bold{A_i}x_i \succeq \bold{0}\}$ with $K(\bold{\bar{x}}) \cap \mathbb{Z}^m \neq \emptyset$, the vectors 
    \begin{align*}
    \bold{g_j} := 
        \begin{bmatrix}
            \langle \bold{A_1}, \bold{V_j} \rangle & \dots & \langle \bold{A_m}, \bold{V_j} \rangle
            \end{bmatrix}^\top \quad \text{for } j \in J:= \{ j \in [k] \, : \, \, \Coll(\bold{V_j}) \subseteq \mathcal{R}_{\bold{\bar{x}}} \} 
    \end{align*}
    form a Hilbert basis of $K(\bold{\bar{x}})$. 
\end{theorem}
\begin{proof}
    Let $\bold{b} \in \mathbb{Z}^m$. Since $P$ is bounded, the maximum of $\bold{b}^\top \xb$ over $\xb \in P$ is attained at a face of $P$. Thus, there exists an extreme point $\bold{\bar{x}}$ of $P$ with $\bold{b} \in K(\bold{\bar{x}})$. As $P$ contains a Slater feasible point, we have
    \begin{align} \label{ProofEq:Hilbertduality}
        \max_\xb \left\{ \bold{b}^\top \xb  \, : \, \, \xb \in P \right\} = \min_{\Xb} \left\{ \langle \Cb, \Xb \rangle \, : \, \, \langle \Ab, \Xb \rangle = b_i,~i \in [m],~\Xb \succeq \zb \right\}.
    \end{align}
The point $\bold{\bar{x}}$ is optimal for the maximization problem above. 
    Complementary slackness then implies that any $\Xb$ optimal to the dual problem should satisfy~${(\Cb - \sum_{i = 1}^m \Ab \bar{x}_i) \Xb = \zb}$, or equivalently, $\Coll(\Xb) \subseteq \Nul(\Cb - \sum_{i=1}^m \Ab \bar{x}_i) = \mathcal{R}_{\bold{\bar{x}}}$.
    To show that   $\bold{g_j}$ is contained in $K(\bold{\bar{x}})$ for~$j \in J$, we first observe that $\bold{V_j}$ is feasible for the minimization problem $$\min_{\Xb} \left\{ \langle \Cb, \Xb \rangle \, : \, \, \langle \Ab, \Xb \rangle = (\bold{g_j})_i,~i \in [m],~\Xb \succeq \zb \right\}.$$
    Then, since  $\Coll(\bold{V_j}) \subseteq \mathcal{R}_{\bold{\bar{x}}}$, we know that $(\Cb - \sum_{i = 1}^m \Ab \bar{x}_i) \bold{V_j} = \zb$. Therefore, 
    $\bold{\bar{x}}$ and $\bold{V_j}$ are optimal solutions 
    to  $\max_\xb \left\{ \bold{g_j}^\top \xb  \, : \, \, \xb \in P \right\}$ and $\min_{\Xb} \{ \langle \Cb, \Xb \rangle \, : \, \, \langle \Ab, \Xb \rangle = (\bold{g_j})_i,~i \in [m],~\Xb \succeq \zb \}$, respectively. This    
    implies that $\bold{g_j}$ is indeed contained in~$K(\bold{\bar{x}})$ for $j \in J$.

    Now, suppose that the vectors $\bold{g_j}$, $j \in J$ form a Hilbert basis of $K(\bold{\bar{x}})$. Then, we have~${\bold{b} = \sum_{j \in J}\alpha_j \bold{g_j}}$ for some $\alpha_j \geq 0$, $\alpha_j \in \mathbb{Z}$, $j \in J$. Consequently,~${\Xb := \sum_{j \in J} \alpha_j \bold{V}_j}$ is feasible for the minimization problem in~\eqref{ProofEq:Hilbertduality} with $\Coll(\Xb) \subseteq \mathcal{R}_{\bold{\bar{x}}}$. Since this establishes complementary slackness between $\Xb$ and $\bold{\bar{x}}$, it follows that $\Xb$ is a dual optimal solution that satisfies property~\ref{PZ}. 

    Conversely, if the LMI is totally dual integral on $\mathbb{Z}^m$, it follows that the dual problem in~\eqref{ProofEq:Hilbertduality} has an optimal solution $\Xb$ satisfying property~\ref{PZ}. Therefore, $\Xb = \sum_{j = 1}^k \alpha_j \bold{V_j}$ for some $\alpha_j \geq 0$, $\alpha_j \in \mathbb{Z}$, $j \in [k]$. Now, let $J^C := [k] \setminus J$. Then, 
    \begin{align*}
        \Xb = \sum_{j \in J} \alpha_j \bold{V_j} + \sum_{j \in J^C} \alpha_j \bold{V_j}. 
    \end{align*}
    By complementary slackness, we have $\Coll(\Xb) \subseteq \mathcal{R}_{\bold{\bar{x}}}$, implying that $\Coll( \sum_{j \in J^C} \alpha_j \bold{V_j}) = \Coll(\Xb - \sum_{j \in J}\alpha_j \bold{V_j}) \subseteq \mathcal{R}_{\bold{\bar{x}}}$. Since the $\bold{V_j}$'s are positive semidefinite, we also know that $\Coll(\alpha_j \bold{V_j}) \subseteq \Coll( \sum_{j \in J^C} \alpha_j \bold{V_j}) \subseteq \mathcal{R}_{\bold{\bar{x}}}$ for all $j \in J^C$. However, by the definition of $J^C$ we have $\Coll(\bold{V_j}) \nsubseteq \mathcal{R}_{\bold{\bar{x}}}$, so we must have $\alpha_j = 0$ for all $j \in J^C$. We conclude that $\Xb$ is a nonnegative integer combination of the matrices $\bold{V_j}$ with $j \in J$. By the constraints of the minimization problem in~\eqref{ProofEq:Hilbertduality}, it finally follows that $\bold{b} = \sum_{j \in J} \alpha_j \bold{g_j}$. As the construction can be repeated for all $\bold{b} \in \mathbb{Z}^m$ in $K(\bold{\bar{x}})$, we conclude that $\{ \bold{g_j} \, : \, \, j \in J\}$ indeed forms a Hilbert basis of $K(\bold{\bar{x}})$. The same holds for all other extreme points $\bold{\bar{x}}$ for which $K(\bold{\bar{x}}) \cap \mathbb{Z}^m \neq \emptyset$.  
\end{proof}
Theorem~\ref{Thm:Hilbert} has a significant implication on the structure of the induced spectrahedron of a bounded LMI that is TDI on $\mathbb{Z}^m$. 
\begin{corollary} \label{Cor:TDIpolyhedral}
    If a bounded LMI $\bold{C} - \sum_{i = 1}^m \bold{A_i}x_i \succeq \bold{0}$ that satisfies Slater's condition is totally dual integral on $\mathbb{Z}^m$, the spectrahedron ${P = \{ \xb \in \mathbb{R}^m \, : \, \, \bold{C} - \sum_{i = 1}^m \bold{A_i}x_i \succeq \bold{0}\}}$ is polyhedral. 
\end{corollary}

\begin{proof}
    Let $h_P: \mathbb{R}^m \rightarrow \mathbb{R}$ denote the support function of $P$, i.e., $h_P(\bold{x}) := \sup_{\bold{a} \in P} \{\bold{x}^\top \bold{a}\}$ and let $(\bold{c},d) \in \mathbb{Z}^m \times \mathbb{R}$ be such that $P \subseteq \{\bold{x} \in \mathbb{R}^m \, : \, \, \bold{c}^\top \bold{x} \leq d\}$. Then, there exists an extreme point $\bold{\bar{x}}$ of $P$ such that $\bold{c} \in K(\bold{\bar{x}})$.  By Theorem~\ref{Thm:Hilbert}, it follows that there exists a subset $J \subseteq [k]$ and $\alpha_j \geq 0$, $\alpha_j \in \mathbb{Z}$, $j \in J$ such that $\bold{c} = \sum_{j \in J} \alpha_j \bold{g_j}$. Obviously, $h_P(\bold{c}) = \bold{c}^\top \bold{\bar{x}}$ and, since $\bold{g_j} \in K(\bold{\bar{x}})$, $h_P(\bold{g_j}) = \bold{g_j}^\top \bold{\bar{x}}$ for all $j \in J$. Now, the conical combination of the inequalities $\bold{g_j}^\top \bold{x} \leq h_P(\bold{g_j})$, each with weight $\alpha_j$, results in $\sum_{j \in J} \alpha_j \bold{g_j}^\top \bold{x} \leq \sum_{j \in J} \alpha_j h_P(\bold{g_j}) = \sum_{j \in J}\alpha_j\bold{g_j}^\top \bold{\bar{x}} = \bold{c}^\top \bold{\bar{x}} \leq d$. Since the left-hand side equals $\bold{c}^\top \bold{x}$, the halfspace $\{\bold{x} \in \mathbb{R}^m \, : \, \, \bold{c}^\top \bold{x} \leq d\}$ is implied by the inequalities $\bold{g_j}^\top \bold{x} \leq h_P(\bold{g_j})$, $j \in [k]$. 

Since this construction can be repeated for all halfspaces of the form $\{\bold{x} \in \mathbb{R}^m \, : \, \, \bold{c}^\top \bold{x} \leq d\}$ where $(\bold{c},d) \in \mathbb{Z}^m \times \mathbb{R}$, and $P$ equals the intersection of all such halfspaces, see Proposition~\ref{Prop:compactouter}, it follows that $P$ is contained in the polyhedron induced by $\bold{g_j}^\top \bold{x} \leq h_P(\bold{g_j})$, $j \in [k]$. Since the converse inclusion is also true, $P$ is polyhedral.
\end{proof}

Corollary~\ref{Cor:TDIpolyhedral} implies that the only bounded linear matrix inequalities that may be TDI on $\mathbb{Z}^m$ can be described by a finite number of linear inequalities. This is the case, for instance, when the matrices $\Cb$ and $\Ab$, $i \in [m]$, are diagonal or simultaneously diagonalizable. In general, it is $\mathcal{NP}$-hard to decide whether a spectrahedron is polyhedral, see Ramana~\cite{Ramana}. The following result provides a characterization of polyhedral spectrahedra that are full-dimensional. Observe that any spectrahedron can be transformed to a full-dimensional spectrahedron by a restriction to its affine hull.  
\begin{theorem}[Ramana~\cite{Ramana}] \label{Thm:RamanaSpectrahedron}
    Let $P = \{ \xb \in \mathbb{R}^m \, : \, \, \Cb - \sum_{i = 1}^m \Ab x_i \succeq \zb \}$ be a full-dimensional spectrahedron. Then, $P$ is polyhedral if and only if there exists a non-singular matrix $M \in \mathbb{R}^{n \times n}$ and $\bold{d}, \bold{a_i} \in \mathbb{R}^{\ell}$,  $\Cb', \Ab' \in \mathcal{S}^{n - \ell}$, $i \in [m]$, with $\ell \leq n$
    such that for all $\xb \in \mathbb{R}^m$ we have
    \begin{align} \label{Eq:SpecIsPoly}\small
        M\left(\Cb - \sum_{i = 1}^m \Ab x_i \right)M^\top = \begin{bmatrix}
            \Cb' - \sum_{i = 1}^m \Ab' x_i & \zb \\
            \zb & \Diag(\bold{d}) - \sum_{i = 1}^m \Diag(\bold{a_i})x_i
        \end{bmatrix}
    \end{align}
     with $P = \{\xb \in \mathbb{R}^m \, : \, \, \Diag(\bold{d}) - \sum_{i = 1}^m \Diag(\bold{a_i})x_i \succeq \zb\}$. 
\end{theorem}

It is well-known that any rational polyhedron $P$ can be described by a totally dual integral system of linear inequalities, see Giles and Pulleyblank~\cite{GilesPulleyblank}. Hence, if a spectrahedron $P$ satisfies Theorem~\ref{Thm:RamanaSpectrahedron} with rational $\bold{d}, \bold{a_i}$ for all $i \in [m]$, then $P$ is totally dual integral on $\mathbb{Z}^m$ with respect to generating matrices $\mathcal{V} = \{\Diag(\bold{e_1}), \ldots, \Diag(\bold{e_n})\} \subseteq \mathcal{S}^{\ell}_+$. 

By relaxing the notion of total dual integrality to a strict subset $Z$ of $\mathbb{Z}^m$, it might be possible to identify other conditions of TDIness that go beyond polyhedrality. In return, the best one can hope for is a description of $\cl(P,Z)$, see Theorem~\ref{Thm:TDI}.  

As shown by Bhardwaj et al.~\cite{BhardwajEtAl}, any full-dimensional spectrahedron $P$ can be expressed by a linear matrix inequality in the form of~\eqref{Eq:SpecIsPoly}, even if $P$ is non-polyhedral. When the residual linear matrix form $\Cb' - \sum_{i = 1}^m \Ab' x_i$ cannot be further diagonalized, the form on the right-hand side of~\eqref{Eq:SpecIsPoly} is called the normal form of the linear matrix inequality. Intuitively speaking, the bottom right block of~\eqref{Eq:SpecIsPoly} can be viewed as the polyhedral part of the spectrahedron. As an extension of the result by Giles and Pulleyblank~\cite{GilesPulleyblank}, the following result shows that the polyhedral part of a spectrahedron can, under mild conditions, be made totally dual integral on an appropriate set $Z$. 

\begin{theorem}\label{Thm:polyhedralonZ} Let $P = \{ \xb \in \mathbb{R}^m \, : \, \, \Cb - \sum_{i = 1}^m \Ab x_i \succeq \zb \}$ be a full-dimensional spectrahedron that can be written in the normal form~\eqref{Eq:SpecIsPoly} for some non-singular matrix $M \in \mathbb{R}^{n \times n}$ and $\bold{d}, \bold{a_i} \in \mathbb{Q}^{\ell}$,  $\Cb', \Ab' \in \mathcal{S}^{n - \ell}$, $i \in [m]$ with $ 1 \leq \ell \leq n$. Let $Z \subseteq \mathbb{Z}^m$ be such that
\begin{align*}
    \max_{\xb} \left \{\bold{b}^\top \xb \, : \, \, \xb \in P \right \} = \max_{\xb} \left \{ \bold{b}^\top \xb \, : \, \, \Diag(\bold{d}) - \sum_{i = 1}^m \Diag(\bold{a_i})x_i \succeq \zb \right\}
    \end{align*}
    for all $\bold{b} \in Z$. Then there exists a linear matrix inequality describing $P$ that is totally dual integral on $Z$. 
\end{theorem}
\begin{proof}
    Let $Q = \{\xb \in \mathbb{R}^m \, : \, \, \Diag(\bold{d}) - \sum_{i=1}^m \Diag(\bold{a_i})x_i \succeq \zb\}$. Since $\bold{d}$ and $\bold{a_i}$ are rational for all $i \in [\ell]$, it follows from Giles and Pulleyblank~\cite{GilesPulleyblank} that there exists some totally dual integral representation of $Q$, i.e., $Q = \{\xb \in \mathbb{R}^m \, : \, \, \bold{\hat{A}}\xb \leq \bold{\hat{d}}\}$ for some $\bold{\hat{A}} \in \mathbb{Z}^{\ell' \times m}, \bold{\hat{d}} \in \mathbb{Q}^{\ell'}$ with $\bold{\hat{A}}\xb \leq \bold{\hat{d}}$ TDI. For all $i \in [m]$, let $\bold{\hat{a}_i}$ denote the $i$th column of $\bold{\hat{A}}$. Then, $P$ can be written as
    \begin{align} \label{Eq:TDI_halfpolyhedral}
        P = \left\{ \xb \in \mathbb{R}^m \, : \, \, \begin{bmatrix}
            \Cb' - \sum_{i = 1}^m \Ab' x_i & \zb \\
            \zb & \Diag(\bold{\hat{d}}) - \sum_{i = 1}^m \Diag(\bold{\hat{a}_i})x_i
        \end{bmatrix} \succeq \zb \right\}.  
    \end{align}
    We will show that the LMI in~\eqref{Eq:TDI_halfpolyhedral} is totally dual integral on $Z$. For any $\bold{b} \in Z$, we have that 
    \begin{align*}
        \max_{\xb} \left \{\bold{b}^\top \xb \, : \, \, \xb \in P \right \} = \max_{\xb}\left \{\bold{b}^\top \xb \, : \, \, \xb \in Q \right\} = \min_{\yb} \left \{\bold{\hat{d}}^\top \yb \, : \, \, \yb \geq \zb,~\yb^\top \bold{\hat{A}} = \bold{b}^\top \right \}. 
    \end{align*}
    By construction, the minimization problem above has an optimal solution $\bold{\hat{y}} \in \mathbb{Z}_+^{\ell'}$. Now, we define
    \begin{align*}
        \bold{\hat{X}} := \begin{bmatrix}
            \zb & \zb \\
            \zb & \Diag(\bold{\hat{y}})
        \end{bmatrix} \in \mathbb{S}^{n - \ell}_+ \oplus \mathbb{S}^{\ell'}_+. 
    \end{align*}
    It follows from above that 
      $        \left \langle \begin{bsmallmatrix}
            \Cb' & \zb \\
            \zb & \Diag(\bold{\hat{d}})
        \end{bsmallmatrix}, \bold{\hat{X}} \right \rangle = \bold{\hat{d}}^\top \bold{\hat{y}}$ and  $\left \langle \begin{bsmallmatrix}
            \Ab' & \zb \\
            \zb & \Diag(\bold{\hat{a}_i})
        \end{bsmallmatrix}, \bold{\hat{X}} \right \rangle = b_i$  for all  $i \in [m]. $ 
    Therefore, $\bold{\hat{X}}$ is optimal to the SDP dual to $\max_{\xb}\{\bold{b}^\top \xb \, : \, \, \xb \in P\}$. By construction, $\bold{\hat{X}}$ is an integer conical combination of matrices in the set $\mathcal{V} = \{ \zb \oplus \Diag(\bold{e_i}) \, : \, \, i \in [\ell']\}$ of integer PSD matrices. We conclude that the LMI given in~\eqref{Eq:TDI_halfpolyhedral} is totally dual integral on $Z$.
\end{proof}
Our final condition for total dual integrality on a set $Z$ is not related to the polyhedrality of the spectrahedron induced by the linear matrix inequality, but related to the feasible set of its corresponding dual problem to be polyhedral. It is possible for a spectrahedron to be non-polyhedral, while the feasible set of its dual problem is polyhedral. For instance, consider the non-polyhedral spectrahedron $Q = \{\xb \in \mathbb{R}^2 \, : \, \, x_2 \geq x_1^2/2 \}$ considered in Example~\ref{Example:spectrahedra}. For any $\bold{b} \in \mathbb{Z}^2_{-}$, its dual feasible set is given by 
$\left\{ \begin{bsmallmatrix}
        x_1 & x_2 \\ x_2 & x_3
    \end{bsmallmatrix} \in \mathcal{S}^2 \, : \, \, x_1 \geq -{b_1^2}/{2b_2},~ x_2 = -{1}/{2}b_1,~x_3 = -{1}/{2}b_2 \right\},$
which is polyhedral. Let us formalize the criterion of polyhedrality of the dual feasible set.

\begin{definition} \label{Def:FDG}
    The set $\{\bold{A_1}, \ldots, \bold{A_m}\}$ is called finitely generative on  $Z \subseteq \mathbb{Z}^m$ if there exists a finite set of integer PSD matrices $\mathcal{V} = \{\bold{V_1}, \ldots, \bold{V_k}\}$ such that
    $    \left\{ \bold{X} \, : \, \, \langle \bold{A_i}, \bold{X} \rangle = b_i,~i \in [m],~\bold{X} \succeq \bold{0} \right\}    $
    is contained in $\cone (\mathcal{V})$ for all integer vectors $\bold{b} \in Z$. 
\end{definition}
The condition of the dual feasible set to be polyhedral is also considered in recent works on SDP exactness \cite{WangKilincKarzan}.
Observe that if $\{\bold{A_1}, \ldots, \bold{A_m}\}$ is finitely generative, then $\{\bold{X} \, : \, \,  \langle \bold{A_i}, \bold{X} \rangle = b_i,~i \in [m],~\bold{X} \succeq \bold{0}\}$ is polyhedral for all $\bold{b} \in Z$ (since $\cone(\mathcal{V}) \subseteq \mathcal{S}^n_+$). Moreover, if $\{\bold{A_1}, \ldots, \bold{A_m}\}$ is finitely  generative on $Z$, then $\{t\bold{A_1}, \ldots, t\bold{A_m}\}$ is also finitely generative on $Z$ for any scalar $t > 0$. 

As shown below, the constraint matrices being finitely generative and integer is a sufficient condition for the existence of a totally dual integral description of the spectrahedron. 

\begin{theorem} \label{Thm:TDI_FDG}
Let $\Cb - \sum_{i = 1}^m \Ab x_i \succeq \zb$ be an LMI satisfying Slater's condition with $\{\bold{A_1}, \ldots, \bold{A_m}\} \subseteq \mathbb{Z}^{n \times n}$ finitely generative on $Z$. Then, the spectrahedron $P = \left\{ \bold{x} \in \mathbb{R}^m \, : \, \, \bold{C} - \sum_{i = 1}^m \bold{A_i}x_i \succeq \bold{0} \right\}$ can be described by a linear matrix inequality that is totally dual integral on $Z$. 
\end{theorem}
\begin{proof}
Let $\mathcal{V} = \{\bold{V_1}, \ldots, \bold{V_k}\}$ denote the finite set of integer PSD matrices corresponding to $\{\bold{A_1}, \ldots, \bold{A_m}\}$ in Definition~\ref{Def:FDG}. Let $\bold{b} \in Z$ and let $t > 0$ be a positive rational number. We consider the following semidefinite program and its dual: 
    \begin{align}
        & \sup_{\bold{x}} \left\{ \bold{b}^\top \bold{x} \, : \, \, t\bold{C} - \sum_{i =1}^m t\bold{A_i}x_i \succeq \bold{0} \right\} =  \inf_{\bold{X}} \left\{\langle t\bold{C}, \bold{X} \rangle \, : \, \, \langle t\bold{A_i}, \bold{X} \rangle = b_i,~i \in [m],~ \bold{X} \succeq \bold{0} \right\}. \label{dual_FG_TDI}
    \end{align}
Based on the fact that $\{\bold{A_1}, \ldots, \bold{A_m}\}$ is finitely generative, we know that the feasible set of the minimization problem in~\eqref{dual_FG_TDI} is contained in $\cone (\mathcal{V})$. Since we also know the minimum is attained due to Slater's condition, we can rewrite the dual problem as follows: 
\begin{align*}
    & \min_{\bold{X}} \Big\{\langle t\bold{C}, \bold{X} \rangle \, : \, \, \langle t\bold{A_i}, \bold{X} \rangle = b_i,~i \in [m],~ \bold{X} = \alpha_1 \bold{V}_1 + \dots + \alpha_k \bold{V_k},~ \mathbold{\alpha}\geq \bold{0} \Big\} \\
    = & \min_{\bold{X}} \left\{ \langle t \bold{C}, \bold{X} \rangle \, : \begin{aligned} \, \, t~\triu (\bold{A_i})^\top \svec(\bold{X}) = b_i,~~ i \in [m],~~\mathbold{\alpha} \geq \bold{0}  \\
    ~\svec(\bold{X}) - \alpha_1 \svec (\bold{V_1}) - \dots - \alpha_k \svec(\bold{V_k})= 0 \end{aligned}  \right\} \\
    = & \min_{\bold{X}} \left\{ \langle t \bold{C}, \bold{X} \rangle \, : \, \, \begin{bmatrix} t\bold{A'} & \bold{0} \\
    \bold{I} & - \bold{V'}
    \end{bmatrix} \begin{bmatrix}
        \svec(\bold{X}) \\ 
        \mathbold{\alpha}
    \end{bmatrix} = \begin{bmatrix}
        \bold{b} \\ \bold{0}
    \end{bmatrix} ,~ \mathbold{\alpha}\geq \bold{0} \right\},
\end{align*}
where $\bold{A'} := \begin{bmatrix}
    \triu (\bold{A_1}) & \dots & \triu (\bold{A_m})
\end{bmatrix}^\top$, $\bold{V'} := \begin{bmatrix}
    \svec (\bold{V_1}) & \dots & \svec (\bold{V_k})
\end{bmatrix}$, $\triu: \mathcal{S}^n \rightarrow \mathbb{R}^{\frac{1}{2}(n^2 + n)}$ is the operator that maps a matrix to a vector containing its upper-triangular entries and $\svec : \mathcal{S}^n \rightarrow \mathbb{R}^{\frac{1}{2}(n^2 + n)}$ is the symmetric vectorization operator that maps a matrix to a column vector containing its upper-triangular part with weight two on the off-diagonal elements and weight one on the diagonal elements.
The linear system in the dual problem above can be written as
\begin{align*}
    t\begin{bmatrix} \bold{A'} & \bold{0} \\
    \bold{I} & - \bold{V'}
    \end{bmatrix} \begin{bmatrix}
        \svec(\bold{X}) \\ 
        \mathbold{\alpha}
    \end{bmatrix} = \begin{bmatrix}
        \bold{b} \\ \bold{0}
    \end{bmatrix}, \quad \text{or equivalently,} \quad \begin{bmatrix} \bold{A'} & \bold{0} \\
    \bold{I} & - \bold{V'}
    \end{bmatrix} \begin{bmatrix}
        \svec(\bold{X}) \\ 
        \mathbold{\alpha}
    \end{bmatrix} = \frac{1}{t}\begin{bmatrix}
        \bold{b} \\ \bold{0}
    \end{bmatrix}. 
\end{align*}
Each basic feasible solution to this system with $\mathbold{\alpha} \geq \bold{0}$ is the unique solution to one of its non-singular subsystems. Following the proof by Giles and Pulleyblank~\cite{GilesPulleyblank}, it is possible to find a rational number $t^*$ such that for all $\bold{b} \in Z$, there exists an optimal solution that satisfies $\svec (\bold{X}) \in \mathbb{Z}^{\frac{1}{2}(n^2 +n)}$ and $\mathbold{\alpha} \in \mathbb{Z}^k$. 
When mapping $\svec(\bold{X})$ back to $\bold{X} \in \mathcal{S}^n$, it follows that the SDP dual to $ \max\{\bold{b}^\top \bold{x} \, : \, \, t^* \bold{C} - \sum_{i = 1}^m t^* \bold{A_i}x_i \succeq \bold{0} \}$ for all $\bold{b} \in Z$ has an optimal solution $\bold{X}$ satisfying 
$ \bold{X} = \sum_{j \in [k]} \alpha_j \bold{V_j},~ \alpha_j \geq 0,~j \in [k]. $
with $\mathbold{\alpha}$ integer. Hence, property \ref{PZ} holds for $\bold{X}$. We conclude that $ t^* \bold{C} - \sum_{i = 1}^m t^* \bold{A_i}x_i \succeq \bold{0}$ is a linear matrix inequality describing $P$ that is totally dual integral on $Z$.
\end{proof}
}

\subsection{Strengthened Chv\'atal-Gomory cuts} \label{Section:SCGcuts}
Dash et al.\ \cite{DashEtAl} consider a strengthening of the CG {cuts} for rational polyhedra. We briefly present here their approach that can be applied to general convex sets.

For all $\bold{c}\in \mathbb{Z}^m$ such that $P \subseteq \{\bold{x} \in \mathbb{R}^m \, : \, \, \bold{c}^\top\bold{x} \leq d\}$, the corresponding CG cut is $\bold{c}^\top \bold{x} \leq \lfloor d \rfloor$. The validity of this cut follows from the inequality
$
\lfloor d \rfloor \geq \max \left\{ \bold{c}^\top \bold{x} \, : \, \, \bold{c}^\top \bold{x} \leq d, \, \bold{x} \in \mathbb{Z}^m \right\},
$
where equality holds if the entries in $\bold{c}$ are relatively prime. However, the gap between $\lfloor d \rfloor$ and $\max\{ \bold{c}^\top \bold{x} \, : \, \, \bold{x} \in P \cap \mathbb{Z}^m\}$ can generally be very large. In order to reduce this gap, suppose that we know that $P \cap \mathbb{Z}^m$ is contained in some set $S \subseteq \mathbb{Z}^m$. Given a valid inequality $\bold{c}^\top \bold{x} \leq d$ for $P$, we define
\begin{align} \label{eq:strongfloor}
\lfloor d \rfloor_{S,c} := \max \left\{\bold{c}^\top \bold{x} \, : \, \, \bold{c}^\top \bold{x} \leq d, \, \bold{x} \in S \right\}.
\end{align}
By construction, $\bold{c}^\top \bold{x} \leq \lfloor d \rfloor_{S,c}$ is valid for $P \cap \mathbb{Z}^m$. We refer to these type of cuts as $S$-Chv\'atal-Gomory ($S$-CG) cuts. These cuts are at least as strong as standard CG cuts, since taking $S = \mathbb{Z}^m$ provides the standard CG cut. The geometric interpretation of an $S$-CG cut is that we shift the hyperplane $\{\bold{x} \in \mathbb{R}^m \, : \, \, \bold{c}^\top \bold{x} = d\}$ in the direction of $P \cap \mathbb{Z}^m$ until it hits a point in $S$. An example for $S$ is the set $\{0,1\}^{m}$ in the case of binary optimization problems.

\section{A CG-based branch-and-cut algorithm for ISDPs} \label{Section:B&C}

Solving ISDPs is a relatively new field of research for which only a few general-purpose solution approaches have been proposed. Gally et al.~\cite{GallyEtAl} present a B\&B algorithm called SCIP-SDP for solving (M)ISDPs  with continuous SDPs as subproblems. Alternatively, Kobayashi and Takano~\cite{KobayashiTakano} propose a B\&C algorithm that initially relaxes the PSD constraint and solves a mixed integer linear program (MILP), where the PSD constraint is imposed dynamically via cutting planes. 
Numerical results in \cite{KobayashiTakano} show that the B\&C algorithm of \cite{KobayashiTakano} outperforms the B\&B algorithm of \cite{GallyEtAl}.
The difference can be explained by the high performance of the current MILP solvers compared to the much less robust conic interior point methods that are used in \cite{GallyEtAl}. It has to be noted, however, that an older version of SCIP-SDP with DSDP \cite{BensonEtAl} as SDP solver was used in the computational results of \cite{KobayashiTakano}. {The authors of~\cite{MatterPfetsch} also compare the two approaches and conclude that SCIP-SDP  is much faster on average   than the approach by Kobayashi and Takano. However, they use  Mosek~\cite{Mosek} as an SDP solver and  an improved implementation of  SCIP-SDP.}
Another project that encounters MISDPs is YALMIP~\cite{Lofberg}, although its performance is inferior compared to the other two methods \cite{GallyEtAl,KobayashiTakano}.

In this section we present a generic B\&C algorithm for solving ISDPs that exploits CG cuts of the underlying spectrahedron. This algorithm can be seen as an extension of the works of \cite{KobayashiTakano, CezikIyengar}.
In Section~\ref{Subsection:B&C} we provide a general B\&C framework for ISDPs which uses a cut generation routine based on $S$-CG cuts. Section \ref{SeparationBinarySDP} presents a separation routine for the special class of binary SDPs. 

\subsection{Generic Branch-and-Cut framework} \label{Subsection:B&C}
We start this section by presenting the B\&C framework proposed by Kobayashi and Takano~\cite{KobayashiTakano}
for ISDPs in standard dual form, see \eqref{dualISDP}. However, the approach can be extended to problems in primal form in a straightforward way.
We define
\begin{align} \label{SetF}
\mathcal{F} := \left\{ \bold{x} \in \mathbb{R}^m \, : \, \, \diag\left(\bold{C} - \sum_{i = 1}^m \bold{A_i}x_i \right) \geq 0 \right\},
\end{align}
which can be seen as the polyhedral part of the spectrahedron $P$, see \eqref{def:P}. We assume that the problem of maximizing $\bold{b}^\top \bold{x}$ over $\mathcal{F}$ is bounded, which is a non-restrictive assumption whenever the original ISDP is bounded.

The B\&C  algorithm of \cite{KobayashiTakano} is based on a dynamic constraint generation known as a lazy constraint callback.
The algorithm starts with optimizing over the set $\mathcal{F}\cap \mathbb{Z}^m $, i.e.,
\begin{align} \label{RootNode}
\max \left\{ \bold{b}^\top \bold{x} \, : \, \, \bold{x} \in \mathcal{F} \cap \mathbb{Z}^m  \right\},
\end{align}
which can be solved using a B\&B algorithm. Whenever an integer point $\hat{\bold{x}}$ is found in the branching tree, it is verified whether $\bold{C} - \sum_{i = 1}^m\bold{A_i}\hat{x}_i \succeq \mathbf{0}$ is satisfied.
If so, the solution is feasible for $(D_{ISDP})$ and provides a possibly better lower bound to prune other nodes in the tree.
If not, then $\langle \bold{C} - \sum_{i  =1}^m \bold{A_i}\hat{x}_i, \bold{dd}^\top \rangle < 0$ where $\bold{d}$ is a normalized eigenvector corresponding to the smallest eigenvalue of $\bold{C} - \sum_{i  =1}^m \bold{A_i}\hat{x}_i$.
This leads to the following valid constraint for $(D_{ISDP})$:
\begin{align} \label{psdcut:standard}
\left \langle \bold{C} - \sum_{i  =1}^m \bold{A_i}x_i, \bold{dd}^\top \right \rangle \geq 0, \quad  \text{or equivalently,} \quad  \sum_{i=1}^m \langle \bold{A_i}, \bold{dd}^\top \rangle x_i \leq \langle \bold{C}, \bold{dd}^\top \rangle ,
\end{align}
which separates $\bold{\hat{x}}$ from $P$.
Now, the algorithm adds to $\mathcal{F}$ a cut of type \eqref{psdcut:standard} to cut off the current point and continues the branching scheme using this additional constraint. This process is iterated until the optimality of a solution for $(D_{ISDP})$ is guaranteed by the B\&B procedure.

It follows from the Rayleigh principle that $\langle \bold{C} - \sum_{i = 1}^m \bold{A_i} \hat{x}_i , \bold{U} \rangle$ is minimized by taking $\bold{U} = \bold{dd}^\top$ with $\bold{d}$ as defined above. In that sense, the cut \eqref{psdcut:standard} is the strongest cut with respect to violation in the PSD constraint. However, this type of separator ignores the fact that an optimal solution is also integer. 
We now propose an alternative stronger separator based on the CG procedure that exploits both the PSD and the integrality constraint.

Let $S \subseteq \mathbb{Z}^m$ be a set containing the feasible set of $(D_{ISDP})$, {with $S = \mathbb{Z}^m$ in case of no prior knowledge about the problem. If $\bold{\hat{x}} \notin P$, and consequently $\bold{\hat{x}} \notin \cl(P)$, it follows from \eqref{eq:coniccuts} that} there exists a dual multiplier $\bold{U} \in \mathcal{S}^n_+$ with $\langle \bold{A_i}, \bold{U} \rangle \in \mathbb{Z}$  for all $i \in [m]$, such that $\sum_{i = 1}^m \langle \bold{A_i}, \bold{U} \rangle \hat{x}_i > \lfloor \langle \bold{C}, \bold{U} \rangle \rfloor$. 
Taking such $\bold{U}$ and defining $\bold{v(U)} := \left( \langle \bold{A_1}, \bold{U} \rangle, \ldots, \langle \bold{A_m}, \bold{U} \rangle \right)^\top$, we obtain the following $S$-CG cut:
\begin{align}
\sum_{i = 1}^m \langle \bold{A_i}, \bold{U} \rangle x_i & \leq \lfloor \langle \bold{C}, \bold{U} \rangle \rfloor_{S,\bold{v(U)}}, \label{psdcut:strong}
\end{align}
\noindent see \eqref{eq:strongfloor}. The cut \eqref{psdcut:strong} exploits both the PSD and the integrality constraints in $(D_{ISDP})$ by separating $\bold{\hat{x}}$ from $\cl(P)$ instead of only from $P$. As $\cl(P) \subseteq P$ for bounded spectrahedra, this type of cut is possibly stronger than the eigenvalue cut \eqref{psdcut:standard} for all $S$ containing $P \cap \mathbb{Z}^m$. Figure~\ref{fig:B&Cexample} depicts a simplified example indicating the geometric difference between the cuts \eqref{psdcut:standard} and \eqref{psdcut:strong}.

\begin{wrapfigure}{r}{7.3cm}
\centering
\vspace{-0.3cm}
\includegraphics[scale=0.7]{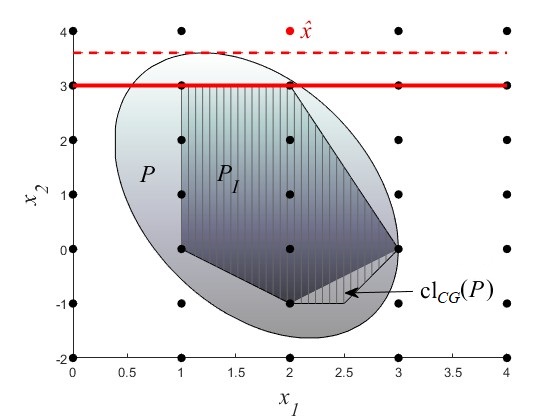}
\caption{Simplified example of strengthened separation routine on spectrahedron $P$ from Example~\ref{Example:spectrahedra}. 
The dotted line shows an eigenvalue cut \eqref{psdcut:standard} separating $\bold{\hat{x}}$ from $P$, the solid line shows a CG cut \eqref{psdcut:strong} separating $\bold{\hat{x}}$ from $\cl(P)$, where $S = \mathbb{Z}^m$. \label{fig:B&Cexample}}
\vspace{-0.5cm}
\end{wrapfigure}

It is not clear in general how to find an appropriate {cut~\eqref{psdcut:strong} separating $\bold{\hat{x}}$ from $\cl(P)$}. 
Indeed, this is closely related to the CG separation problem, which was proven to be $\mathcal{NP}$-hard even for polytopes contained in the unit hypercube, see  Cornu\'{e}jols et al.~\cite{CornuejolsLeeLi}.
Fischetti and Lodi~\cite{FischettiLodi} show how to solve the separation problem for polyhedra using a mixed integer programming problem. Extending their procedure to the class of spectrahedra, implies solving a MISDP. 
Instead, we can adopt problem-specific separation routines that are efficient and provide strong cuts. {For instance, in the next subsection we present a separation routine for binary SDPs in primal form.} Moreover, we later provide various separation routines for cuts of the form \eqref{psdcut:strong} for the quadratic traveling salesman problem.

Alongside extending the approach of Kobayashi and Takano \cite{KobayashiTakano}, our framework also continues on the work of \c{C}ezik and Iyengar \cite{CezikIyengar}. In \cite{CezikIyengar} CG cuts for binary conic programs are introduced. It is noted that there is no method known for separating CG cuts from fractional points, and consequently the CG cuts are not included in the numerical experiments of \cite{CezikIyengar}. Since our approach separates on integer points only, we partly resolve this issue for certain classes of problems by exploiting the underlying structure of the  programs. As a result, we present the first practical algorithm that utilizes CG cuts in conic problems.

We end this section by providing a pseudocode of the B\&C framework, see Algorithm~\ref{AlgB&C}. Suppose \textsc{SeparationRoutine} 
is a separation routine for constructing CG cuts of the form \eqref{psdcut:strong}, where we assume this routine can generate multiple dual matrices at a time.

\begin{algorithm}[H]
\footnotesize
	\SetAlgoLined
	\KwIn{$\bold{C}, \bold{A_i}$, $i \in [m]$ , $S$, $\epsilon > 0$,~~ \textbf{Output:} $\bold{\hat{x}}, OPT := \bold{b}^\top \bold{x}$}
	Initialize $\mathcal{F}$ as defined in (\ref{SetF}).\\
	\textbf{B\&B procedure:} Start or continue the branch-and-bound algorithm for solving the MILP
$\max \left\{ \bold{b}^\top \bold{x} \, : \, \, \bold{x} \in \mathcal{F} \cap \mathbb{Z}^m \right\}$ incorporating the callback function below at each node in the branching tree. \\[0.5em]
	\textbf{Callback procedure:} \If{\text{an integer point $\bold{\hat{x}} \in \mathcal{F}$ is found}}
	    {\eIf{$\lambda_{\min} \left( \bold{C} - \sum_{i = 1}^m \bold{A_i}\hat{x}_i \right) < - \epsilon$}
	    {Call \textsc{SeparationRoutine}$(\bold{C},\bold{A_1},\ldots,\bold{A_m},S,\bold{\hat{x}})$ which provides matrices $\bold{U_j}$, $j \in [K]$. Add the cuts $\sum_{i = 1}^m \langle \bold{A_i}, \bold{U_j} \rangle x_i  \leq \lfloor \langle \bold{C}, \bold{U_j} \rangle \rfloor_{S,\bold{v(U_j)}}$ for $j \in [K]$ to $\mathcal{F}$.} 
{Use $\bold{\hat{x}}$ to cut off other nodes in the branching tree.}
{\textsc{Return to Step 2}}
	         }

	\caption{CG-based B\&C algorithm for solving $(D_{ISDP})$}\label{AlgB&C}
\end{algorithm}

\subsection{A separation routine for binary SDPs} \label{SeparationBinarySDP}
We now focus on binary semidefinite programming problems in primal form, i.e.,
\begin{align}
\tag{$P_{BSDP}$} & \left \{ \begin{aligned} \inf \quad &  \langle \bold{C}, \bold{X} \rangle \\
\text{s.t.} \quad & \langle \bold{A_i} , \bold{X} \rangle = b_i \,\, \mbox{ for all } i \in [m] \\
&  \bold{X} \succeq \mathbf{0}, \,\, \bold{X} \in \{0,1\}^{n \times n}.  \end{aligned} \right. \label{primalBSDP}
\end{align}
In this section we present a separation routine for generating CG cuts for problems of the form \eqref{primalBSDP} and provide {an illustrative example}. To do so, we use the following characterization of binary PSD matrices. 

\begin{proposition}[Letchford and S{\o}rensen \cite{LetchfordSorensen}]
Let $\bold{X} \in \{0,1\}^{n \times n}$ be a symmetric matrix. Then $\bold{X} \succeq \mathbf{0}$ if and only if $\bold{X} = \sum_{i = 1}^k \bold{x_ix_i}^\top$ for some $\bold{x_i} \in \{0,1\}^n$, $i \in [k]$. \label{Prop:binPSD}
\end{proposition}

Each vector $\bold{x_i}$ in Proposition~\ref{Prop:binPSD} may be thought of as the characteristic vector of a clique in the complete graph $K_n$. 
Therefore, $\bold{X}$ provides a decomposition 
of $K_n$ into a set of non-overlapping cliques.

Suppose we solve \eqref{primalBSDP} using the B\&C algorithm presented in Section~\ref{Subsection:B&C}. In a certain node in the branching tree we have obtained a symmetric matrix $\bold{\hat{X}} \in \{0,1\}^{n\times n}$ that satisfies $\langle \bold{A_i}, \bold{\hat{X}} \rangle = b_i$ for all $i \in [m]$. The separation oracle that we present below distinguishes two types of certificates for $\bold{\hat{X}}$ not being positive semidefinite. The first one is obtained by a so-called dominated diagonal, i.e., $\hat{X}_{ii} = 0$, while $\hat{X}_{ij} = 1$ for some $j$, which clearly implies that $\bold{\hat{X}} \nsucceq \mathbf{0}$. {The second certificate is the presence of a so-called conflicting vertex, i.e., a vertex that is contained in two separate cliques implied by $\bold{\hat{X}}$.} By Proposition~\ref{Prop:binPSD}, it follows that $\bold{\hat{X}} \nsucceq \mathbf{0}$. These certificates correspond to the existence of the following induced submatrices in $\bold{\hat{X}}$ (up to a permutation of the rows and columns):
$$
\small
\bbordermatrix{
& i & j \cr
i & 0 & 1 \cr 
j & 1 & \star
}
\quad \text{and} \quad
\bbordermatrix{
 & i & j & k \cr
i & 1 & 1 & 1 \cr
j & 1 & 1 & 0 \cr
k & 1 & 0 & 1
},
$$
where $\star$ indicates a position that can be either 0 or 1. The following result shows that these certificates are necessary and sufficient to characterize positive semidefiniteness.
\begin{proposition} Let $\bold{\hat{X}} = ({\hat{x}_{ij}})$ be binary and symmetric. Then, $\bold{\hat{X}}$ is positive semidefinite if and only if $\bold{\hat{X}}$ contains no dominated diagonal or conflicting vertex. 
\end{proposition}
\begin{proof}
Necessity follows from the discussion above. To prove {sufficiency}, let $D(i) := \{j \in [n] \, : \, \, \hat{x}_{ij} = 1 \}$ for all $i \in [n]$ with $\hat{x}_{ii} = 1$. If $\hat{x}_{ij} = 1$ and $\hat{x}_{ik} = 1$, it must follow that $\hat{x}_{jk} = 1$, otherwise $i$ would be conflicting. Hence, the sets $D(i)$ for all $i$ with $\hat{X}_{ii} = 1$ are cliques. Since $i \in D(j)$ if and only if $j \in D(i)$, {it follows that the collection $\mathcal{D}$ of all distinct sets $D(i)$ is a set of non-overlapping cliques.}
Then, $\bold{\hat{X}} = \sum_{D \in \mathcal{D}}\bold{\mathbbm{1}_D \mathbbm{1}_D}^\top$, hence $\bold{\hat{X}} \succeq \zb$  by Proposition~\ref{Prop:binPSD}. 
\end{proof}

In case of a dominated diagonal, i.e., indices $i,j \in [n], i \neq j$ with $\hat{x}_{ii} = 0$ and $\hat{x}_{ij} = 1$, the dual matrix $\bold{U} = (\mathbf{e}_i - \mathbf{e}_j)(\mathbf{e}_i - \mathbf{e}_j)^\top$ separates $\bold{\hat{X}}$ from $\mathbb{S}^n_+$. In case of a conflicting vertex, say $i$, with $\hat{x}_{ij} = 1$, $\hat{x}_{ik} = 1$, but $\hat{x}_{jk} = 0$, the dual matrix $\bold{U} = (\mathbf{e}_j + \mathbf{e}_k - \mathbf{e}_i)(\mathbf{e}_j + \mathbf{e}_k - \mathbf{e}_i)^\top$ provides a separating hyperplane. Since dominated diagonals and conflicting vertices can be found efficiently by enumeration, this approach defines an efficient separation routine for binary SDPs in primal form.

The cuts $\langle \bold{U}, \bold{X} \rangle \geq 0$ can be further strengthened by exploiting the affine constraints in a CG rounding step. We show how this can be done for a class of binary semidefinite programming problems that often appears in relaxations of combinatorial problems.

\begin{example}[Binary SDPs over the simplex]
Many combinatorial optimization problems have formulations including a constraint on the trace of the matrix variable, i.e.,
\begin{align}
\tag{$P_2$} \left \{ \begin{aligned} \inf \quad &  \langle \bold{C}, \bold{X} \rangle \\
\text{s.t.} \quad & \langle \bold{A_i} , \bold{X} \rangle = b_i \,\, \text{for all } i \in [m] \\
& \tr(\bold{X}) = K, 
\,\,  \bold{X} \succeq \mathbf{0}, \,\, \bold{X} \in \{0,1\}^{n \times n}, \end{aligned} \right. \label{exampleP2}
\end{align}
for some $K \in \mathbb{N}$. One can solve \eqref{exampleP2} using Algorithm~\ref{AlgB&C} with $\mathcal{F} := \{ \bold{X} \in \mathcal{S}^n \, : \, \, \langle \bold{A_i}, \bold{X} \rangle = b_i, i \in [m], \, \tr(\bold{X}) = K, \, \mathbf{0} \leq \bold{X} \leq \mathbf{J} \}$.
Assume that the separation routine provides a dual matrix $\bold{U} = (\mathbf{e}_j + \mathbf{e}_k - \mathbf{e}_i)(\mathbf{e}_j + \mathbf{e}_k - \mathbf{e}_i)^\top$ for some distinct $i, j, k$. Taking the linear combination of $\langle \bold{U}, \bold{X} \rangle \geq 0$, $\tr(\bold{X}) = K$ and $x_{ll} \geq 0$ for all $l \notin \{i, j, k\}$, each with weight $\frac{1}{2}$, yields:
\begin{align*}
\left\langle \frac{1}{2}\bold{U} +  \frac{1}{2} \bold{I} + \frac{1}{2} \sum_{l \notin \{i,j,k\}}\bold{E_{ll}}
, \, \bold{X} \right\rangle & \geq \frac{1}{2}K.
\end{align*}
For $K$ odd, we can strengthen the cut by replacing the right-hand side by $\lceil \frac{1}{2}K \rceil$. This procedure can be repeated for dual matrices resulting from a dominated diagonal certificate.
\end{example}

\section{The Chv\'atal-Gomory procedure for ISDP formulations of the QTSP} \label{Section:QTSP}
In this section we provide an in-depth study on solving the Quadratic Traveling Salesman Problem using our B\&C approach. We formally define the \textsc{QTSP} in  Section~\ref{Section:DefQTSP}.
In Section~\ref{Section:SDPforQTSP} we derive two ISDP formulations of the \textsc{QTSP}.
Our first ISDP model exploits the algebraic connectivity of a directed tour.
Our second formulation exploits  the algebraic connectivity of a directed tour and the distance two matrix that originates from the product of a tour matrix with itself.
Finally, in Section~\ref{Section:CutsforQTSP} we derive CG cuts for the two ISDPs and show that we can obtain various classes of well-known cuts in this way.

\subsection{The Quadratic Traveling Salesman Problem} \label{Section:DefQTSP}

Let $G = (N, A)$ be a directed simple graph on $n := |N|$ nodes and $m := |A|$ arcs. A directed cycle $\mathcal{C}$ in $G$ that visits all the nodes exactly once is called a directed Hamiltonian cycle or a directed tour in $G$. For the sake of simplicity, we often omit the adjective `directed' in the sequel.

A tour in $G$ can be represented by a binary matrix $\bold{X} = (x_{ij}) \in \{0,1\}^{n \times n}$ such that $x_{ij} = 1$ if and only if arc $(i,j)$ is used in the tour. We refer to such a matrix as a \emph{tour matrix}. The set of all tour matrices in $G$ is defined as follows:
\begin{align} \label{def:Tau}
\mathcal{T}_n(G) := \left\{ \bold{X}^\mathcal{C} \in \{0,1\}^{n \times n} \, : \, \, x^\mathcal{C}_{ij} = 1 \text{ if and only if } (i,j) \in \mathcal{C} \text{ for Hamiltonian cycle } \mathcal{C} \right\}.
\end{align}
It follows from \eqref{def:Tau} that for all $\bold{X} \in \mathcal{T}_n(G)$ we have $x_{ij} = 0$ if $(i,j) \notin A$. In particular, $\diag(\bold{X}) = \mathbf{0}_n$.
Given a distance matrix $\bold{D} = (d_{ij}) \in \mathbb{R}^{n \times n}$, the (linear) traveling salesman problem (\textsc{TSP}) is the problem of finding a Hamiltonian cycle $\mathcal{C}$ of $G$ that minimizes $ \sum_{(i,j) \in \mathcal{C}}d_{ij}$. As $G$ is directed and $\bold{D}$ is not necessarily symmetric, this version of the problem is sometimes referred to as the asymmetric traveling salesman problem. Using the set defined in \eqref{def:Tau}, we can state the \textsc{TSP} as follows:
\begin{align} \label{def:TSP}
TSP(\bold{D},G) := \min \left \{ \sum_{i = 1}^n \sum_{j = 1}^n d_{ij}x_{ij} \, : \, \, \bold{X} \in \mathcal{T}_n(G) \right\}.
\end{align}
We now define the quadratic version of the \textsc{TSP}, where the total cost is given by the sum of interaction costs between arcs used in the tour. In accordance with most of the literature, we assume that a quadratic cost is incurred only if two arcs are placed in succession on the tour, see e.g.,~\cite{Fischer, Fischer2014,AandFFischer,Jager,RostamiEtAl}. To model this problem, we define the set of the so-called 2-arcs of $G$, i.e.,
\begin{align} \label{def:twoarcs}
\mathcal{A} := \left\{ (i,j,k) \, : \, \, (i,j), (j,k) \in A, |\{i,j,k\}| = 3 \right\},
\end{align}
which consists of all node triples of $G$ that can be placed in succession on a cycle. 
Now let $\bold{Q} = (q_{ijk}) \in \mathbb{R}^{n \times n \times n}$ be a cost matrix such that $q_{ijk}= 0$ if $(i,j,k) \notin \mathcal{A}$. Then the quadratic traveling salesman problem (\textsc{QTSP}) is formulated as:
\begin{align}\label{def:QTSP}
QTSP(\bold{Q},G) := \min \left\{ \sum_{i = 1}^n \sum_{j = 1}^n \sum_{k = 1}^n q_{ijk}x_{ij}x_{jk} \, : \, \, \bold{X} \in \mathcal{T}_n(G) \right\}.
\end{align}
Since the in- and outdegree of each node on a Hamiltonian cycle is exactly one, we have $\bold{X}\mathbf{1} = \mathbf{1}$ and $\bold{X}^\top \mathbf{1} = \mathbf{1}$ for all $\bold{X} \in \mathcal{T}_n(G)$. The set of square binary matrices that satisfy this property is known as the set of permutation matrices $\Pi_n$, i.e.,
$\Pi_n := \left \{\bold{X} \in \{0,1\}^{n \times n} \, : \, \, \bold{X}\bold{1} = \bold{1}, \, \bold{X}^\top\bold{1} = \bold{1} \right\}.$
The permutation matrices that additionally satisfy $\diag(\bold{X}) = \mathbf{0}_n$ induce a disjoint cycle cover in $K_n$. 

Similar to the definition of $\mathcal{T}_n(G)$, we can also restrict $\Pi_n$ to the entries induced by $G$. That is, $\Pi_n(G)$ has a zero on position $(i,j)$ whenever $(i,j) \notin A$.

\subsection{ISDP based on algebraic connectivity in directed graphs}
\label{Section:SDPforQTSP}
Cvetkovi\'c et al.~\cite{Cvetkovic} derive an ISDP formulation of the symmetric linear \textsc{TSP} based on algebraic connectivity. We now exploit the equivalent of this notion for directed graphs to derive two ISDP formulations of the \textsc{QTSP}. Different from our approach, there was no attempt in \cite{Cvetkovic} to solve the ISDP itself, only its SDP relaxation.

Let $\bold{D_G}$ be an $n \times n$ diagonal matrix that contains the outdegrees of the nodes of $G$ on the diagonal. Moreover, let $\bold{A_G}$ denote the adjacency matrix of $G$.
That is, $(A_G)_{ij} = 1$ if there exists an arc from $i$ to $j$ in $G$, and $(A_G)_{ij} = 0$ otherwise. We define the directed out-degree Laplacian matrix of $G$ as $\bold{L_G} := \bold{D_G} - \bold{A_G}$. The matrix $\bold{L_G}$ can be asymmetric and has a zero eigenvalue with corresponding eigenvector $\mathbf{1}_n$.
Observe that there exist also other ways for defining the directed graph Laplacian of $G$, see e.g.,~\cite{CaughmanVeerman}.
Wu~\cite{Wu} generalized Fiedler's notion of algebraic connectivity of an undirected graph~\cite{Fiedler} to directed graphs, by exploiting the out-degree Laplacian matrix.
\begin{definition} \label{def:algconnect}
The algebraic connectivity of a directed graph $G$ is given by
\begin{align*}
a(G) := \min_{\bold{x} \in S} \bold{x}^\top \bold{L_G} \bold{x} = \min_{\substack{\bold{x} \in \mathbb{R}^n \\ \bold{x} \neq \mathbf{0}, \bold{x} \perp \mathbf{1}_n}} \frac{\bold{x}^\top \bold{L_G} \bold{x}}{\bold{x}^\top \bold{x}} = \lambda_{\min} \left(\frac{1}{2} \bold{W}^\top \left( \bold{L_G} + \bold{L_G}^\top \right) \bold{W} \right),
\end{align*}
where
$S := \left\{\bold{x} \in \mathbb{R}^n \, : \, \, \bold{x} \perp \bold{1}_n \, , \, \, \Vert \bold{x} \Vert_2 = 1 \right\}$
and
$\bold{W} \in \mathbb{R}^{n \times (n-1)}$ is a matrix whose columns form an orthonormal basis for $\mathbf{1}_n^\perp$.
\end{definition}
The last equality in Definition~\ref{def:algconnect} follows from the Courant-Fischer theorem.
Observe that $a(G)$ is not necessarily equal to the second smallest eigenvalue of the directed Laplacian matrix, which is the definition of its undirected counterpart. The algebraic connectivity $a(G)$ as defined in Definition~\ref{def:algconnect} is a real number that can be negative.

A directed graph is called {\em balanced} if for each node its indegree is equal to its outdegree. Let $\bold{B} \in \{-1,0,1\}^{n \times m}$ be the signed incidence matrix of $G$, i.e., $B_{i,e} = -1$ if arc  leaves node $i$, $B_{i,e} = 1$ if $e$ enters node $i$ and $B_{i,e}=0$ otherwise.
One can verify that $G$ is balanced if and only if $\bold{L_G} + \bold{L_G}^\top = \bold{BB}^\top$. This implies that for balanced graphs the matrix $\frac{1}{2}(\bold{L_G} + \bold{L_G}^\top)$ is positive semidefinite.
Wu~\cite{Wu} observes that if $G$ is balanced, then
$a(G) = \lambda_2 ( ( \bold{L_G} + \bold{L_G}^\top )/2) \geq 0.$
A directed graph is called {\em strongly connected} if for every pair of distinct nodes $u, v \in N$ there exists a directed path from $u$ to $v$ in $G$. The balanced graphs that are strongly connected are characterized by their algebraic connectivity, see Proposition~\ref{prop:balanced} below. Connectedness of directed graphs is also studied in \cite{CaughmanVeerman, VeermanLyons}.
\begin{proposition}[Wu \cite{Wu}] \label{prop:balanced}
Let a directed graph $G$ be balanced. Then, $a(G) >0$ if and only if $G$ is strongly connected.
\end{proposition}
This characterization can be exploited to derive a certificate for a tour matrix via a linear matrix inequality. In order to do so, we consider the spectrum of a Hamiltonian cycle. Let $\mathcal{C}$ be a Hamiltonian cycle in $G$ corresponding to the tour matrix $\bold{X} \in \mathcal{T}_n(G)$, see~\eqref{def:Tau}. We then have $\frac{1}{2}\left( \bold{L_{\mathcal{C}}} + \bold{L_{\mathcal{C}}}^\top \right) = \bold{I_n} - \frac{1}{2}(\bold{X} + \bold{X}^\top )$. The matrix $\bold{X} + \bold{X}^\top$ with $\bold{X} \in \mathcal{T}_n(G)$ has the same spectrum as the adjacency matrix of the standard undirected $n$-cycle.
As a result, the spectrum of $\frac{1}{2}(\bold{X} + \bold{X}^\top)$ is given by $\cos \left( \frac{2 \pi j}{n}\right)$ for $j \in [n]$
see e.g.,~\cite{Cvetkovic}. From this, it follows that the spectrum of $\frac{1}{2}\left( \bold{L_\mathcal{C}} + \bold{L_\mathcal{C}}^\top \right)$ is given by
$1 -  \cos \left( {2 \pi j}/{n}\right)  \text{for } j \in [n],$
and the algebraic connectivity of a directed Hamiltonian cycle $\mathcal{C}$ is $a(\mathcal{C}) = 1 - \cos (2 \pi / n)$. We define:
\begin{align}\label{knhn}
k_n :=  \cos \left(\frac{2\pi}{n}\right) \quad \text{and} \quad h_n := 1 - k_n.
\end{align}
Next, we extend a result by Cvetkovi\'c et al.~\cite{Cvetkovic} from undirected to directed Hamiltonian cycles.
\begin{theorem} \label{Thm:CvetkovicDirected}
Let $H$ be a spanning subgraph of a directed graph $G$ where the in- and outdegree equals one for all nodes in $H$. Let $\bold{X}$ be its adjacency matrix and let $\alpha, \beta \in \mathbb{R}$ be such that $\alpha \geq h_n / n$ and $k_n \leq \beta < 1$, with $k_n, h_n$ as defined in \eqref{knhn}. Then, $H$ is a directed Hamiltonian cycle if and only if
$$\bold{Z} := \beta \bold{I_n} + \alpha \bold{J_n} - \frac{1}{2}\left( \bold{X} + \bold{X}^\top \right) \succeq \mathbf{0}.$$
\end{theorem}

\begin{proof}
Let $\bold{L_H}$ be the Laplacian matrix of $H$ and let $\bold{W}$ be as given in Definition~\ref{def:algconnect}. Then $a(H) = \lambda_{\min} \left( \frac{1}{2} \bold{W}^\top \left(\bold{L_H} + \bold{L_H}^\top \right) \bold{W} \right)$.
Let $\bold{Z} \succeq \mathbf{0}$. This implies that $\bold{W}^\top \bold{ZW} \succeq \mathbf{0}$, i.e.,
\begin{align*}
\bold{W}^\top \bold{Z W} & = \bold{W}^\top \left( \beta \bold{I_n} + \alpha \bold{J_n} - \frac{1}{2}\left( \bold{X} + \bold{X}^\top \right) \right) \bold{W} 
 = \beta \bold{W}^\top \bold{W} + \alpha \bold{W}^\top \bold{J_n} \bold{W} - \frac{1}{2} \bold{W}^\top \left( \bold{X} + \bold{X}^\top \right) \bold{W} \\
& = \beta \bold{I_{n-1}} - \frac{1}{2} \bold{W}^\top \left( \bold{X} + \bold{X}^\top \right) \bold{W} 
 = (\beta - 1)\bold{I_{n-1}} + \frac{1}{2}\bold{W}^\top \left( \bold{L_H} + \bold{L_H}^\top \right) \bold{W} \succeq \mathbf{0},
\end{align*}
where we used the fact that $\bold{J_n W} = \mathbf{0}$ and $\frac{1}{2}(\bold{L_H} + \bold{L_H}^\top) = \bold{I_n} - \frac{1}{2}(\bold{X} + \bold{X}^\top )$. The linear matrix inequality above can be rewritten as
\begin{align*}
\frac{1}{2} \bold{W}^\top \left(\bold{L_H} + \bold{L_H}^\top \right) \bold{W} \succeq (1 - \beta)\bold{I_{n-1}} \quad \Longrightarrow \quad a(H) = \lambda_{\min} \left( \frac{1}{2} \bold{W}^\top \left(\bold{L_H} + \bold{L_H}^\top \right) \bold{W} \right) \geq 1 - \beta.
\end{align*}
Since $\beta < 1$, we have $\alpha(H) > 0$. Because $H$ is balanced, it follows from Proposition~\ref{prop:balanced} that $H$ is strongly connected and, thus, $H$ is a directed Hamiltonian cycle.

Conversely, let $H$ be a directed Hamiltonian cycle. Then, $a(H) = \lambda_{\min} \left( \frac{1}{2} \bold{W}^\top \left(\bold{L_H} + \bold{L_H}^\top \right) \bold{W} \right) = 1 - k_n$. Since $\beta \geq k_n$, we have
$
\frac{1}{2} \bold{W}^\top \left(\bold{L_H} + \bold{L_H}^\top \right) \bold{W} - (1 - \beta)\bold{I_{n-1}} \succeq \mathbf{0} \Longleftrightarrow \bold{W}^\top \bold{Z W} \succeq \mathbf{0},
$
following the same derivation as above. Now, let $\bold{x} \in \mathbb{R}^n$. Since the columns of $\bold{W}$ form a basis for $\mathbf{1}_n^\perp$, $\bold{x}$ can be written as $\bold{x} = \bold{Wy} + \delta \mathbf{1}_n$ for some $\bold{y} \in \mathbb{R}^{n-1}$ and $\delta \in \mathbb{R}$. This yields:
\begin{align*}
\bold{x}^\top \bold{Z x} & = \bold{y}^\top \bold{W}^\top \bold{Z W y} + 2 \delta \bold{y}^\top \bold{W}^\top \bold{Z} \mathbf{1}_n + \delta^2  \mathbf{1}_n^\top \bold{Z} \mathbf{1}_n \\
& =  \underbrace{\bold{y}^\top \bold{W}^\top \bold{Z W} \bold{y}}_{\geq 0} + \underbrace{2 \delta \bold{y}^\top \bold{W}^\top \left( (\beta-1) \mathbf{1}_n + \alpha n \mathbf{1}_n \right)}_{= 0} + \underbrace{\delta^2 n  \left( (\beta-1) + \alpha n \right)}_{\geq 0},
\end{align*}
where we used the facts that $\bold{W}^\top \bold{Z W} \succeq \mathbf{0}, \bold{W}^\top \mathbf{1}_n = \mathbf{0}$ and $\beta -1 + \alpha n \geq k_n -1 + n\frac{1 - k_n}{n} = 0$. Thus, $\bold{Z} \succeq \mathbf{0}$.
\end{proof}

In order to present our first ISDP formulation of the \textsc{QTSP}, we derive an explicit expression for the set $\mathcal{T}_n(G)$ and linearize the objective function. The former can be done using Theorem~\ref{Thm:CvetkovicDirected}. The set $\mathcal{T}_n(G)$ can be fully characterized by the permutation matrices that satisfy a linear matrix inequality. That is,
\begin{align} \label{T_n:SDP}
\mathcal{T}_n(G) = \Pi_n(G) \cap \left\{ \bold{X} \in \mathcal{S}^n \, : \, \, \beta \bold{I_n} + \alpha \bold{J_n} - \frac{1}{2}(\bold{X} + \bold{X}^\top ) \succeq \mathbf{0}  \right\},
\end{align}
for all $\alpha \geq h_n / n$ and $k_n \leq \beta < 1$. Recall that $\Pi_n(G)$ is the set of permutation matrices implied by $G$, see Section~\ref{Section:DefQTSP}.

To linearize the objective function, we follow the same construction as proposed by Fischer et al.~\cite{AandFFischer}. For all two-arcs $(i,j,k) \in \mathcal{A}$, see \eqref{def:twoarcs}, we define a variable $y_{ijk} := x_{ij}x_{jk}$. This equality can be guaranteed by the introduction of the following set of linear coupling constraints:
\begin{align*}
x_{ij} = \sum_{\substack{k \in N : \\ (k,i,j) \in \mathcal{A}}} y_{kij} = \sum_{\substack{k \in N : \\ (i,j,k) \in \mathcal{A}}} y_{ijk} \text{ for all } (i,j) \in A \quad \text{and} \quad y_{ijk} \geq 0  \text{ for all } (i,j,k) \in \mathcal{A}.
\end{align*}

We define the following set:
\begin{align}\label{setF1}
\mathcal{F}_1 := \left \{ (\bold{y}, \bold{X}) \in \{0,1\}^{\mathcal{A}} \times \Pi_n(G) \, : \, \,
 x_{ij} = \sum_{\substack{k \in N : \\ (k,i,j) \in \mathcal{A}}} y_{kij} = \sum_{\substack{k \in N : \\ (i,j,k) \in \mathcal{A}}} y_{ijk} \quad \forall (i,j) \in A  \right\}.
\end{align}
Now, our first ISDP formulation of the \textsc{QTSP} is as follows:
\begin{align*}  
\tag{$ISDP_{1}$} \label{eq:ISDP1}
\left\{
\begin{aligned}
\min\quad & \sum_{(i,j,k) \in \mathcal{A}} q_{ijk}y_{ijk} \\
\text{s.t.} \quad &   \beta \bold{I_n} + \alpha \bold{J_n} - \frac{1}{2} \left( \bold{X} + \bold{X}^\top \right) \succeq \mathbf{0}, 
\,\, (\bold{y},\bold{X}) \in \mathcal{F}_1,
\end{aligned} \right.
\end{align*}
where $\alpha \geq h_n /n$ and $k_n \leq \beta < 1$. One can verify that setting $\alpha = h_n /n$ and $\beta = k_n$ leads to the strongest linear matrix inequality among all possible values for $\alpha$ and $\beta$. Thus, we use these values in the computational results of Section~\ref{Section:ComputationResults}.

\begin{remark} \label{Remark:integrality}
In fact, we do not need to enforce integrality on $\bold{y}$ explicitly. Namely, if $\bold{X} \in \mathcal{T}_n(G)$, it follows from the integrality of $\bold{X}$ and the coupling constraints that $y_{ijk} = 1$ if $(i,j,k) \in \mathcal{A}$ is used in the tour and 0 otherwise. Hence, when optimizing over $\mathcal{F}_1$ using a B\&B or B\&C algorithm, we relax the integrality constraint on $\bold{y}$ and branch on $\bold{X}$ only. 
\end{remark}

In what follows, we further exploit properties of tour matrices to derive our second ISDP formulation of the \textsc{QTSP}.
Let $\bold{X} \in \mathcal{T}_n(G)$ be a tour matrix and define $\bold{X^{(2)}} = (x^{(2)}_{ij}) := \bold{X}\cdot \bold{X}$.
For $i,k \in N$ we have  $x^{(2)}_{ik}= \sum_{j=1}^n x_{ij}x_{jk} = \sum_{j \in N: (i,j,k) \in \mathcal{A}} y_{ijk}$,
where the last equality follows from the definition of $\bold{y}$.
Thus, $\bold{X^{(2)}}$ is a binary matrix and $x^{(2)}_{ik} =1$ if and only if the length of the shortest directed path from $i$ to $k$ in the subgraph induced by $\bold{X}$ is equal to two.

We can again characterize a tour matrix as in Theorem~\ref{Thm:CvetkovicDirected} by combining the variables $\bold{X}$ and $\bold{X^{(2)}}$. Observe that the directed graph induced by $\bold{X^{(2)}}$ is  balanced with in- and outdegree one, and circulant (but not strongly connected for even $n$). Moreover, the circulant graph $\mathcal{C}_2$ corresponding to $\bold{X}+\bold{X^{(2)}}$  is strongly connected  and balanced with  in- and outdegree two.
The spectrum of $ \frac{1}{2}( (\bold{X}+\bold{X^{(2)}}) +(\bold{X}+\bold{X^{(2)}} )^\top )$ for any $\bold{X} \in \mathcal{T}_n(G)$ and $\bold{X^{(2)}} = \bold{X} \cdot \bold{X}$ is given by
\begin{align} \label{specCirc}
\cos \left( \frac{2 \pi j}{n}\right) +  \cos \left( \frac{4 \pi j}{n}\right) \quad \text{for } j \in [n],
\end{align}
which results in the algebraic connectivity of $\mathcal{C}_2$ being $a(\mathcal{C}_2) = 2 - ( \cos (2 \pi / n)+\cos (4 \pi / n))$. We define
\begin{align}\label{knhntilde}
k_n^{(2)} :=  \cos \left(\frac{2\pi}{n}\right) + \cos \left(\frac{4\pi}{n}\right) \quad \text{and} \quad
h_n^{(2)} := 2 - k_n^{(2)}.
\end{align}
Now, we are ready to state the following theorem.

\begin{theorem} \label{Thm:DistanceTwoSDP}
Let $H$ be a spanning subgraph of a directed graph $G$ where the in- and outdegree equals one for all nodes in $H$. Let $\bold{X}$ be its adjacency matrix and let $\bold{X^{(2)}} := \bold{X}\cdot \bold{X}$ be the distance two adjacency matrix. Let $\alpha^{(2)}, \beta^{(2)} \in \mathbb{R}$ be such that $\alpha^{(2)} \geq h_n^{(2)} / n$ and $k_n^{(2)} \leq \beta^{(2)} < 2$,
with  $k_n^{(2)}$, $h_n^{(2)}$ as defined in \eqref{knhntilde}. Then $H$ is a directed Hamiltonian cycle if and only if
$$\bold{Z} := \beta^{(2)} \bold{I_n} + \alpha^{(2)} \bold{J_n} - \frac{1}{2}\left( (\bold{X} + \bold{X^{(2)}}) +(\bold{X} + \bold{X^{(2)}} )^\top  \right) \succeq \mathbf{0}.$$
\end{theorem}
\begin{proof}
Let $\tilde{H}$ be the subgraph of $G$ that has adjacency matrix $\bold{X} +  \bold{X^{(2)}}$. Observe that $\tilde{H}$ is balanced, and thus, $\tilde{H}$ is strongly connected if and only if $a(\tilde{H}) > 0$.

Let $\bold{Z} \succeq \mathbf{0}$, which implies that $\bold{W}^\top \bold{Z W} \succeq \mathbf{0}$. Now we can use a similar derivation as in the proof of Theorem~\ref{Thm:CvetkovicDirected}, which results in the following:
\begin{align*}
\frac{1}{2} \bold{W}^\top \left( \bold{L_{\tilde{H}}} + \bold{L_{\tilde{H}}}^\top \right) \bold{W} \succeq \left( 2 - \beta^{(2)} \right) \bold{I_{n-1}} \,\, \Longrightarrow \,\, a(\tilde{H}) = \lambda_{\min} \left( \frac{1}{2} \bold{W}^\top \left( \bold{L_{\tilde{H}}} + \bold{L_{\tilde{H}}}^\top \right) \bold{W} \right) \geq 2 - \beta^{(2)}.
\end{align*}
Since $\beta^{(2)} < 2$, we have $a(\tilde{H}) > 0$, and thus, $\tilde{H}$ is strongly connected. As $\tilde{H}$ is the union of a directed cycle cover and its implied distance two graph, $\tilde{H}$ can only be strongly connected if $H$ is strongly connected. We conclude that $H$ is a Hamiltonian cycle.

Conversely, let $H$ be a Hamiltonian cycle. In that case, the algebraic connectivity of $\tilde{H}$ is $a(\tilde{H}) = 2 - k_n^{(2)}$, i.e., $\lambda_{\min} \left( \frac{1}{2} \bold{W}^\top \left( \bold{L_{\tilde{H}}} + \bold{L_{\tilde{H}}}^\top \right) \bold{W} \right) = 2 - k_n^{(2)}$. Since $\beta^{(2)}  \geq k_n^{(2)}$, this yields
$$
\frac{1}{2} \bold{W}^\top \left( \bold{L_{\tilde{H}}} + \bold{L_{\tilde{H}}}^\top \right) \bold{W} - \left( 2 - \beta^{(2)} \right)\bold{I_{n-1}} \succeq \mathbf{0} \quad \Longleftrightarrow \quad \bold{W}^\top \bold{Z W} \succeq \mathbf{0}.
$$
Now we can use the same argument as in the proof of Theorem~\ref{Thm:CvetkovicDirected} to show that $\bold{Z} \succeq 0$ where $\beta, \alpha$ and $k_n$ are replaced by $\beta^{(2)}, \alpha^{(2)} $ and $k_n^{(2)}$, respectively.
\end{proof}
We define the set $\mathcal{F}_2$ as follows:
\begin{align}\label{setF2}
\mathcal{F}_2 := \left \{ \left(\bold{y}, \bold{X}, \bold{X^{(2)}} \right) \in \mathcal{F}_1  \times \Pi_n(G^2) \, : \, \,
x^{(2)}_{ik} = \sum_{ \substack{j \in N : \\ (i,j,k) \in \mathcal{A}} }^n y_{ijk} \,\, \forall (i,k) \in A^2
\right \},
\end{align}
where 
\[
\Pi_n(G^2):=\left \{ \bold{X^{(2)}} \in  \{0,1\}^{n \times n} \,  : \, \,
\bold{X^{(2)}} \bold{1} = \bold{1}, \, (\bold{X^{(2)}})^\top \bold{1} = \bold{1}, \, \diag(\bold{X^{(2)}}) = \bold{0}, \,\, x^{(2)}_{ij} = 0 \,\, \forall (i,j) \notin A^2 \right \},
\]
and $A^2$ is the set of node pairs $(i,j)$ for which there exists a directed path from $i$ to $j$ of length $2$. The set $\mathcal{F}_2$ and the result of Theorem~\ref{Thm:DistanceTwoSDP} lead to our second ISDP formulation of the \textsc{QTSP}:
\begin{align*} 
\tag{$ISDP_{2}$} \label{eq:ISDP2}
\left\{
\begin{aligned}
\min\quad & \sum_{(i,j,k) \in \mathcal{A}} Q_{ijk}y_{ijk} \\
\text{s.t.} \quad &   \beta \bold{I_n} + \alpha \bold{J_n} - \frac{1}{2} \left( \bold{X} + \bold{X}^\top \right) \succeq \mathbf{0}  \\
& \, \,  \beta^{(2)} \bold{I_n} + \alpha^{(2)} \bold{J_n} - \frac{1}{2}\left( (\bold{X} + \bold{X^{(2)}} ) +(\bold{X} + \bold{X^{(2)}}  )^\top  \right) \succeq \mathbf{0}, 
\, \, (\bold{y},\bold{X}, \bold{X^{(2)}}) \in \mathcal{F}_2,
\end{aligned} \right.
\end{align*}
where $\alpha \geq h_n /n$, $k_n \leq \beta < 1$,  $\alpha^{(2)} \geq h_n^{(2)} / n$ and $k_n^{(2)} \leq \beta^{(2)} < 2$.
Again the choice of $\alpha, \beta, \alpha^{(2)}$ and $\beta^{(2)}$ equal to their lower bounds provides the strongest continuous relaxation.

It follows  from Theorem~\ref{Thm:DistanceTwoSDP} that one can remove the first linear matrix inequality in  $(ISDP_{2})$ and still obtain an exact formulation of the \textsc{QTSP}. However,  the bound obtained from the SDP relaxation of $(ISDP_{2})$ dominates the bound obtained from the SDP relaxation of  $(ISDP_{1})$. In that sense, the formulation \eqref{eq:ISDP2} can be seen as a level two formulation of the \textsc{QTSP}, whose continuous relaxation is stronger than that of the first level formulation. An additional advantage of the level two formulation is that both linear matrix inequalities may be used to generate CG cuts, as we show in the following section.

In the same vein, one can construct level $k$ formulations of the \textsc{QTSP} for $k = 3, \ldots, n$. This leads to a hierarchy of formulations, whose SDP relaxations are of increasing strength and complexity.

\subsection{Chv\'atal-Gomory cuts for the ISDPs of the QTSP}
\label{Section:CutsforQTSP}
In order to solve $(ISDP_1)$ and $(ISDP_2)$ using our B\&C algorithm, we study various CG-based separation routines for the \textsc{QTSP}.
We first derive a general CG cut generator for the formulations \eqref{eq:ISDP1} and \eqref{eq:ISDP2}. Thereafter, we show how different types of well-known inequalities for the \textsc{QTSP} can be derived as CG cuts of the formulations $(ISDP_{1})$  and $(ISDP_{2})$.

Let us consider $(ISDP_{1})$.  The set $\mathcal{F}_1$, see \eqref{setF1}, consists of all tuples $(\bold{y},\bold{X})$ where $\bold{X}$ represents a node-disjoint cycle cover in $G$. Our B\&C algorithm starts with optimizing over the set $\mathcal{F}_1$, where we are allowed to relax the integrality of $\bold{y}$ at no cost, see Remark~\ref{Remark:integrality}.
If an integer point $(\bold{\hat{y}}, \bold{\hat{X}}) $ is found in the branching tree, it is verified whether $\lambda_{\min} \left(\beta \bold{I_n} + \alpha \bold{J_n} - \frac{1}{2}\left( \bold{\hat{X}} + \bold{\hat{X}}^\top \right) \right) \geq 0$.
If so, then $\bold{\hat{X}} \in \mathcal{T}_n(G)$ and we have found a possibly new incumbent solution.
 If not, then $\bold{\hat{X}}$ is the adjacency matrix of a node-disjoint cycle cover that is not a Hamiltonian cycle. Therefore we have to generate dual matrices that cut off the current point.

The first separation routine that we present is based on finding a set of integer eigenvectors corresponding to a negative eigenvalue of $\beta \bold{I_n} + \alpha \bold{J_n} - \frac{1}{2}\left( \bold{\hat{X}} + \bold{\hat{X}}^\top \right)$.
\begin{proposition} \label{Prop:GomoryCut} Let $\bold{X} \in \Pi_n(G)$ be the adjacency matrix of a directed node-disjoint cycle cover consisting of $k \geq 2$ cycles. Let $\{S_1, \ldots, S_k \}$ be the partition of the nodes implied by the cycle cover and define for each $l \in [k]$ the vector
\begin{align*}
v^l_i := \begin{cases} n - |S_l| & \text{if } i \in S_l \\
- |S_l| & \text{if } i \notin S_l.
\end{cases}
\end{align*}
Then $\left\langle \bold{v^l}(\bold{v^l})^\top , \, \beta \bold{I_n} + \alpha \bold{J_n} - \frac{1}{2}( \bold{X} + \bold{X}^\top ) \right\rangle < 0$ for all $l \in [k]$.
\end{proposition}
\begin{proof}
The vectors $\bold{v^l}$ are eigenvectors of $\bold{X}$ and $\bold{X}^\top$ corresponding to  eigenvalue $1$. Therefore we have:
\begin{align*}
\left( \beta \bold{I_n} + \alpha \bold{J_n} - \frac{1}{2}( \bold{X} + \bold{X}^\top ) \right) \bold{v^l} & = \beta \bold{v^l} + \alpha \left( (n - |S_l|)\cdot|S_l| + (n-|S_l|)\cdot(-|S_l|) \right) \bold{1} - \frac{1}{2}\bold{v^l} - \frac{1}{2}\bold{v^l} \\
& = (\beta - 1)\bold{v^l},
\end{align*}
from where it follows that $\bold{v^l}$ is an eigenvector of $\beta \bold{I_n} + \alpha \bold{J_n} - \frac{1}{2}( \bold{X} + \bold{X}^\top )$ corresponding to eigenvalue $\beta - 1$. Since we assume $\beta < 1$, this eigenvalue is negative, from which the conclusion follows.
\end{proof}
The result of Proposition~\ref{Prop:GomoryCut} can be used within our B\&C algorithm in the following way. Let $\{S_1, \ldots, S_k\}$ be the partition of the nodes implied by the current solution $\bold{\hat{X}}$ and let $\bold{U^l} := \bold{v^l}(\bold{v^l})^\top$ where $\bold{v^l}$ is as defined in Proposition~\ref{Prop:GomoryCut}. Then for each $l \in [k]$ we construct the following CG cuts:
\begin{align}\label{constraint:eig_ISDP1}
\left\langle \bold{U^l}, \frac{1}{2}(\bold{X} + \bold{X}^\top)\right \rangle \leq \left\lfloor \langle \bold{U^l}, \beta \bold{I_n} + \alpha \bold{J_n} \rangle \right \rfloor \quad \Longleftrightarrow \quad \left\langle \bold{U^l}, \bold{X} \right \rangle \leq \left\lfloor \langle \bold{U^l}, \beta \bold{I_n} + \alpha \bold{J_n} \rangle \right \rfloor,
\end{align}
which cut off the current point. Observe that the choice $\alpha = h_n /n$ and $\beta = k_n$ leads to non-integer values for $\alpha$ and $\beta$, i.e., the CG rounding step provides a strengthened eigenvalue cut.

Since the result of Proposition~\ref{Prop:GomoryCut} can be repeated for the extended linear matrix inequality in Theorem~\ref{Thm:DistanceTwoSDP}, we also obtain the following CG cuts with respect to $(ISDP_2)$:
\begin{align} \label{constraint:eig_ISDP2}
\left\langle \bold{U^l}, \bold{X} + \bold{X^{(2)}} \right \rangle \leq \left\lfloor \langle \bold{U^l}, \beta^{(2)} \bold{I_n} + \alpha^{(2)} \bold{J_n} \rangle \right \rfloor \qquad \forall l \in [k].
\end{align}

Next, we consider the class of subtour elimination constraints. It has been shown by \c{C}ezik and Iyengar~\cite{CezikIyengar} that the ordinary subtour elimination constraints defined by Dantzig et al.~\cite{DantzigEtAl} can be obtained as CG cuts for the symmetric TSP, provided that $\alpha$ and $\beta$  equal their lower bounds.
We extend the result from~\cite{CezikIyengar} and present five types of subtour elimination constraints that are in fact (strengthened) CG cuts of $(ISDP_1)$ and/or $(ISDP_2)$, see Table~\ref{Table:SEC}. Many of these constraints do not follow directly from the linear matrix inequalities, but require the addition of a positive multiple of a subset of the affine constraints. It is shown by Fischer~\cite{Fischer2014} that the inequalities IV and V of Table~\ref{Table:SEC} define facets of the asymmetric quadratic traveling salesman polytope.

\begin{table}[h]
\footnotesize
\centering
\begin{tabular}{@{}lcl@{}}
\toprule
 & \textbf{Inequality} & \textbf{Description}  \\ \midrule
\addlinespace[2ex] I & $\sum\limits_{\substack{i \in S \\ j \in S}}x_{ij} \leq |S| - 1, \quad \forall S \subset N, 2 \leq |S| < n$ &   \begin{tabular}[c]{@{}l@{}} CG cut of $\beta \bold{I_n} + \alpha \bold{J_n} - \frac{1}{2}\left( \bold{X} + \bold{X}^\top \right) \succeq \mathbf{0}$ with dual\\ multiplier $\bold{U} = \bold{\mathbbm{1}_S}\bold{\mathbbm{1}_S}^\top$.  \end{tabular}                      \\
\addlinespace[3ex] II                                                                                                   &      $\sum\limits_{\substack{i \in S \\ j \notin S}}x_{ij} \geq 1, \quad \forall S \subset N, 2 \leq |S| < n$                    &     \begin{tabular}[c]{@{}l@{}} CG cut of  $\beta \bold{I_n} + \alpha \bold{J_n} - \frac{1}{2}\left( \bold{X} + \bold{X}^\top \right) \succeq \mathbf{0}$ with dual \\multiplier $\bold{U} = \bold{\mathbbm{1}_S}\bold{\mathbbm{1}_S}^\top$ and $-\bold{X}\mathbf{1} = -\mathbf{1}$ with dual \\ multiplier $\bold{\mathbbm{1}_S}$.      \end{tabular}      \\
\addlinespace[4ex] III & \begin{tabular}[c]{@{}l@{}} $\sum\limits_{l = 1}^k \sum\limits_{\substack{i \in S_l \\ j \in S_l}} x_{ij} - \sum\limits_{l \neq p} \sum\limits_{\substack{i \in S_l \\ j \in S_p}}x_{ij} \leq n - 2k$ \\[0.7cm] $\forall (S_1, \ldots, S_k), \cup_{l = 1}^k S_l = N, S_l\cap S_p = \emptyset \, \, \forall l \neq p$ \end{tabular} & \begin{tabular}[c]{@{}l@{}} CG cut of $\beta \bold{I_n} + \alpha \bold{J_n} - \frac{1}{2}\left( \bold{X} + \bold{X}^\top \right) \succeq \mathbf{0}$ with dual\\ multiplier $\bold{U} = 2 \sum_{l = 1}^k\bold{\mathbbm{1}_{S_l}} \bold{\mathbbm{1}_{S_l}}^\top$
and $-\bold{X} \mathbf{1} = - \mathbf{1}$ with\\dual multiplier $\mathbf{1}$. \end{tabular} \\
\addlinespace[4ex] IV & \begin{tabular}[c]{@{}l@{}} $x_{ij} + x_{ji} + \sum\limits_{\substack{k \in N : \\ (i,k,j) \in \mathcal{A}}} y_{ikj}  + \sum\limits_{\substack{k \in N : \\ (j,k,i) \in \mathcal{A}}} y_{jki} \leq 1$ \\[0.7cm]
$\forall i, j \in N, i \neq j, n \geq 5$
\end{tabular} & \begin{tabular}[c]{@{}l@{}} $S$-CG cut of
$\beta^{(2)} \bold{I_n} + \alpha^{(2)} \bold{J_n} - \frac{1}{2}\big ( (\bold{X} + \bold{X^{(2)}})$ \\ $ +(\bold{X} + \bold{X^{(2)}} )^\top  \big ) \succeq \mathbf{0}$
with dual multiplier \\ $\bold{U} = \bold{\mathbbm{1}_{\{i,j\}}\mathbbm{1}_{\{i,j\}}}^\top$ and $\sum_{\substack{k \in N: \\(i,k,j) \in \mathcal{A}}}y_{ikj} - x^{(2)}_{ij} = 0$, \\
$\sum_{\substack{k \in N: \\(j,k,i) \in \mathcal{A}}}y_{jki} - x^{(2)}_{ji} = 0$,
$-x_{ii} = 0$, $-x_{jj}=0$, \\ $-x^{(2)}_{ii} = 0$ and $-x^{(2)}_{jj} = 0$, each
with dual multiplier 1.
\end{tabular}
\\
\addlinespace[3ex] V & \begin{tabular}[c]{@{}l@{}} $\sum\limits_{\substack{i \in S \\ j \in S}}x_{ij} + \sum\limits_{\substack{i \in S \\ j \in S}}\sum\limits_{\substack{k \in N \setminus S : \\ (i,k,j) \in \mathcal{A}}} y_{ikj}   \leq |S| - 1$ \\[0.7cm] $\forall S \subset N, 2 \leq |S| < \frac{1}{2}n$ \end{tabular} &
\begin{tabular}[c]{@{}l@{}} $S$-CG cut of
$\beta^{(2)} \bold{I_n} + \alpha^{(2)} \bold{J_n} - \frac{1}{2}\big ( (\bold{X} + \bold{X^{(2)}})$ \\ $ +(\bold{X} + \bold{X^{(2)}})^\top  \big ) \succeq \mathbf{0}$
with dual multiplier \\ $\bold{U} = \bold{\mathbbm{1}_{S}\mathbbm{1}_{S}}^\top$ and $\sum_{k \in N: (i,k,j) \in \mathcal{A}}y_{ikj} - x^{(2)}_{ij} = 0$,
for all \\ $i,j \in S$, each with dual multiplier 1, and $-y_{ikj} \leq 0$ \\for all $(i,k,j) \in \mathcal{A}$ with $i,k,j \in S$, each with dual \\ multiplier 1.
\end{tabular} \\
 \bottomrule
\end{tabular}
\caption{Five types of subtour elimination constraints for the \textsc{QTSP} that can be obtained as (strengthened) CG cuts of $(ISDP_1)$ and/or $(ISDP_2)$. The third column describes which (in)equalities and dual multipliers are used to construct the inequality.\label{Table:SEC}}
\end{table}
In Appendix~\ref{App:SEC}, we explicitly derive these inequalities as (strengthened) CG cuts.

\section{Computational Results} \label{Section:ComputationResults}
In this section we test our ISDP formulations of the \textsc{QTSP}, see Section~\ref{Section:QTSP}. We solve the ISDPs using various settings of our CG-based B\&C framework, see Algorithm~\ref{AlgB&C}, where we include different sets of cuts from Section~\ref{Section:CutsforQTSP} in the separation routines. We compare the performance of our approach with the two other ISDP solvers from the literature.

\subsection{Design of numerical experiments}
In total we compare seven different approaches, among which two from the literature and five variants of our B\&C approach. The former class consists of the following:
\begin{itemize}
\item \textbf{\textit{KT}:} The B\&C algorithm of Kobayashi and Takano \cite{KobayashiTakano}, see Section~\ref{Subsection:B&C}.
\item \textbf{\textit{SCIP-SDP}:} The general ISDP solver of Gally et al.\ \cite{GallyEtAl}. This approach is based on solving continuous SDPs in a B\&B framework.
\end{itemize}
A third project that is known for its ability to solve ISDPs is YALMIP \cite{Lofberg}. Preliminary experiments show, however, that the solver of \cite{Lofberg} is significantly outperformed by the solvers from \cite{GallyEtAl} and \cite{KobayashiTakano}. Therefore, we do not take the solver of YALMIP into account. 

On top of the approaches from the literature, we consider five variants of our B\&C procedure that differ in the initial feasible set and the type of cuts that we add in the separation routine:
\begin{itemize}
\item \textbf{\textit{CG1}:} In this setting we solve $(ISDP_1)$ where we initially optimize over $\mathcal{F}_1$, see \eqref{setF1}. In the separation routine we add the CG cut of the form \eqref{constraint:eig_ISDP1} for each subtour present in the current candidate solution.

\item \textbf{\textit{CG2}:} In this setting we solve the second \textsc{QTSP} formulation $(ISDP_2)$. We initially optimize over $\mathcal{F}_2$, see \eqref{setF2}, and in each callback iteration we add the CG cuts of the form \eqref{constraint:eig_ISDP1} and \eqref{constraint:eig_ISDP2} for each subtour in the current candidate solution.

\item \textbf{\textit{SEC-simple}:} In this setting we solve $(ISDP_1)$ by starting from optimizing over $\mathcal{F}_1$, see \eqref{setF1}. In the callback procedure, we add the ordinary subtour elimination constraints, see Type~I in Table~\ref{Table:SEC}, for all subtours in the current candidate solution.

\item \textbf{\textit{SEC}:} \label{page:SEC}This setting solves $(ISDP_2)$ with subtour elimination constraints of Type~I, IV and V from Table~\ref{Table:SEC}. The latter type of constraint is added only for the subtours of size less than $\frac{1}{2}n$. Since the order two variables $\bold{X^{(2)}}$ in this setting do not appear directly in the cutting planes, we eliminate them also from the initial MILP based on preliminary tests. That is, we start optimizing over $\mathcal{F}_1$, see \eqref{setF1}. Moreover, based on a result by Fischer et al.\ \cite{AandFFischer} we also add additional cuts to forbid subtours of three nodes. For a triple $i, j, k$ of distinct nodes, the following cut is valid for any tour:
$y_{ijk} + y_{kij} \leq x_{ij}.$
We add this cut for all distinct $i, j, k \in S$ in the separation routine whenever a subtour on $S$ with $|S| = 3$ is present in the current candidate solution. Observe that there are six of them for each triple of nodes.

\item \textbf{\textit{SEC-CG}:} This setting solves $(ISDP_2)$, starting from $\mathcal{F}_2$, see \eqref{setF2}. In the separation routines, we add all the cuts that are included in the previous setting SEC. Moreover, on top of that we also add the CG cuts \eqref{constraint:eig_ISDP1} and \eqref{constraint:eig_ISDP2} in the callback procedure.

\end{itemize}
Recall that the separation routines are only called at integer points, which represent cycle covers of $G$. Therefore, the separation of all mentioned cuts boils down to identifying the subtours in the cycle cover. Also, recall that the integrality of $\bold{y}$ is relaxed in all settings, see Remark~\ref{Remark:integrality}.

The setting SEC looks similar to the best exact \textsc{QTSP} solving strategy of Fischer et al.\ \cite{Fischer2014}. However, there are two main differences between the methods. First, our separation routine is only called on integer points, while the algorithm of \cite{AandFFischer} separates on fractional points. The separation on integer points is computationally very cheap compared to the fractional separation method applied by~\cite{AandFFischer}. Consequently, the former separation can lead to superior behavior, as observed by Aichholzer et al.\ \cite{Aichholzer} for the symmetric \textsc{QTSP}. Second, our approach results from a more general B\&C framework for solving integer SDPs, which is not limited to the \textsc{QTSP}.

Notice that the derived CG cuts of Type~II and III from Table~\ref{Table:SEC} are not added in the test settings. Preliminary experiments have shown that the cut-set subtour elimination constraints (Type~II of Table~\ref{Table:SEC}) have similar practical behaviour compared to the ordinary subtour elimination constraints. Also, preliminary tests show that the addition of one merged Type~III cut instead of all separate Type~I cuts leads to worse behaviour in terms of overall computation time. We expect this difference to be caused by the sparsity of the Type~I cuts, compared to the very dense Type~III cuts.

For our tests, we consider three types of instances\footnote{Instances can be downloaded from \url{https://github.com/frankdemeijer/CGforISDP}.}:
\begin{itemize}
\item \textbf{Real instances from bioinformatics:} J\"ager and Molitor~\cite{Jager}, Fischer~\cite{Fischer2014} and Fischer et al.~\cite{AandFFischer, AandFFischer_RNA} consider an important application of the \textsc{QTSP} in computational biology. In order to recognise transcription factor binding sites or RNA splice sites in a given set of DNA sequences, Permuted Markov (PM) models \cite{EllrotEtAl} or Permuted Variable Length Markov (PVLM) models~\cite{ZhaoEtAl_bio} can be used. Finding the optimal order two PM or PVLM model boils down to solving a \textsc{QTSP} instance. We consider three classes of bioinformatics instances used in~\cite{Fischer, Fischer2014}, which are denoted by `bma', `map' and `ml'. Each class consists of 38 instances with $n \in \{3, \ldots, 40\}$.

\item \textbf{Reload instances:} The reload instances are the same as the ones used by Rostami et al.~\cite{RostamiEtAl} and De Meijer and Sotirov~\cite{DeMeijerSotirov2}. The reload model~\cite{WirthSteffan} is inspired by logistics and energy distribution, where a certain cost is incurred whenever the underlying type of arc in a network changes, e.g., the means of transport. Let $G$ be a directed graph where each arc $(i,j)$ is present with probability $p$. Each arc in $G$ is randomly assigned a color from a color set $L$ with cardinality $c$. If two successive arcs $e$ and $f$ have colors $s$ and $t$, respectively, the quadratic cost among $e$ and $f$ equals $r(s,t)$, where $r: L \times L \rightarrow \mathbb{R}$ is a reload cost function such that $r(s,s) = 0$ for all $s \in L$. We consider two types of reload classes:
\begin{itemize}
\item \textit{Reload class 1}: For each pair of distinct colors $s,t \in L$ the reload cost equals $r(s,t) = 1$;
\item \textit{Reload class 2}: For each pair of distinct colors $s,t \in L$, the reload cost $r(s,t)$ is chosen uniformly at random from $\{1, \ldots, 10\}$.
\end{itemize}
For each class, we consider 10 distinct instances for each possible combination of $n \in \{10, 15, 20, {25}\}$, $p \in \{0.5, 1\}$ and $c \in \{5, 10, 20\}$, {except for the combination between $n = 25$ and $p = 1$ due to extremely large computation times.} Thus, in total we consider {420} reload instances.

\item \textbf{Turn cost instances:} The special case of the \textsc{QTSP} where the nodes are points in Euclidean space and the angle cost of a tour is the sum of the direction changes at the points is called the Angular-Metric Traveling Salesman Problem (\textsc{Angle-TSP}) \cite{Aggarwal}. The \textsc{Angle-TSP} is motivated by VLSI design and proven to be $\mathcal{NP}$-hard \cite{Aggarwal}. The problem is in the literature also known as the Minimum Bends Traveling Salesman Problem \cite{SteinWagner}. We consider two classes of this type:
\begin{itemize}
\item \textit{TSPLIB instances:} The TSP library (TSPLIB) \cite{TSPLIB} contains a broad set of TSP test instances, among which a large number of Euclidean instances. We construct a corresponding \textsc{QTSP} instance as follows: Given points $v_1, \ldots, v_n$ in $\mathbb{R}^2$, we let $G$ be the complete graph on $n$ vertices. For $i, j, k$, $i \neq j$, $j \neq k$, $i \neq k$, we define $q_{ijk}$ to be proportional to the angle between edges $\{i, j\}$ and $\{j, k\}$. More precisely,
\begin{align*}
q_{ijk} := \left\lceil 10 \cdot \left( 1 - \frac{1}{\pi}\arccos\left( \frac{(v_i - v_j)^\top (v_k - v_j)}{\Vert v_i - v_j \Vert \cdot \Vert v_j - v_k \Vert} \right) \right) \right \rceil .
\end{align*}
This cost structure is similar to the angle-distance costs considered in Fischer et al.~\cite{AandFFischer} and De Meijer and Sotirov \cite{DeMeijerSotirov}.
In total, we consider 9 TSPLIB instances with $n$ ranging from 15 to 70. Figure~\ref{fig:kn57} depicts one of the TSPLIB instances including its optimal tour with respect to the defined quadratic cost structure.

\item \textit{Grid instances:} Fekete and Krupke \cite{FeketeKrupke, FeketeKrupke2} consider problems of computing optimal covering tours and cycle covers under a turn cost model, see also Arkin et al.\ \cite{ArkinEtAl}. These problems have many practical applications, such as pest control and precision farming. Following this line, we consider the \textsc{Angle-TSP} on grid graphs. We construct a 2D connected grid graph using the Type II instance generator of \cite{FeketeKrupke2}. Given the vertex coordinate vectors $v^1, \ldots, v^n \in \{0, \ldots, N_1\} \times \{0, \ldots, N_2\}$ for integers $N_1, N_2$, we include an edge between vertex $i$ and $j$ if and only if ($v_1^i = v_1^j$ and $|v^i_2 - v^j_2| = 1$) or ($v^i_2 = v^j_2$ and $|v^i_1 - v^j_1| = 1$). If two edges $\{i, j\}$ and $\{j, k\}$ are present, the quadratic costs are computed similar as for the TSPLIB instances. In total we consider 9 grid instances with $N_1$ and $N_2$ running from 20 to 80, corresponding to $n$ ranging from $430$ to $2646$. An example of a grid instance including its minimum bend tour is given in Figure~\ref{fig:grid1}. 
\end{itemize}
Both types of turn cost instances are in fact instances of the symmetric \textsc{QTSP}, as they are defined on undirected graphs. To account for this, we use symmetrized versions of $(ISDP_1)$ and $(ISDP_2)$ instead. We refer to Appendix~\ref{App:SQTSP} for the construction of these formulations.
\end{itemize}
\begin{figure}[h]
\centering
\begin{subfigure}[b]{0.45\textwidth}
\centering
\includegraphics[width = 0.8\textwidth]{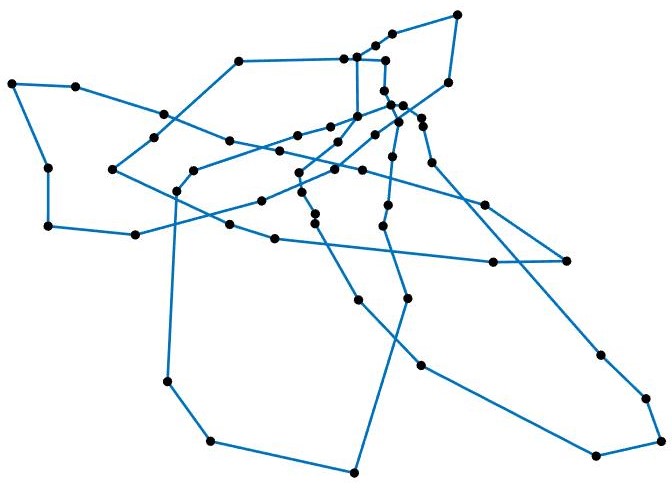}
\caption{Instance `kn57'}
\label{fig:kn57}
\end{subfigure}
\begin{subfigure}[b]{0.45\textwidth}
\centering
\includegraphics[width = 0.8\textwidth]{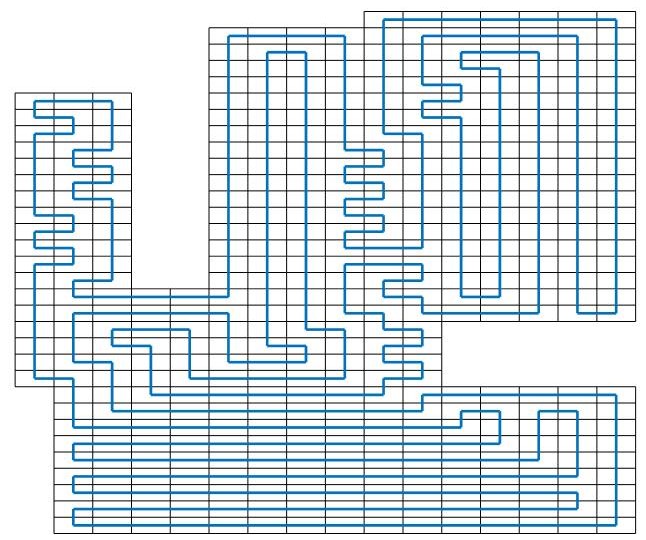}
\caption{Instance `grid1' \label{fig:grid1}}
\end{subfigure}
\caption{Optimal tours of two turn instances: the TSPLIB instance `kn57' ($n = 57$) and the grid instance `grid1' ($n = 430$).  Each square in Figure 2b represents a vertex in the grid graph.}
\end{figure}
All our algorithms, including the algorithm of Kobayashi and Takano \cite{KobayashiTakano}, are implemented in Julia~1.5.3 using JuMP v0.21.10 \cite{DunningEtAl} to model the mathematical optimization problems. In particular, we exploit the solver-independent lazy constraint callback option of JuMP to include the separation routines. Solving the underlying MILP in the subproblems is done using Gurobi v9.10 \cite{Gurobi} in the default settings including built-in cuts. Experiments are carried out on a PC with an Intel(R) Core(TM) i7-8700 CPU, 3.20GHz, 8GB RAM. To run SCIP-SDP, we use SCIP-SDP version 3.2.0 on the NEOS Server \cite{NEOS}, where the B\&B framework of SCIP 7.0.0 \cite{SCIP700} and the SDP solver Mosek~9.2~\cite{Mosek} are combined in the default configuration. 

Observe that an older version of SCIP-SDP with DSDP \cite{BensonEtAl} as SDP solver was used in the numerical experiments of \cite{KobayashiTakano}, which partly explains the poor behaviour of SCIP-SDP compared to the B\&C algorithm of \cite{KobayashiTakano}. However, our computational study that uses SCIP-SDP with the state-of-the-art SDP solver Mosek \cite{Mosek} also shows superior behaviour of the B\&C algorithms.

We test all seven settings on the bioinformatics and reload instances. Since these instance classes give a clear and consistent overview of the superior approaches, we restrict ourselves to the best three settings for the turn cost instances. The maximum computation time for all our approaches is set to 8 hours, which is in correspondence with the maximum computation time on the NEOS Server \cite{NEOS}.

\subsection{Comparison of approaches}
Table~\ref{Tab:bio_results} and Figure~\ref{fig:bio_results} provide an overview of the performance on the instances from bioinformatics. For each setting, the average values in Table~\ref{Tab:bio_results} are only computed over the instances that could be solved to optimality for that setting. An extended table on the results per instance can be found in Appendix~\ref{App:computations}.
{Observe that the percentage of instances solved is quite similar over the three instance classes. This indicates that it is mainly the size rather than the cost structure that determines whether a bioinformatics instance can be solved or not.}
It is clear that our B\&C settings significantly outperform the other two ISDP solvers SCIP-SDP and KT, which can solve at most 60\% of the instances to optimality. Since the separation routine of CG1 is based on the identification of an integer eigenvector corresponding to a negative eigenvalue, the settings KT and CG1 are almost identical apart from the CG rounding step. The large decrease in the number of branching nodes of CG1 compared to KT is remarkable. This indicates that the effect of deeper cuts as shown in Figure~\ref{fig:B&Cexample} is not solely theoretical, but also {turns out to be substantial} from a practical point of view.

When comparing the five different separation routines of our B\&C approach, we also see a clear pattern. The settings SEC and SEC-CG turn out to be superior, being able to solve all instances within short computation times. Although SEC generally provides the fastest algorithm, it sometimes happens that SEC-CG solves the instance faster, see Figure~\ref{fig:bio_results}, due to the smaller number of B\&C nodes. This shows that the additional CG cuts can sometimes improve on the subtour elimination constraints. The two approaches are followed by SEC-simple, which is able to solve instances up to $n = 35$ to optimality. This difference is mainly due to the strengthened subtour elimination cuts (type IV and V in Table~\ref{Table:SEC}) that work well for the bioinformatics instances, as also noted by Fischer et al.\ \cite{AandFFischer}. Finally, the settings CG1 and CG2 are only able to solve instances up to $n = 32$ and $n = 27$, respectively. Although the distance two CG cuts \eqref{constraint:eig_ISDP2} significantly reduce the number of needed branching steps, the overall computation time is larger due to the increase in the number of variables and constraints in CG2.

\begin{table}[h]
\centering
\footnotesize
\begin{tabular}{@{}ll|ccccccc@{}}
\hline
\textbf{Type}\rule{0pt}{3ex} & \textbf{Statistic}       & \textbf{SCIP-SDP} & \textbf{KT} & \textbf{CG1} & \textbf{CG2} & \textbf{SEC-simple} & \textbf{SEC}   & \textbf{SEC-CG} \\ \hline
\textbf{bma}\rule{0pt}{3ex}  & Instances solved (\%)     & 34.21             & 60.53       & 78.95        & 65.79        & 84.21               & \textbf{100}   & \textbf{100}    \\
\textbf{}     & Average comp. time & 1519              & 1581        & 846.1        & 639.40       & 817.43              & \textbf{28.31} & 182.3           \\
\textbf{}     & Average B\&C nodes        & 30308             & 854964      & 119144       & 42515        & 85984               & 200.4          & 142.5           \\
\textbf{}     & Average time per node    & 0.025             & 0.001       & 0.002        & 0.006        & 0.005               & 0.184          & 1.158           \\[0.5em]
\textbf{map}  & Instances solved (\%)     & 34.21             & 57.89       & 78.95        & 65.79        & 81.58               & \textbf{100}   & \textbf{100}    \\
              & Average comp. time & 2247              & 1721        & 1768         & 806.6        & 911.4               & \textbf{25.83} & 199.6           \\
              & Average B\&C nodes        & 30385             & 896340      & 245732       & 56197        & 79869               & 244            & 173             \\
              & Average time per node    & 0.037             & 0.001       & 0.002        & 0.009        & 0.004               & 0.496          & 2.094           \\[0.5em]
\textbf{ml}   & Instances solved (\%)     & 34.21             & 57.89       & 76.32        & 65.79        & 81.58               & \textbf{100}   & \textbf{100}    \\
              & Average comp. time & 2891              & 1315        & 460.9        & 805.6        & 520.2               & \textbf{27.34} & 221.7           \\
              & Average B\&C nodes        & 33185             & 658640      & 86743        & 44342        & 51495               & 252.0          & 186.3           \\
              & Average time per node    & 0.034             & 0.001       & 0.002        & 0.007        & 0.005               & 0.096          & 0.961           \\ \hline
\end{tabular}
\caption{Summary table of the performance on the bioinformatics instances per setting and per instance type. The best performing setting per row is given in bold. \label{Tab:bio_results}}
\end{table}

\begin{figure}[H]
\centering
\includegraphics[scale=0.34]{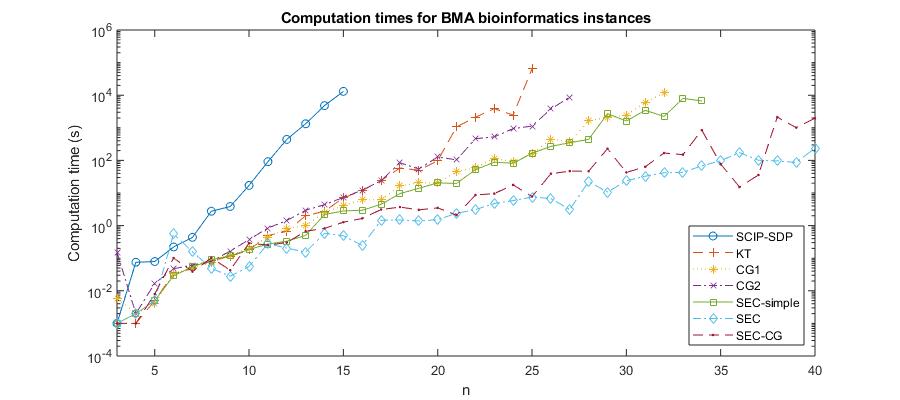}
\includegraphics[scale=0.34]{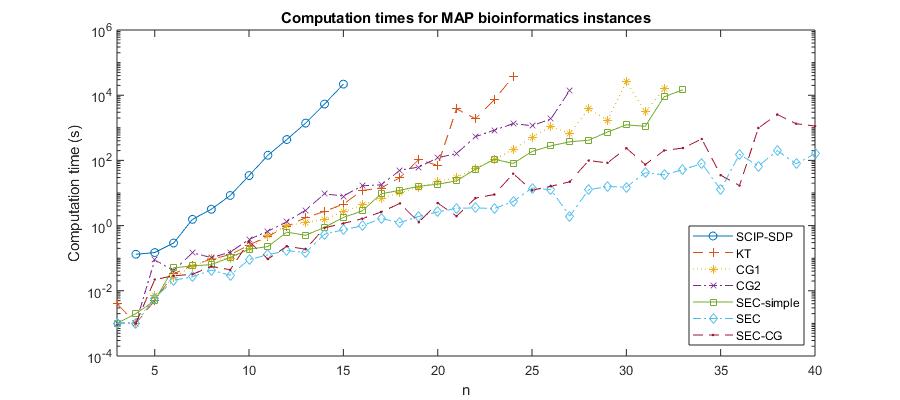}
\includegraphics[scale=0.34]{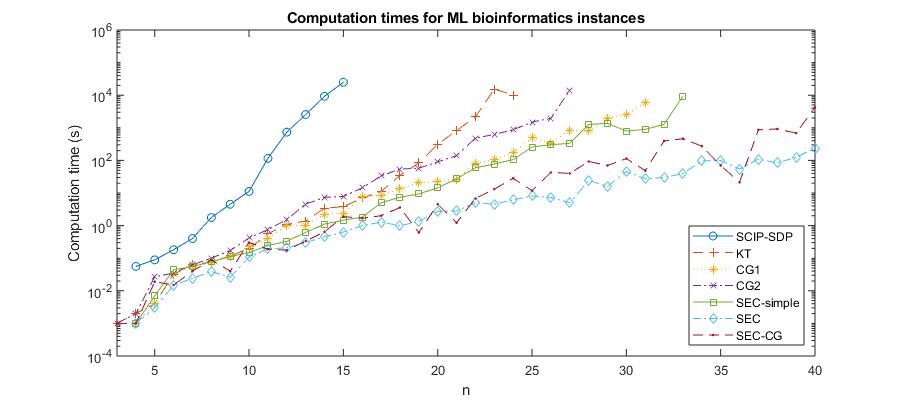}
\caption{Computation times versus instance size for the bioinformatics classes `bma' (top), `map' (middle) and `ml' (bottom). The computation times are given on a logarithmic scale.}
\label{fig:bio_results}
\end{figure}

Next, we discuss the results on the set of reload instances. For both class 1 and 2 and for each value of $n, p$ and $c$ we consider 10 randomly generated instances. The averaged results for each combination of parameters can be found in Appendix~\ref{App:computations}, see Table~\ref{Tab:Reload_ext1}, \ref{Tab:Reload_ext2} { and \ref{Tab:Reload_ext3}}. In general, we see that the computation times increase with the number of nodes $n$ and the graph density $p$. On the other hand, if the number of colors $c$ increases, the instances become easier to solve as the number of (optimal) solutions will decrease. Table~\ref{Tab:reload_results} shows a summary of the results accumulated over the number of colors $c$. 
Accordingly, Figure~\ref{fig:reload_results} shows the spread of the computation times, where we also accumulate both reload classes.

When comparing the different settings, we draw  similar conclusions as before. Note that SCIP-SDP performs very poorly on the reload instances. The difference between KT and CG1 is not as significant as before, although CG1 is still favourable above KT on almost all instance types. The settings that involve the variables $\bold{X^{(2)}}$ in the root node, i.e., CG2 and SEC-CG, are outperformed by SEC-simple and SEC. Apparently, the increase in the number of variables does not contribute much to the pruning of the branching tree. In fact, the results in Appendix~\ref{App:computations} even suggest that the number of branching nodes sometimes becomes larger. The large spread in computation times for these settings, see  Figure~\ref{fig:reload_results}, also suggests that $(ISDP_2)$ leads to a search process that is less robust and that this effect becomes more visible as the instances become larger. However, the $S$-CG cuts resulting from $(ISDP_2)$ do contribute to the pruning of the tree, as is suggested by the strong performance of SEC.
The settings SEC and SEC-simple overall perform best. None of the two algorithms outperforms the other in terms of computation time, even when the problem size goes up, see the additional numerical results in Table~\ref{Tab:Reload_ext3} of Appendix~\ref{App:computations}.

\begin{table}[H]
\footnotesize
\centering
\begin{tabular}{ccccccccccc}
\hline
\rule{0pt}{3ex}                       & \multicolumn{2}{c}{\textbf{Instance}}  & \textbf{}                  & \multicolumn{7}{c}{\textbf{Average computation times (s)}}                                                                       \\[0.5em]
\textbf{Class}                  & \boldmath{$n$} & \boldmath{$p$}            & \textbf{OPT}               & \textbf{SCIP-SDP} & \textbf{KT} & \textbf{CG1} & \textbf{CG2} & \textbf{SEC-simple} & \textbf{SEC} & \textbf{SEC-CG} \\ \hline
\multicolumn{1}{c|}{\textbf{1}} & 10\rule{0pt}{3ex}           & \multicolumn{1}{c|}{0.5}  & \multicolumn{1}{c|}{6.233} & 0.161             & 0.035       & 0.028        & 0.035        & 0.024               & \textbf{0.019}       & 0.031           \\
\multicolumn{1}{c|}{\textbf{}}  & 10           & \multicolumn{1}{c|}{1} & \multicolumn{1}{c|}{3.3}   & 1.627             & 0.133       & 0.127        & 0.165        & 0.128               & \textbf{0.113}        & 0.171           \\
\multicolumn{1}{c|}{\textbf{}}  & 15           & \multicolumn{1}{c|}{0.5}  & \multicolumn{1}{c|}{6.367} & 2.256             & 0.158       & 0.160        & 0.251        & 0.142               & \textbf{0.139}        & 0.223           \\
\multicolumn{1}{c|}{\textbf{}}  & 15           & \multicolumn{1}{c|}{1} & \multicolumn{1}{c|}{2.8}   & 244.0             & 1.426       & 1.124        & 4.503        & 1.095               & \textbf{1.040}        & 2.825           \\
\multicolumn{1}{c|}{\textbf{}}  & 20           & \multicolumn{1}{c|}{0.5}  & \multicolumn{1}{c|}{6.2}   & 82.08             & 0.610       & 0.625        & 1.510        & 0.483               & \textbf{0.465}        & 1.237           \\
\multicolumn{1}{c|}{\textbf{}}  & 20           & \multicolumn{1}{c|}{1} & \multicolumn{1}{c|}{2.314} & 3908              & 183.8       & 91.56        & 1910         & 162.5               & \textbf{43.21}        & 3278            \\
\multicolumn{1}{c|}{\textbf{}}  & 25           & \multicolumn{1}{c|}{0.5} & \multicolumn{1}{c|}{6.6} &      -     & 76.20  &     35.70     &     1141         &   17.10      &     \textbf{16.30} & 249.1       \\
\multicolumn{1}{c|}{\textbf{2}} & 10           & \multicolumn{1}{c|}{0.5}  & \multicolumn{1}{c|}{22.74} & 0.185             & 0.036       & 0.039        & 0.049        & \textbf{0.029}               & \textbf{0.029}        & 0.044           \\
\multicolumn{1}{c|}{\textbf{}}  & 10           & \multicolumn{1}{c|}{1} & \multicolumn{1}{c|}{8.2}   & 0.962             & 0.164       & 0.148        & 0.152        & \textbf{0.116}               & 0.139        & 0.157           \\
\multicolumn{1}{c|}{\textbf{}}  & 15           & \multicolumn{1}{c|}{0.5}  & \multicolumn{1}{c|}{22.73} & 2.989             & 0.193       & 0.174        & 0.255        & 0.173               & \textbf{0.172}        & 0.261           \\
\multicolumn{1}{c|}{\textbf{}}  & 15           & \multicolumn{1}{c|}{1} & \multicolumn{1}{c|}{6.767} & 277.5             & 1.363       & 1.768        & 4.643        & 1.293               & \textbf{1.190}        & 3.290           \\
\multicolumn{1}{c|}{\textbf{}}  & 20           & \multicolumn{1}{c|}{0.5}  & \multicolumn{1}{c|}{18.1}  & 58.68             & 0.575       & 0.585        & 1.246        & \textbf{0.552}               & 0.576        & 1.352           \\
\multicolumn{1}{c|}{\textbf{}}  & 20           & \multicolumn{1}{c|}{1} & \multicolumn{1}{c|}{4.745} & 2689              & 43.99       & 20.88        & 1187         & \textbf{11.89}               & 16.88        & 850.2           \\ 
\multicolumn{1}{c|}{\textbf{}}  &  25           & \multicolumn{1}{c|}{0.5} & \multicolumn{1}{c|}{16.37} &   -           &  1298     &  315.1      &    5159      &   94.81             &    \textbf{75.81}     & 1701  \\ \hline
\end{tabular}
\caption{Overview of average computation times for the reload instances. Each row provides averages of 30 instances, namely 10 random instances for each value of $c = 5, 10, 20$. \label{Tab:reload_results}}
\end{table}

\begin{figure}[H]
\centering
\includegraphics[scale=0.31]{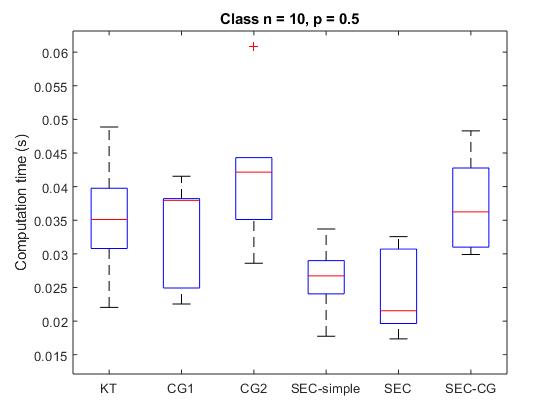}
\includegraphics[scale=0.31]{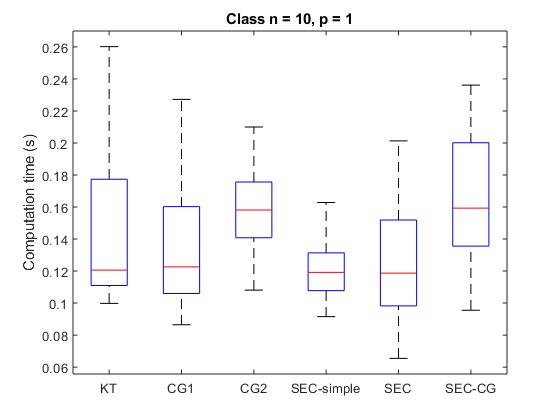}
\includegraphics[scale=0.31]{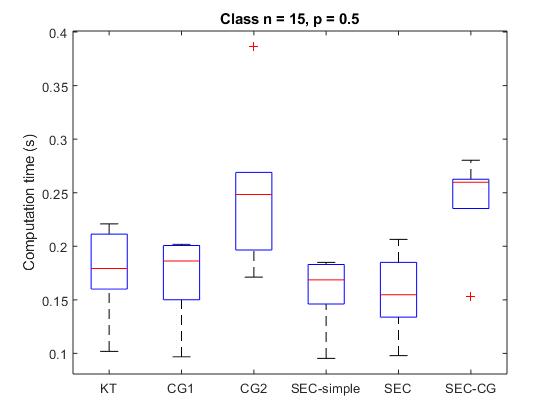}
\includegraphics[scale=0.31]{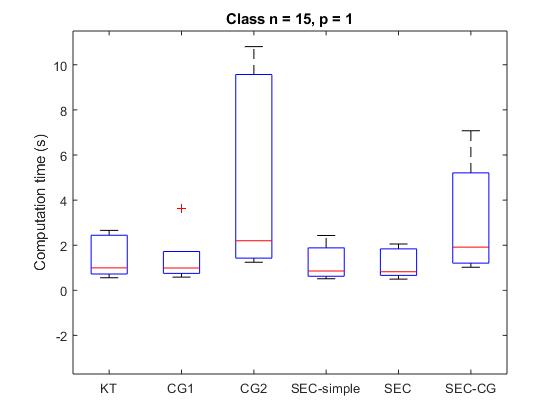} 
\end{figure}
\begin{figure}[H]
\centering
\includegraphics[scale=0.31]{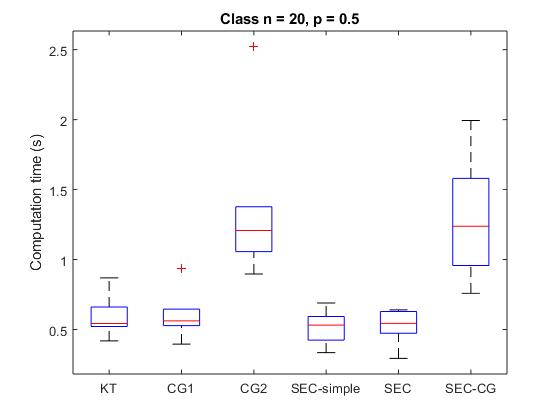}
\includegraphics[scale=0.31]{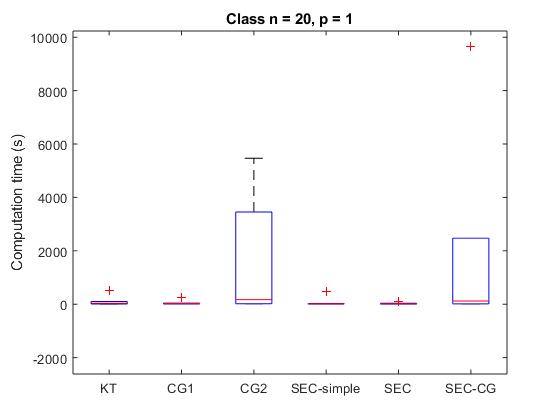}
\includegraphics[scale=0.31]{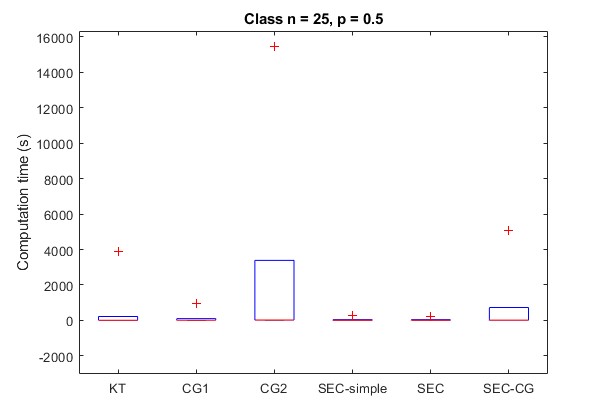}
\caption{Boxplots showing the computation times for the reload instances for different values of $n$ and $p$, accumulated over the reload class and the number of colors $c$. We omit the results of SCIP-SDP, since these computation times are several magnitudes larger. \label{fig:reload_results}}
\end{figure}

Finally, we consider the turn cost instances. From the class of bioinformatics and reload instances it is clear that the settings SEC-simple, SEC and SEC-CG generally perform best. Hence, we restrict the numerical results on the turn cost instances to these three settings. Table~\ref{Tab:tsplib_results} and \ref{Tab:grid_results} show the computation times and number of branching nodes for the TSPLIB and grid instances, respectively.

The TSPLIB graphs are complete graphs, and hence we can only solve up to $n = 70$ for this instance type. We are able to solve all TSPLIB instances in a time span 900 seconds. Since the grid instances are more sparse, we can solve much larger instance sizes to optimality. For this type, instances up to 2646 nodes (!)~can be solved to optimality within 15 seconds. These are currently the largest solved QTSP instances in the literature.

When comparing the three settings, we see that SEC-simple and SEC perform slightly better than SEC-CG on the turn cost instances. Since the different separation routines lead to different relaxations, the branching strategy between the methods can differ. Not surprisingly, the favourable setting is often the one with the smallest number of B\&C nodes, regardless of the time per branching node. Taking both the TSPLIB and grid instances into account, this happens slightly more often for the setting SEC-simple.

\begin{table}[h]
\centering
\footnotesize
\begin{tabular}{cccc|cc|cc|cc}
\hline
               \rule{0pt}{3ex}                           &              &              &                & \multicolumn{2}{c|}{\textbf{SEC-simple}}                                                                                                     & \multicolumn{2}{c|}{\textbf{SEC}}                                                                                                            & \multicolumn{2}{c}{\textbf{SEC-CG}}                                                                                                          \\[0.5em]
\textbf{Instance}                       & \boldmath{$n$} & \boldmath{$m$} & \textbf{OPT} & \textbf{\begin{tabular}[c]{@{}c@{}}Comput.\\ time (s)\end{tabular}} & \textbf{\begin{tabular}[c]{@{}c@{}}Number\\ of nodes\end{tabular}} & \textbf{\begin{tabular}[c]{@{}c@{}}Comput.\\ time (s)\end{tabular}} & \textbf{\begin{tabular}[c]{@{}c@{}}Number\\ of nodes\end{tabular}} & \textbf{\begin{tabular}[c]{@{}c@{}}Comput.\\ time (s)\end{tabular}} & \textbf{\begin{tabular}[c]{@{}c@{}}Number\\ of nodes\end{tabular}} \\ \hline
\multicolumn{1}{c|}{\textbf{lau15}}     & 15\rule{0pt}{3ex}           & 105          & 47             & 0.278                                                                   & 1                                                                  & 0.458                                                                   & 1                                                                  & \textbf{0.115}                                                          & 1                                                                  \\
\multicolumn{1}{c|}{\textbf{wg22}}      & 22           & 231          & 63             & 0.643                                                                   & 1                                                                  & \textbf{0.436}                                                          & 1                                                                  & 0.490                                                                   & 1                                                                  \\
\multicolumn{1}{c|}{\textbf{bays29}}    & 29           & 406          & 78             & 1.519                                                                   & 96                                                                 & \textbf{0.905}                                                          & 93                                                                 & 0.949                                                                   & 78                                                                 \\
\multicolumn{1}{c|}{\textbf{dantzig42}} & 42           & 861          & 96             & \textbf{11.25}                                                          & 994                                                                & 12.04                                                                   & 1059                                                               & 21.20                                                                   & 1458                                                               \\
\multicolumn{1}{c|}{\textbf{att48}}     & 48           & 1128         & 105            & 53.07                                                                   & 4104                                                               & \textbf{47.64}                                                          & 3627                                                               & 55.52                                                                   & 3375                                                               \\
\multicolumn{1}{c|}{\textbf{berlin52}}  & 52           & 1326         & 118            & \textbf{702.9}                                                          & 36523                                                              & 1115                                                                    & 49265                                                              & 1070                                                                    & 41041                                                              \\
\multicolumn{1}{c|}{\textbf{kn57}}      & 57           & 1596         & 120            & 153.1                                                                   & 2425                                                               & \textbf{110.8}                                                          & 1539                                                               & 138.6                                                                   & 1804                                                               \\
\multicolumn{1}{c|}{\textbf{wg59}}      & 59           & 1711         & 121            & \textbf{391.2}                                                          & 10503                                                              & 898.7                                                                   & 13627                                                              & 650.1                                                                   & 10269                                                              \\
\multicolumn{1}{c|}{\textbf{st70}}      & 70           & 2415         & 137            & 861.8                                                                 & 8596                                                               & \textbf{838.1}                                                          & 4649                                                               & 1862                                                                    & 12222                                                              \\ \hline
\end{tabular}
\caption{Computation times and number of branching nodes for the TSPLIB instances. \label{Tab:tsplib_results}}
\end{table}

\begin{table}[h]
\footnotesize
\centering
\begin{tabular}{cccc|cc|cc|cc}
\hline
      \rule{0pt}{3ex}                               &              &              &              & \multicolumn{2}{c|}{\textbf{SEC-simple}}                                                                                                     & \multicolumn{2}{c|}{\textbf{SEC}}                                                                                                            & \multicolumn{2}{c}{\textbf{SEC-CG}}                                                                                                          \\[0.5em]
\textbf{Instance}                   & \boldmath{$n$} & \boldmath{$m$} & \textbf{OPT} & \textbf{\begin{tabular}[c]{@{}c@{}}Comput.\\ time (s)\end{tabular}} & \textbf{\begin{tabular}[c]{@{}c@{}}Number\\ of nodes\end{tabular}} & \textbf{\begin{tabular}[c]{@{}c@{}}Comput.\\ time (s)\end{tabular}} & \textbf{\begin{tabular}[c]{@{}c@{}}Number\\ of nodes\end{tabular}} & \textbf{\begin{tabular}[c]{@{}c@{}}Comput.\\ time (s)\end{tabular}} & \textbf{\begin{tabular}[c]{@{}c@{}}Number\\ of nodes\end{tabular}} \\ \hline
\multicolumn{1}{c|}{\textbf{grid1}} & 430 \rule{0pt}{3ex}          & 795          & 620          & 1.431                                                                   & 6                                                                  & 1.538                                                                   & 1                                                                  & \textbf{1.020}                                                          & 1                                                                  \\
\multicolumn{1}{c|}{\textbf{grid2}} & 734          & 1393         & 460          & \textbf{14.86}                                                          & 2781                                                               & 20.29                                                                   & 7942                                                               & 19.33                                                                  & 2562                                                               \\
\multicolumn{1}{c|}{\textbf{grid3}} & 880          & 1672         & 590          & \textbf{3.303}                                                          & 30                                                                 & 5.019                                                                   & 78                                                                 & 5.945                                                                   & 207                                                                \\
\multicolumn{1}{c|}{\textbf{grid4}} & 960          & 1802         & 840          & 4.954                                                                   & 3                                                                  & \textbf{3.731}                                                          & 1                                                                  & 7.507                                                                   & 1                                                                  \\
\multicolumn{1}{c|}{\textbf{grid5}} & 1038         & 1965         & 440          & 8.452                                                                   & 24                                                                 & \textbf{4.514}                                                          & 16                                                                 & 8.192                                                                   & 10                                                                 \\
\multicolumn{1}{c|}{\textbf{grid6}} & 1214         & 2335         & 480          & 19.67                                                                   & 57                                                                 & \textbf{15.61}                                                          & 25                                                                 & 23.27                                                                   & 55                                                                 \\
\multicolumn{1}{c|}{\textbf{grid7}} & 1302         & 2493         & 730          & \textbf{9.121}                                                          & 330                                                                & 17.83                                                                   & 177                                                                & 14.65                                                                   & 181                                                                \\
\multicolumn{1}{c|}{\textbf{grid8}} & 1788         & 3469         & 540          & 4.800                                                                   & 1                                                                  & 4.917                                                                   & 1                                                                  & \textbf{4.619}                                                          & 1                                                                  \\
\multicolumn{1}{c|}{\textbf{grid9}} & 2646         & 5172         & 760          & 13.79                                                                   & 1                                                                  & 13.80                                                                   & 1                                                                  & \textbf{13.39}                                                          & 1                                                                  \\ \hline
\end{tabular}
\caption{Computation times and number of branching nodes for the grid instances. \label{Tab:grid_results}}
\end{table}

\section{Conclusions}
In this work we study the Chv\'atal-Gomory cuts for spectrahedra and their strength in solving integer semidefinite programs resulting from combinatorial optimization problems.
Accordingly, this paper increases the theoretical understanding of integer semidefinite programming, which in turn contributes to new solution techniques for this type of problems.

In Section~\ref{Section:CGprocedure} we study the elementary closure of spectrahedra and the hierarchy obtained by iterating this procedure.
Using an alternative formulation of the elementary closure, see~\eqref{eq:coniccuts}, we provide simple proofs of several properties, including
a homogeneity property for bounded spectrahedra, see Theorem~\ref{Thm:closureHyperplane}. 
Although some of the here presented results are already known in the literature, the proofs we present are considerably simpler and are mainly based on concepts from mathematical optimization and number
theory.  We also present the polyhedral description of the elementary closure of spectrahedra whose defining linear matrix inequality is totally dual integral, see Theorem~\ref{Thm:TDI}.
To the best of our knowledge, this is the first such description for the elementary closure of a non-polyhedral set.  A full characterization of bounded LMIs that are TDI on  $\mathbb{Z}^m$ is given in Theorem~\ref{Thm:Hilbert}. 
Sufficient conditions for TDI-ness on an appropriate set $Z \subseteq \mathbb{Z}^m$ are given in Theorem~\ref{Thm:polyhedralonZ} and \ref{Thm:TDI_FDG}.

A generic B\&C algorithm for ISDPs based on strengthened CG cuts is presented in Section~\ref{Section:B&C}, see Algorithm~\ref{AlgB&C}. Our algorithm is a refinement of the algorithm from \cite{KobayashiTakano},
where the authors use eigenvector based inequalities to separate infeasible integer points. Moreover, our work can be seen as an extension of \cite{CezikIyengar}, in which the authors introduce CG cuts for conic programs, but leave the efficient separation of CG cuts as an open problem. Our numerical results indicate the effectiveness of the use of deeper CG cuts.
We also provide a separation routine for binary SDPs originating from combinatorial optimization problems, see Section~\ref{SeparationBinarySDP}.

In Section 4 we extensively study the application of our approach to the quadratic traveling salesman problem. Based on a generalization of the notion of algebraic connectivity to directed graphs,
we present two exact ISDP formulations of the \textsc{QTSP}, see \eqref{eq:ISDP1} and \eqref{eq:ISDP2}.  We show that the simplest CG separation routine boils down to finding integer eigenvectors of the adjacency matrix of a node-disjoint cycle cover, see Proposition~\ref{Prop:GomoryCut}.
However, more intricate dual multipliers lead to some well-known families of cuts, e.g., the ordinary and strengthened versions of the subtour elimination constraints, see Table~\ref{Table:SEC}.
We test several variants of our B\&C procedure that involve different separation routines.

Numerical results on the \textsc{QTSP} show that our B\&C algorithm significantly outperforms the two alternative ISDP solvers of \cite{GallyEtAl} and \cite{KobayashiTakano}. For the real instances from bioinformatics \cite{AandFFischer, AandFFischer_RNA}, these solvers are able to solve instances up to only $n = 15$ and $n = 25$, respectively, whereas our method can solve all instances up to $n = 40$ in a short timespan.
As one would expect, the extension to CG inequalities leads to deeper cuts, which successfully reduces the size of the branching tree compared to \cite{KobayashiTakano}. From all considered separation routines, it turns out that the setting SEC, see page~\pageref{page:SEC}, is overall most effective. This setting was able to solve {almost all of the 552 tested \textsc{QTSP} instances} to optimality within 5 minutes, where the largest instance contains $m = 5172$ arcs. This is currently the largest solved \textsc{QTSP} instance in the literature. 

Our work inspires several future research directions. It would be interesting to study the performance of our B\&C algorithm when applied to other optimization problems that can be formulated as ISDPs. 
We expect the exploitation of CG cuts in the branching scheme to be effective for such ISDPs. {Moreover, as for the QTSP many known classes of cuts turned out to be (strengthened) CG cuts with respect to the ISDP formulation, it would be interesting to know whether this also holds for other problems.} \\
\hfill \break
\noindent \textbf{Acknowledgements.} We thank Anja Fischer and Borzou Rostami for sharing with us the instances from bioinformatics and reload instances, respectively. {Moreover, we thank three anonymous reviewers for providing us with insightful feedback on an earlier version of this manuscript.}

\bibliography{QCCP_QTSP}

\begin{appendix}

\section{Derivation of subtour elimination constraints as CG cuts} \label{App:SEC}
In this appendix we elaborate on the construction of the five types of subtour elimination constraints given in Table~\ref{Table:SEC} as ($S$-)CG cuts.

\subsection{Ordinary subtour elimination constraint}
Let $S \subseteq N$ with $|S| < n$. The well-known subtour elimination constraint corresponding to $S$ can be obtained as a CG cut, see also \cite{CezikIyengar}. Let $\bold{\mathbbm{1}_{S}}$ be the indicator vector of the set of nodes $S$. Then
\begin{align*}
\left\langle \bold{\mathbbm{1}_{S}\mathbbm{1}_{S}}^\top , \, \beta \bold{I_n} + \alpha \bold{J_n} - \frac{1}{2}\left( \bold{X} + \bold{X}^\top \right) \right\rangle \geq 0
\end{align*}
is a valid cut. Applying the CG procedure to this cut, yields
\begin{align*} 
\left\langle \bold{\mathbbm{1}_{S}\mathbbm{1}_{S}}^\top, \frac{1}{2}(\bold{X} + \bold{X}^\top)\right \rangle
& \leq  \left \lfloor \left \langle \bold{\mathbbm{1}_{S}\mathbbm{1}_{S}}^\top, \beta \bold{I_n} + \alpha \bold{J_n} \right\rangle \right \rfloor \quad
\Longleftrightarrow \quad \sum_{i \in S, j \in S} x_{ij}  \leq \left \lfloor |S| \left( \beta + \alpha |S| \right) \right \rfloor.
\end{align*}
If $\beta = k_n$ and $\alpha = h_n /n$, then for all $S$ with $|S| < n$ we have $
\beta + \alpha |S| < 1$. Hence, the CG cut above implies
\begin{align} \label{cut:SEC1}
\sum_{i \in S, j \in S} x_{ij} \leq |S| - 1.
\end{align}
The cut \eqref{cut:SEC1} is the common subtour elimination constraint introduced by Dantzig et al.\ \cite{DantzigEtAl}. 

\subsection{Cut-set subtour elimination constraints}
The cut-set subtour elimination constraints are known to be equivalent to the ordinary subtour elimination constraints of \cite{DantzigEtAl}. It is therefore no surprise that these cuts can be obtained similarly as the ordinary subtour elimination constraints.

Let $\bold{U} = \bold{\mathbbm{1}_S\mathbbm{1}_S}^\top$ be the dual multiplier of the linear matrix inequality $\beta \bold{I_n} + \alpha \bold{J_n} - \frac{1}{2}\left( \bold{X} + \bold{X}^\top \right)$ and let $\bold{\mathbbm{1}_S}$ be the dual multiplier of the constraints $-\bold{X}\mathbf{1} = -\mathbf{1}$. The sum of these constraints yields
\begin{align*}
\left\langle \bold{\mathbbm{1}_{S}\mathbbm{1}_{S}}^\top, \frac{1}{2}(\bold{X} + \bold{X}^\top)\right \rangle - \bold{\mathbbm{1}_S}^\top \bold{X} \mathbf{1}
& \leq  \left \lfloor \left \langle \bold{\mathbbm{1}_{S}\mathbbm{1}_{S}}^\top, \beta \bold{I_n} + \alpha \bold{J_n} \right\rangle - \bold{\mathbbm{1}_S}^\top \mathbf{1} \right \rfloor \\
\Longleftrightarrow \quad -\sum_{i \in S, j \notin S} x_{ij} & \leq \lfloor |S| \left( \beta + \alpha |S|\right) \rfloor - |S|.
\end{align*}
If $\beta = k_n$ and $\alpha = h_n / n$, then the right-hand side becomes $|S| - 1 - |S| = -1$, which yields the desired cut.

\subsection{Merged subtour elimination constraint}
Let $(S_1, \ldots, S_k)$ be a partition of the node set of $G$, i.e., $\bigcup_{l = 1}^k S_l = N$ and $S_l \cap S_p = \emptyset$ for all $l \neq p$. We can obtain a merged subtour elimination constraint via the CG procedure in the following way.

Let $\bold{U} = 2\sum_{l = 1}^k \bold{\mathbbm{1}_{S_l}\mathbbm{1}_{S_l}}^\top$ be the dual multiplier for $\beta \bold{I_n} + \alpha \bold{J_n} - \frac{1}{2}\left( \bold{X} + \bold{X}^\top \right)$. Since each dual multiplier $\bold{\mathbbm{1}_{S_l}\mathbbm{1}_{S_l}}^\top$ leads to a CG cut of Type I in Table~\ref{Table:SEC}, its weighted sum also belongs to the elementary closure and looks as follows:
$$
2\sum_{l = 1}^k\sum_{\substack{i \in S_l \\j \in S_l}} x_{ij} \leq 2 \sum_{l = 1}^k \left(|S_l| - 1\right) = 2(n - k).
$$
Now we add to this cut the equality $-\bold{X}\mathbf{1} = - \mathbf{1}$ with dual multiplier $\mathbf{1}$, which yields the desired merged cut
\begin{align*}
2\sum_{l = 1}^k\sum_{\substack{i \in S_l \\ j \in S_l}} x_{ij} - \mathbf{1}^\top \bold{X} \mathbf{1} & \leq 2(n - k) - \mathbf{1}^\top \mathbf{1}
\quad \Longleftrightarrow \quad \sum\limits_{l = 1}^k \sum\limits_{\substack{i \in S_l \\ j \in S_l}} x_{ij} - \sum\limits_{l \neq p} \sum\limits_{\substack{i \in S_l \\ j \in S_p}}x_{ij}  \leq n - 2k.
\end{align*}

\subsection{Strengthened subtour elimination constraints of size two} \label{App:strongSEC2}
Let $i \neq j$ and define $\bold{U} = \bold{\mathbbm{1}_{\{i,j\}} \mathbbm{1}_{\{i,j\}}}^\top$. Taking $\bold{U}$ as the dual multiplier with respect to $\beta^{(2)} \bold{I_n} + \alpha^{(2)} \bold{J_n} - \frac{1}{2}\left( (\bold{X} + \bold{X^{(2)}}) +(\bold{X} + \bold{X^{(2)}} )^\top  \right) \succeq \mathbf{0}$, provides the following valid cut:
\begin{align*}
\left\langle \bold{\mathbbm{1}_{\{i,j\}}\mathbbm{1}_{\{i,j\}}}^\top, \beta^{(2)} \bold{I_n} + \alpha^{(2)} \bold{J_n} - \frac{1}{2}\left( (\bold{X} +  \bold{X^{(2)}} ) +(\bold{X} + \bold{X^{(2)}} )^\top  \right) \right\rangle \geq 0.
\end{align*}
Moreover, adding the coupling constraints $\sum_{k \in N: (i,k,j) \in \mathcal{A}}y_{ikj} - x^{(2)}_{ij} = 0$ and $\sum_{k \in N: (j,k,i) \in \mathcal{A}}y_{jki} - x^{(2)}_{ji} = 0$, each with dual multiplier $1$, and the constraints $- x_{ii} = 0$, $-x_{jj} = 0$, $-x^{(2)}_{ii} = 0$ and $-x^{(2)}_{jj} = 0$, also each with dual mulitplier $1$, gives
\begin{align*}
x_{ij} + x_{ji} + \sum_{\substack{k \in N: \\ (i,k,j) \in \mathcal{A}}}y_{ikj} + \sum_{\substack{k \in N: \\ (j,k,i) \in \mathcal{A}}}y_{jki} \leq 2 \beta^{(2)} + 4 \alpha^{(2)}.
\end{align*}
We now take $\beta^{(2)} = k_n^{(2)}$ and $\alpha^{(2)} = h_n^{(2)}/n$. Applying the standard CG procedure to this inequality results in the cut
\begin{align} \label{eq:weakCGcut}
x_{ij} + x_{ji} + \sum_{\substack{k \in N: \\ (i,k,j) \in \mathcal{A}}}y_{ikj} + \sum_{\substack{k \in N: \\ (j,k,i) \in \mathcal{A}}}y_{jki} \leq \left\lfloor 2 k_n^{(2)} + 4 \frac{h_n^{(2)}}{n} \right\rfloor.
\end{align}
The right-hand side of this cut equals one if $5 \leq n \leq 7$, two if $8 \leq n \leq 12$ and three if $n \geq 13$.

For $n \geq 5$, we can strengthen this cut by applying the $S$-CG procedure as explained in Section~\ref{Section:SCGcuts}. Since the cut \eqref{eq:weakCGcut} only involves variables $\bold{y}$ and $\bold{X}$, we can restrict the set $S$ to the space corresponding to these variables. Let $S = \mathcal{F}_1 \cap \left(\{0,1\}^\mathcal{A} \times \mathcal{T}_n(G)\right)$ and let $c_1$ be the coefficient vector of the left hand side in \eqref{eq:weakCGcut}. Then the strengthened rounding looks as follows:
\begin{align*}
\left\lfloor 2 k_n^{(2)} + 4 \frac{h_n^{(2)}}{n} \right\rfloor_{S, c_1} := \max \left\{ x_{ij} + x_{ji} + \sum_{\substack{k \in N: \\ (i,k,j) \in \mathcal{A}}}y_{ikj} + \sum_{\substack{k \in N: \\ (j,k,i) \in \mathcal{A}}}y_{jki} \, : \, \, \eqref{eq:weakCGcut}, \, (\bold{y},\bold{X}) \in S  \right\}.
\end{align*}
One can verify that the value of this maximization is equal to $1$ for $n \geq 5$. Namely, if it would be greater than 1, this implies a subtour of size two (if $x_{ij}= x_{ji} = 1$), size three (e.g., if $x_{ij} = 1$ and $y_{jki} = 1$ for some $k \in N \setminus \{i,j\}$) or size four (e.g., if $y_{ikj} = 1$ and $y_{jli} = 1$ for some distinct $k,l \in N \setminus \{i,j\}$), which contradicts the fact that $\bold{X} \in \mathcal{T}_n(G)$. We conclude that $\left\lfloor 2 k_n^{(2)} + 4 \frac{h_n^{(2)}}{n} \right\rfloor_{S, c_1} = 1$. Thus, we obtain the strengthened CG cut
\begin{align*}
x_{ij} + x_{ji} + \sum_{\substack{k \in N: \\ (i,k,j) \in \mathcal{A}}}y_{ikj} + \sum_{\substack{k \in N: \\ (j,k,i) \in \mathcal{A}}}y_{jki} \leq 1.
\end{align*}

\subsection{Strenghtened subtour elimination constraints}
Let $S \subset N$ with $2 \leq |S| < \frac{1}{2} n$ and define $\bold{U} = \bold{\mathbbm{1}_S \mathbbm{1}_S}^\top$. Taking $\bold{U}$ as the dual multiplier with respect to $\beta^{(2)} \bold{I_n} + \alpha^{(2)} \bold{J_n} - \frac{1}{2}\left( (\bold{X} + \bold{X^{(2)}}) +(\bold{X} + \bold{X^{(2)}})^\top  \right) \succeq \mathbf{0}$ provides the inequality
\begin{align*}
\left\langle \bold{\mathbbm{1}_S \mathbbm{1}_S}^\top , \beta^{(2)} \bold{I_n} + \alpha^{(2)} \bold{J_n} - \frac{1}{2}\left( (\bold{X} + \bold{X^{(2)}}) +(\bold{X} + \bold{X^{(2)}} )^\top  \right) \right\rangle \geq 0.
\end{align*}
For all $i,j \in S$ we now add the coupling constraints $\sum_{k \in N: (i,k,j) \in \mathcal{A}} y_{ikj} - {x}^{(2)}_{ij} = 0$ with dual multiplier 1. Moreover, for all $(i,k,j) \in \mathcal{A}$ with $i, k, j \in S$ we add the constraint $-y_{ikj} \leq 0$ with multiplier 1. This yields the following valid cut
\begin{align*}
\sum_{\substack{i \in S \\ j \in S}} x_{ij} + \sum_{\substack{i \in S \\ j \in S}} \sum_{\substack{k \in N \setminus S : \\ (i,k,j) \in \mathcal{A}}} y_{ikj} \leq |S|\beta^{(2)} + |S|^2 \alpha^{(2)}.
\end{align*}
Again, we take $\beta^{(2)} = k_n^{(2)}$ and $\alpha^{(2)} = h_n^{(2)} /n$. The standard CG rounding step yields
\begin{align} \label{eq:weakCGcut2}
\sum_{\substack{i \in S \\ j \in S}} x_{ij} + \sum_{\substack{i \in S \\ j \in S}} \sum_{\substack{k \in N \setminus S : \\ (i,k,j) \in \mathcal{A}}} y_{ikj} \leq \left\lfloor |S|\left( k_n^{(2)} + |S| \frac{h_n^{(2)}}{n}\right) \right\rfloor.
\end{align}
Since $|S| < \frac{1}{2}n$, we know
\begin{align*}
\left\lfloor |S|\left( k_n^{(2)} + |S| \frac{h_n^{(2)}}{n}\right) \right\rfloor \leq \left \lfloor |S| \left( k_n^{(2)} + \frac{1}{2} n \frac{2 - k_n^{(2)}}{n}\right) \right\rfloor = \left \lfloor |S| \left( 1 + \frac{1}{2} k_n^{(2)} \right) \right\rfloor \leq 2|S| - 1.
\end{align*}
However, similar to the approach in Appendix~\ref{App:strongSEC2}, we obtain a tighter bound if we apply the strengthened CG procedure. Let $T = \mathcal{F}_1 \cap \left(\{0,1\}^\mathcal{A} \times \mathcal{T}_n(G) \right)$ and let $c_2$ be the coefficient vector of the left hand side of \eqref{eq:weakCGcut2}. Then,
\begin{align*}
\left\lfloor |S|\left( k_n^{(2)} + |S| \frac{h_n^{(2)}}{n}\right) \right\rfloor_{T,c_2} := \max \left\{ \sum_{\substack{i \in S \\ j \in S}} x_{ij} + \sum_{\substack{i \in S \\ j \in S}} \sum_{\substack{k \in N \setminus S : \\ (i,k,j) \in \mathcal{A}}} y_{ikj} \, : \, \, \eqref{eq:weakCGcut2}, \, (\bold{y},\bold{X}) \in T  \right\}.
\end{align*}
It can be verified that this maximum is equal to $|S| - 1$ for all $S$ with $|S| < \frac{1}{2}n$. Namely, if $(\bold{y},\bold{X}) \in T$, we cannot have both $x_{ij} = 1$ and $y_{ikj} = 1$ for some $k \in N$. Hence, $x_{ij} + \sum_{k \in N \setminus S: (i,k,j) \in \mathcal{A}}y_{ikj} \leq 1$ for all $i,j \in S$. If we now sum over all $i,j \in S$, the result must be at most $|S| - 1$, otherwise a subtour would exist. The strengthened CG cut corresponding to \eqref{eq:weakCGcut2} becomes
\begin{align*}
\sum_{\substack{i \in S \\ j \in S}} x_{ij} + \sum_{\substack{i \in S \\ j \in S}} \sum_{\substack{k \in N \setminus S : \\ (i,k,j) \in \mathcal{A}}} y_{ikj} \leq |S| - 1.
\end{align*}

\section{The symmetric quadratic traveling salesman problem} \label{App:SQTSP}
In this appendix we briefly consider the symmetric quadratic traveling salesman problem (\textsc{SQTSP}). Although this problem is very related to the asymmetric version used in the rest of the paper (that we continue to denote by \textsc{QTSP}), the underlying model is different. We show how to construct this model and how all cuts for the \textsc{QTSP} can be extended to the symmetric case.
\medskip

Let $G = (V,E)$ be an undirected graph, where $E$ consists of undirected pairs of nodes $\{i,j\}$ $(= \{j,i\})$, $i, j \in V$. We define  $\mathcal{E} = \{\langle i, j, k\rangle = \langle k, j, i \rangle \, : \, \, i, j, k \in V, |\{i, j, k\}| = 3 \}$ to be the set of two-edges in $G$, where a two-edge is a sequence of three distinct nodes where the reverse sequence is regarded as identical. Given is a quadratic cost matrix $\bold{Q} = (q_{ijk})$, where a cost is zero if $\langle i, j, k\rangle \notin \mathcal{E}$.

The goal of the \textsc{SQTSP} is to find an undirected Hamiltonian cycle in $G$ such that the total quadratic cost is minimized. To model this problem, let $\bold{\bar{x}} \in \{0,1\}^E$ and $\bold{\bar{y}} \in \{0,1\}^\mathcal{E}$ denote indicator vectors that are 1 if and edge, respectively two-edge, is present in the solution and 0 otherwise. We aim to find a tuple $(\bold{\bar{x}}, \bold{\bar{y}})$ with $\bar{y}_{ijk} = \bar{x}_{ij}\bar{x}_{jk}$, representing a Hamiltonian cycle such that $\sum_{\langle i, j k \rangle \in \mathcal{E}}q_{ijk}y_{ijk}$ is minimized.

The symmetric equivalent of the set $\mathcal{F}_1$, see \eqref{setF1}, is now given by:
\begin{align*}\label{setF1_sym}
\mathcal{F}^s_1 := \left \{ (\bold{\bar{y}}, \bold{\bar{x}}) \in \{0,1\}^{\mathcal{E}} \times \{0,1\}^{E} \, : \quad \begin{aligned} \sum_{e \in \delta(i)}\bar{x}_e = 2 \quad \forall i \in V \\
\bar{x}_{ij} = \sum_{\substack{k \in V \\ \langle i, j, k \rangle \in \mathcal{E}}}\bar{y}_{ijk}
= \sum_{\substack{k \in V \\ \langle k, i, j \rangle \in \mathcal{E}}}\bar{y}_{kij} \,\, \forall \{i,j\} \in E
\end{aligned} \right\},
\end{align*}
where $\delta(i) \in V$ denotes the set of edges that are incident to $i$. The formulation used in $\mathcal{F}^s_1$ is introduced by Fischer and Helmberg \cite{FischerHelmberg} where it is shown that the equation $\bar{y}_{ijk} = \bar{x}_{ij}\bar{x}_{jk}$ is indeed established for all $\langle i, j, k \rangle \in \mathcal{E}$. Moreover, similar to the asymmetric case,  we can relax the integrality of $\bold{\bar{y}}$, since it is enforced by the integrality of $\bold{\bar{x}}$ and the coupling constraints, see Remark~\ref{Remark:integrality}. It follows that the tuples in $\mathcal{F}_1^s$ are characteristic vectors of node-disjoint cycle covers in $G$, where the smallest cycles have size 3 due to the definition of $\mathcal{E}$.

The B\&C algorithm presented in Section~\ref{Section:B&C} can now be applied to the \textsc{SQTSP}, starting from optimizing over $\mathcal{F}_1^s$. In order to cut off solutions that do not correspond to a Hamiltonian cycle in $G$, we need separation routines for the symmetric version. Instead of providing symmetric equivalents to all \textsc{QTSP} cutting planes derived in Section~\ref{Section:CutsforQTSP}, we present a transformation that maps any valid cut for the asymmetric version to a cut for the \textsc{SQTSP}. To that end, we introduce a directed graph $H = (V,A)$ that is defined on the same set of nodes as the undirected graph $G$, where $A$ is such that it contains both pairs $(i,j)$ and $(j,i)$ whenever the corresponding edge $\{i,j\}$ is contained in $G$. Moreover, we define the cost of each two-arc $(i,j,k)$ in $H$ to be equal to $q_{\langle i, j, k\rangle }$ for the corresponding two-edge $\langle i, j, k\rangle$ in $G$. Let $\mathcal{I}_\text{S}$ denote the original \textsc{SQTSP} instance and let $\mathcal{I}_\text{A}$ denote the corresponding asymmetric instance. 

The variables in the two programs can now be related as follows: Let $(\bold{y}, \bold{X})$ be variables in $\mathcal{I}_\text{A}$ and define the tuple $(\bold{\bar{y}}, \bold{\bar{x}})$ by
$\bar{x}_{ij} = x_{ij} + x_{ji}$ for all $\{i,j\} \in E$,  and
$\bar{y}_{ijk} = y_{ijk} + y_{kji}$ for all $\langle i, j, k \rangle \in \mathcal{E}.$

From the constraints in $\mathcal{F}_1$ and $\mathcal{F}_1^s$, it follows that any solution $(\bold{y}, \bold{X})$ in $\mathcal{I}_\text{A}$ leads to a solution $(\bold{\bar{y}}, \bold{\bar{x}})$ in $\mathcal{I}_\text{S}$ with the same objective value. Reversely, any solution $(\bold{\bar{y}}, \bold{\bar{x}})$ in $\mathcal{I}_\text{S}$ leads to a solution (or actually two solutions, one for each direction) $(\bold{y}, \bold{X})$ in $\mathcal{I}_\text{A}$ with the same objective value. As a result, any valid cut for $\mathcal{I}_\text{A}$ is also valid for $\mathcal{I}_\text{S}$.

This implies that all cuts defined in Section~\ref{Section:CutsforQTSP} can be converted to cuts for the \textsc{SQTSP}. Namely, given a cut for $\mathcal{I}_\text{A}$, we define the coefficient on $\bar{x}_{ij}$ to be the sum of the coefficients on $x_{ij}$ and $x_{ji}$ for all edges $\{i,j\} \in E$. Similarly, we define the coefficient on $\bar{y}_{ijk}$ to be the sum of the coefficients on $y_{ijk}$ and $y_{kji}$ for all two-edges $\langle i, j, k \rangle \in \mathcal{E}$. If no more violated cuts can be found in $\mathcal{I}_\text{A}$, the corresponding solution in $\mathcal{I}_\text{S}$ is also optimal. This proves the validity of the B\&C algorithm for the symmetric version of the problem.

\section{Extended computational results} \label{App:computations}
In this appendix we present a complete overview of the computational results from which the summarized results in Section~\ref{Section:ComputationResults} follow. We start by considering the instances from bioinformatics, after which we present results for the reload instances. No additional results are presented for the turn instances, since for these instances the complete overview is already given in Section~\ref{Section:ComputationResults}.

In all tables showing computation times, the setting that provides the shortest time is presented in bold for each instance. Moreover, a `-' indicates that a given algorithm could not solve the instance within 8 hours.

The computation times and the number of branching nodes for the class of `bma' instances from bioinformatics are given in Table~\ref{Tab:bma_times} and \ref{Tab:bma_nodes}, respectively. Table~\ref{Tab:map_times} and \ref{Tab:map_nodes} provide computation times and number of branching nodes for the `map' instances, respectively. The computation times and the number of branching nodes for the `ml' instances are presented in Table~\ref{Tab:ml_times} and \ref{Tab:ml_nodes}, respectively.

Finally, we present a more elaborate overview of the reload instances. For each of the two classes and for different values of $n, p$ and $c$, 10 randomly generated instances are considered. In order to save space, we only present the results that are averaged over these 10 similar instances. Table~\ref{Tab:Reload_ext1} and \ref{Tab:Reload_ext2} present the computation times and number of branching nodes, respectively. {Table~\ref{Tab:Reload_ext3} shows the computation times  on 72 additional reload instances on both reload classes with $n \in \{21, \ldots, 26\}$, $p \in \{0.5, 0.8\}$ and $c \in \{5,10,20\}$ in order to further investigate the difference between the settings SEC-simple and SEC. As indicated in Section~\ref{Section:ComputationResults}, the results in Table~\ref{Tab:Reload_ext3} are still not decisive on which of the two settings performs better.}

\begin{landscape}
\begin{table}[H]
\parbox{.45\linewidth}{
\scriptsize
\centering
\setlength{\tabcolsep}{3pt}
\begin{tabular}{@{}cccccccc@{}}
\toprule
\textbf{Instance} & \textbf{SCIP-SDP} & \textbf{KT}                   & \textbf{CG1} & \textbf{CG2} & \textbf{SEC-simple} & \textbf{SEC}   & \textbf{SEC-CG} \\ \midrule
bma2\_3            & \textbf{0.001}    & 0.001                         & 0.006        & 0.153        & \textbf{0.001}      & \textbf{0.001} & \textbf{0.001}  \\
bma2\_4            & 0.075             & \textbf{0.001}                & 0.002        & 0.002        & 0.002               & 0.002          & \textbf{0.001}  \\
bma2\_5            & 0.079             & \textbf{0.005}                & 0.004        & 0.017        & \textbf{0.005}      & \textbf{0.005} & 0.008           \\
bma2\_6            & 0.221             & 0.034                         & 0.034        & 0.047        & \textbf{0.029}      & 0.555          & 0.101           \\
bma2\_7            & 0.436             & 0.048                         & 0.051        & 0.062        & 0.054               & 0.157          & \textbf{0.039}  \\
bma2\_8            & 2.76              & 0.083                         & 0.073        & 0.081        & 0.090               & \textbf{0.049} & 0.102           \\
bma2\_9            & 3.85              & 0.114                         & 0.121        & 0.165        & 0.120               & \textbf{0.028} & 0.043           \\
\bottomrule
\end{tabular}
\caption{Computation times (s) for bioinformatics instances from the `bma' class. \label{Tab:bma_times}}
}
\hfill
\parbox{.45\linewidth}{
\centering
\scriptsize
\setlength{\tabcolsep}{3pt}
\begin{tabular}{@{}cccccccc@{}}
\toprule
\textbf{Instance} & \textbf{SCIP-SDP} & \textbf{KT}                      & \textbf{CG1} & \textbf{CG2} & \textbf{SEC-simple} & \textbf{SEC} & \textbf{SEC-CG} \\ \midrule
bma2\_3           & 1                 & 0                                & 0            & 0            & 0                   & 0            & 0               \\
bma2\_4           & 1                 & 0                                & 0            & 0            & 0                   & 0            & 0               \\
bma2\_5           & 15                & 1                                & 1            & 1            & 1                   & 1            & 1               \\
bma2\_6           & 31                & 21                               & 15           & 8            & 1                   & 1            & 4               \\
bma2\_7           & 91                & 44                               & 44           & 56           & 20                  & 1            & 1               \\
bma2\_8           & 223               & 188                              & 162          & 117          & 106                 & 5            & 10              \\
bma2\_9           & 488               & 439                              & 468          & 391          & 101                 & 1            & 1               \\
\bottomrule
\end{tabular}
\caption{Number of branching nodes for bioinformatics instances from the `bma' class. \label{Tab:bma_nodes}}
}
\end{table}

\begin{table}[H]
\parbox{.45\linewidth}{
\centering
\scriptsize
\setlength{\tabcolsep}{3pt}
\begin{tabular}{@{}cccccccc@{}}
\toprule
\textbf{Instance} & \textbf{SCIP-SDP} & \textbf{KT}               & \textbf{CG1}   & \textbf{CG2}   & \textbf{SEC-simple} & \textbf{SEC}   & \textbf{SEC-CG} \\ \midrule
map2\_3           & \textbf{0}        & 0.004                     & \textbf{0}     & 0.001          & 0.001               & 0.001          & \textbf{0}      \\
map2\_4           & 0.132             & \textbf{0.001}            & \textbf{0.001} & \textbf{0.001} & 0.002               & \textbf{0.001} & \textbf{0.001}  \\
map2\_5           & 0.149             & \textbf{0.005}            & 0.007          & 0.088          & \textbf{0.005}      & 0.006          & 0.022           \\
map2\_6           & 0.291             & 0.033                     & 0.026          & 0.042          & 0.051               & \textbf{0.021} & 0.029           \\
map2\_7           & 1.551             & 0.06                      & 0.057          & 0.147          & 0.058               & \textbf{0.027} & 0.032           \\
map2\_8           & 3.175             & 0.094                     & 0.095          & 0.104          & 0.064               & \textbf{0.043} & 0.056           \\
map2\_9           & 8.373             & 0.128                     & 0.096          & 0.154          & 0.106               & \textbf{0.03}  & 0.044           \\
\bottomrule
\end{tabular}
\caption{Computation times (s) for bioinformatics instances from the `map' class. \label{Tab:map_times}}
} \hfill
\parbox{.45\linewidth}{
\centering
\scriptsize
\setlength{\tabcolsep}{3pt}
\begin{tabular}{@{}cccccccc@{}}
\toprule
\textbf{Instance} & \textbf{SCIP-SDP} & \textbf{KT}               & \textbf{CG1} & \textbf{CG2} & \textbf{SEC-simple} & \textbf{SEC} & \textbf{SEC-CG} \\ \midrule
map2\_3           & 1                 & 0                         & 0            & 0            & 0                   & 0            & 0               \\
map2\_4           & 1                 & 0                         & 0            & 0            & 0                   & 0            & 0               \\
map2\_5           & 11                & 1                         & 1            & 1            & 1                   & 1            & 1               \\
map2\_6           & 40                & 14                        & 17           & 8            & 5                   & 1            & 6               \\
map2\_7           & 109               & 89                        & 89           & 35           & 25                  & 1            & 1               \\
map2\_8           & 232               & 189                       & 160          & 90           & 93                  & 1            & 1               \\
map2\_9           & 571               & 457                       & 405          & 262          & 113                 & 1            & 1               \\
\bottomrule
\end{tabular}
\caption{Number of branching nodes for bioinformatics instances from the `map' class. \label{Tab:map_nodes}}
}
\end{table}

\begin{table}[H]
\parbox{.45\linewidth}{
\centering
\scriptsize
\setlength{\tabcolsep}{3pt}
\begin{tabular}{@{}cccccccc@{}}
\toprule
\textbf{Instance} & \textbf{SCIP-SDP} & \textbf{KT}               & \textbf{CG1}   & \textbf{CG2} & \textbf{SEC-simple} & \textbf{SEC}   & \textbf{SEC-CG} \\ \midrule
ml2\_3            & \textbf{0}        & 0.001                     & \textbf{0}     & 0.001        & \textbf{0}          & \textbf{0}     & 0.001           \\
ml2\_4            & 0.056             & 0.002                     & \textbf{0.001} & 0.002        & \textbf{0.001}      & \textbf{0.001} & \textbf{0.001}  \\
ml2\_5            & 0.089             & 0.004                     & 0.004          & 0.027        & 0.007               & \textbf{0.003} & 0.019           \\
ml2\_6            & 0.182             & 0.032                     & 0.034          & 0.035        & 0.044               & \textbf{0.015} & \textbf{0.015}  \\
ml2\_7            & 0.402             & 0.064                     & 0.057          & 0.063        & 0.056               & \textbf{0.024} & 0.04            \\
ml2\_8            & 1.769             & 0.075                     & 0.075          & 0.100        & 0.077               & \textbf{0.039} & 0.089           \\
ml2\_9            & 4.528             & 0.118                     & 0.116          & 0.172        & 0.112               & \textbf{0.026} & 0.04            \\
\bottomrule
\end{tabular}
\caption{Computation times (s) for bioinformatics instances from the `ml' class.\label{Tab:ml_times}}
} \hfill
\parbox{.45\linewidth}{
\centering
\scriptsize
\setlength{\tabcolsep}{3pt}
\begin{tabular}{@{}ccccclcc@{}}
\toprule
\textbf{Instance} & \textbf{SCIP-SDP}          & \textbf{KT}               & \textbf{CG1} & \textbf{CG2} & \multicolumn{1}{c}{\textbf{SEC-simple}} & \textbf{SEC} & \textbf{SEC-CG} \\ \midrule
ml2\_3            & \multicolumn{1}{l}{1}      & 0                         & 0            & 0            & 0                                       & 0            & 0               \\
ml2\_4            & \multicolumn{1}{l}{1}      & 0                         & 0            & 0            & 0                                       & 0            & 0               \\
ml2\_5            & \multicolumn{1}{l}{11}     & 1                         & 1            & 1            & 1                                       & 1            & 1               \\
ml2\_6            & \multicolumn{1}{l}{31}     & 17                        & 17           & 1            & 16                                      & 1            & 1               \\
ml2\_7            & \multicolumn{1}{l}{79}     & 68                        & 68           & 66           & 28                                      & 1            & 1               \\
ml2\_8            & \multicolumn{1}{l}{232}    & 198                       & 187          & 163          & 85                                      & 1            & 18              \\
ml2\_9            & \multicolumn{1}{l}{505}    & 408                       & 464          & 360          & 106                                     & 1            & 1               \\
\bottomrule
\end{tabular}
\caption{Number of branching nodes for bioinformatics instances from the `ml' class. \label{Tab:ml_nodes}}
}
\end{table}

\end{landscape}

\begin{table}[H]
\centering
\scriptsize
\setlength{\tabcolsep}{2pt}
\begin{tabular}{ccccccccccc}
\hline
  \rule{0pt}{3ex}      & \multicolumn{2}{c}{\textbf{Instance}}                & \multicolumn{8}{c}{\textbf{Average computation times (s)}}                                                                                                                        \\[0.5em]
\textbf{Class}              & \textbf{\boldmath $n$} & \textbf{\boldmath $(p,c)$}                      & \textbf{OPT}                        & \textbf{SCIP-SDP} & \textbf{KT}    & \textbf{CG1}   & \textbf{CG2}   & \textbf{SEC-simple} & \textbf{SEC}   & \textbf{SEC-CG} \\ \hline
\multicolumn{1}{c|}{1} & 10\rule{0pt}{3ex}       & \multicolumn{1}{c|}{(0.5,5)}            & \multicolumn{1}{c|}{5.2}            & 0.131             & 0.049          & 0.025          & 0.029          & 0.018               & 0.020          & 0.030           \\
\multicolumn{1}{c|}{}  & 10           & \multicolumn{1}{c|}{(0.5,10)}           & \multicolumn{1}{c|}{5.8}            & 0.135             & 0.034          & 0.038          & 0.042          & 0.029               & 0.020          & 0.031           \\
\multicolumn{1}{c|}{}  & 10           & \multicolumn{1}{c|}{(0.5,20)}           & \multicolumn{1}{c|}{7.7}            & 0.218             & 0.022          & 0.023          & 0.035          & 0.024               & 0.017          & 0.033           \\
\multicolumn{1}{c|}{}  & 10           & \multicolumn{1}{c|}{(1,5)}           & \multicolumn{1}{c|}{2}              & 2.135             & 0.177          & 0.160          & 0.210          & 0.163               & 0.152          & 0.236           \\
\multicolumn{1}{c|}{}  & 10           & \multicolumn{1}{c|}{(1,10)}          & \multicolumn{1}{c|}{3.4}            & 1.888             & 0.121          & 0.135          & 0.176          & 0.129               & 0.121          & 0.182           \\
\hline
\end{tabular}
\caption{Computation times of the reload instances averaged over 10 generated instances for given values of $n, p$ and $c$. \label{Tab:Reload_ext1}}
%
\end{table}

\begin{table}[H]
\scriptsize
\centering
\setlength{\tabcolsep}{2pt}
\begin{tabular}{cccccccccc}
\hline
\rule{0pt}{3ex}           & \multicolumn{2}{c}{\textbf{Instance}}                & \multicolumn{7}{c}{\textbf{Number of branching nodes}}                                                                         \\[0.5em]
\textbf{Class}         & \boldmath{$n$} & \boldmath{$(p,c)$}                      & \textbf{SCIP-SDP} & \textbf{KT}    & \textbf{CG1}    & \textbf{CG2}   & \textbf{SEC-simple} & \textbf{SEC}   & \textbf{SEC-CG} \\ \hline
\multicolumn{1}{c|}{1} & 10\rule{0pt}{3ex}           & \multicolumn{1}{c|}{(0.5,5)}            & 19.7              & 1.1            & 1.7             & 1.4            & 0.7                 & 0.7            & 0.8             \\
\multicolumn{1}{c|}{}  & 10           & \multicolumn{1}{c|}{(0.5,10)}           & 16.6              & 2.2            & 2.1             & 3.8            & 0.8                 & 0.7            & 0.8             \\
\multicolumn{1}{c|}{}  & 10           & \multicolumn{1}{c|}{(0.5,20)}           & 43.4              & 4.3            & 3.9             & 0.7            & 0.7                 & 1.8            & 2.6             \\
\multicolumn{1}{c|}{}  & 10           & \multicolumn{1}{c|}{(1,5)}           & 132               & 125.9          & 173.1           & 168.6          & 74.6                & 69.4           & 99.4            \\
\multicolumn{1}{c|}{}  & 10           & \multicolumn{1}{c|}{(1,10)}          & 115               & 79             & 82.4            & 72.3           & 43.2                & 32.5           & 43.4            \\
\multicolumn{1}{c|}{}  & 10           & \multicolumn{1}{c|}{(1,20)}          & 46.5              & 15.4           & 15.4            & 11.6           & 9.5                 & 4.6            & 7.1             \\
\hline
\end{tabular}
\caption{Number of branching nodes for the reload instances averaged over 10 generated instances for given values of $n, p$ and $c$.\label{Tab:Reload_ext2}} 
\end{table}

\begin{table}[H]
\scriptsize
\centering

\setlength{\tabcolsep}{3pt}
\begin{tabular}{@{}ccccccc@{}}
\toprule
\textbf{n} & \textbf{p} & \textbf{c} & \textbf{Class} & \textbf{OPT} & \textbf{SEC-simple} & \textbf{SEC}   \\ \midrule
21         & 0.5        & 5          & 1              & 2            & 0.787               & \textbf{0.771} \\
           &            &            & 2              & 4            & 0.738               & \textbf{0.751} \\
           &            & 10         & 1              & 8            & 2.767               & \textbf{2.342} \\
           &            &            & 2              & 13           & \textbf{0.521}      & 0.543          \\
           &            & 20         & 1              & 9            & \textbf{0.501}      & 0.841          \\
           &            &            & 2              & 31           & \textbf{0.395}      & 0.407          \\
           & 0.8        & 5          & 1              & 2            & \textbf{169.1}      & 231.1          \\
\bottomrule
\end{tabular} \hspace{1cm} \begin{tabular}{@{}ccccccc@{}}
\toprule
\textbf{n} & \textbf{p} & \textbf{c} & \textbf{Class} & \textbf{OPT} & \textbf{SEC-simple} & \textbf{SEC}   \\ \midrule

24         & 0.5        & 5          & 1              & 3            & 37.95               & \textbf{29.45} \\
           &            &            & 2              & 10           & \textbf{13.90}      & 17.00          \\
           &            & 10         & 1              & 8            & \textbf{1.947}      & 2.045          \\
           &            &            & 2              & 22           & \textbf{1.295}      & 2.501          \\
           &            & 20         & 1              & 11           & 1.636               & \textbf{1.631} \\
           &            &            & 2              & 28           & \textbf{5.907}      & 9.120          \\
           & 0.8        & 5          & 1              & -            & 
 -              &  - \\
\bottomrule
\end{tabular}
\caption{Computation times of SEC-simple and SEC on 72 additional reload instances  for given values of $n, p$ and $c$. \label{Tab:Reload_ext3}}
\end{table}

\end{appendix}

\end{document}